\documentclass[11pt]{amsart}
\usepackage{amsmath, amsthm, amssymb}
\usepackage[margin=2cm]{geometry}
\usepackage{epsfig}
\usepackage{graphicx}
\usepackage{rawfonts}
\usepackage{enumerate}
\usepackage{amssymb}
\usepackage{mathtools}
\usepackage{float} 
\usepackage{graphicx}
\usepackage{pgf,tikz,pgfplots}
\usepackage{mathrsfs}
\usetikzlibrary{arrows}
\usepackage{amsmath}
\usepackage{amsfonts}
\usepackage{amssymb}
\usepackage{amsthm}
\usepackage{graphicx}
\usepackage{booktabs}
\usepackage{caption}
\usepackage{listings}
\usepackage{setspace}
\usepackage[mathscr]{eucal}
\usepackage{pgfplots}
\usepackage{hyperref}
\usepackage{wrapfig}
\usepackage{floatflt,epsfig}
\usepackage{ dsfont }
\usepackage{amscd}
\usepackage{tikz-cd}
\usepackage{fancyhdr}
\usepackage[all]{xy}
\usepackage{latexsym}
\usepackage{amscd}
\usepackage{pifont}
\usepackage{eufrak}
\usepackage{subfig}
\usepackage{easyReview}

\let\to=\rightarrow

\def\Implies{\ifmmode\Longrightarrow \else
	\unskip${}\Longrightarrow{}$\ignorespaces\fi}
\def\implies{\ifmmode\Rightarrow \else
	\unskip${}\Rightarrow{}$\ignorespaces\fi}
\def\iff{\ifmmode\Longleftrightarrow \else
	\unskip${}\Longleftrightarrow{}$\ignorespaces\fi}

\let\:=\colon

\newcommand{\rank}{\mathop{\rm rank}\nolimits}

\newcommand{\numberset}{\mathbb}

\newcommand{\Z}{\numberset{Z}}

\def\NN{{\mathbb N}}
\def\ZZ{{\mathbb Z}}

\newcommand{\lt}{\mathop{\rm in}\nolimits}

\newcommand{\cA}{\mathcal{A}}
\newcommand{\cC}{\mathcal{C}}

\newcommand{\cP}{\mathcal{P}}
\newcommand{\cH}{\mathcal{H}}
\newcommand{\cI}{\mathcal{I}}

\newcommand{\cB}{\mathcal{B}}
\newcommand{\cM}{\mathcal{M}}
\newcommand{\cF}{\mathcal{F}}
\newcommand{\cQ}{\mathcal{Q}}
\newcommand{\cW}{\mathcal{W}}
\newcommand{\cR}{\mathcal{R}}
\newcommand{\cS}{\mathcal{S}}

\newcommand{\cL}{\mathcal{L}}
\newcommand{\cV}{\mathcal{V}}
\newcommand{\cG}{\mathcal{G}}

\newcommand{\bigslant}[2]{{\raisebox{.2em}{$#1$}\left/\raisebox{-.2em}{$#2$}\right.}}

\theoremstyle{plain}

\theoremstyle{theorem}

\newtheorem{defn}{Definition}[section]
\newtheorem{prop}[defn]{Proposition}
\newtheorem{thm}[defn]{Theorem}
\newtheorem{lemma}[defn]{Lemma}
\newtheorem{conj}[defn]{Conjecture}
\newtheorem{coro}[defn]{Corollary}
\newtheorem{exa}[defn]{Example}
\newtheorem{rmk}[defn]{Remark}

\newtheorem{discussion}[defn]{Discussion}

\theoremstyle{remark}

\begin{document}
	
	\title[NON-SIMPLE POLYOMINOES OF K\H{O}NIG TYPE and their canonical module]{NON-SIMPLE POLYOMINOES OF K\H{O}NIG TYPE and their canonical module}
	
	\author{RODICA DINU}
	\address{
		University of Konstanz, Fachbereich Mathematik und Statistik, Fach D 197 D-78457 Konstanz, Germany, and Institute of Mathematics Simion Stoilow of the Romanian Academy, Calea Grivitei 21, 010702, Bucharest, Romania}
	\email{rodica.dinu@uni-konstanz.de}

	\author{FRANCESCO NAVARRA}
	\address{Sabanci University, Faculty of Engineering and Natural Sciences, Orta Mahalle, Tuzla 34956, Istanbul, Turkey}
	\email{francesco.navarra@sabanciuniv.edu}

	

	\begin{abstract}
		We study the K\H{o}nig type property for non-simple polyominoes. We prove that, for closed path polyominoes, the polyomino ideals are of K\H{o}nig type, extending the results of Herzog and Hibi for simple thin polyominoes. As an application of this result, we give a combinatorial interpretation for the canonical module of the coordinate ring of a sub-class of closed paths, namely circle closed path polyominoes. In this case, we compute also the Cohen-Macaulay type and we show that the coordinate ring is level. 
	\end{abstract}

	\maketitle
	
	\section{Introduction}
	\noindent Polyominoes are plane figures obtained by joining unit squares along their edges. They raise many combinatorial problems, such as tiling a given region or even the entire plane with polyominoes, which are of interest to mathematicians and computer scientists. Even though problems like, for example, the enumeration of pentominoes, have their origins in antiquity, polyominoes were formally defined by Golomb first in 1953 and later, in 1996, in his monograph~\cite{golomb}. The study of polyominoes reveals many connections to different subjects. For instance, in algebraic languages there is a nice relation between polyominoes and Dyck words and Motzkin words \cite{delest}. In statistical physics, polyominoes (and their higher-dimensional analogs known in the literature as lattice animals) appear as models of branched polymers and of percolation clusters \cite{animals}.\\
 \noindent A classic topic in commutative algebra is the study of determinantal ideals. These are the ideals generated by the $t$-minors of a matrix whose entries are indeterminates; see for instance \cite{Bruns_Herzog} and \cite{bernd}. More generally, ideals of $t$-minors of 2-sided ladders have been studied in \cite{conca1}, \cite{conca2}, \cite{conca3}, \cite{gorla}. The concept of polyomino ideal,  alternatively known as the ideal generated by the so-called inner $2$-minors of a $m \times n$ matrix, widely generalize the theory of determinantal ideals. Specifically, if $\cP$ is a polyomino, then the \textit{polyomino ideal} $I_{\cP}$ of $\cP$ is defined as the ideal generated by those sets of 2-minors of the matrix that can be combinatorially associated with the polyomino
 $\cP$. This type of ideal was introduced in 2012 by Qureshi~\cite{Qureshi}. Since then, the study of the main algebraic properties of polyomino ideals and their quotient rings $K[\cP]$, in terms of the shape of $\cP$, has emerged as an exciting area of research. For instance, several mathematicians have studied the primality of $I_\cP$, see \cite{Cisto_Navarra_closed_path}, \cite{Cisto_Navarra_weakly}, \cite{Cisto_Navarra_CM_closed_path}, \cite{def balanced}, \cite{Not simple with localization}, \cite{Trento}, \cite{Trento2}, \cite{Shikama}. Moreover, in \cite{Simple equivalent balanced} and \cite{Simple are prime}, the authors showed that $K[\cP]$ is a normal Cohen-Macaulay domain if $\cP$ is a simple polyomino, i.e. a polyomino without holes. In \cite{Def. Konig type} the authors introduced graded ideals of K\H{o}nig type with respect to a monomial order $<$. These ideals $I$ are characterized by the existence of a sequence, whose length is the same as the height of $I$, of homogeneous polynomials forming part of a minimal system of generators of $I$ and of a monomial order $<$ with respect to whom their initial monomials form a regular sequence. The authors presented interesting consequences that may occur when working with a graded ideal with this property. Moreover, in \cite{Hibi - Herzog Konig type polyomino}, Herzog and Hibi showed that if $\cP$ is a simple thin polyomino, where thin means that the polyomino does not contain the square tetromino, then its polyomino ideal has the K\H{o}nig type property.\\ 
  We are interested in understanding the K\H{o}nig type property for non-simple polyominoes, following the path initiated by Herzog and Hibi. We will focus on a specific class of non-simple thin polyominoes, namely closed path polyominoes. In \cite[Example 5.6]{Hibi - Herzog Konig type polyomino}, the authors provided an example of a closed path whose polyomino ideal is of K\H{o}nig type with respect to a particular monomial order. This motivated us to study the K\H{o}nig type property for the closed paths. However, the monomial order suggested by the authors is different from the ones proposed in this paper. Several interesting results on the class of closed path polyominoes can be found in \cite{Cisto_Navarra_closed_path}, \cite{Cisto_Navarra_Rizwan}, \cite{Cisto_Navarra_CM_closed_path}, \cite{Cisto_Navarra_Hilbert_series} and \cite{Cisto_Navarra_Veer}. Since the first draft of this work was posted on arxiv in October 2022, two interesting results have been obtained on polyomino ideals of closed paths. In \cite{Cisto_Navarra_Veer} it is given a complete description of the minimal primes of the polyomino ideal of a closed path and in \cite{Cisto_Navarra_Rizwan} it is proved that the coordinate ring of a closed path is always Cohen-Macaulay.  \\
 \noindent However, not all polyomino ideals are of K\H{o}nig type and there is no known classification of the polyominoes that have this property. In particular, parallelogram polyominoes give a class of simple polyominoes for which this property does not hold. Indeed, this follows by \cite[Proposition 2.3]{Parallelogram Hilbert series}, where the authors showed that parallelogram polyominoes are simple planar distributive lattices, and by using the classification of distributive lattices of K\H{o}nig type provided in \cite[Theorem 4.1]{Hibi - Herzog Konig type polyomino}. In addition, we found an example of a non-simple thin polyomino that cannot be of K\H{o}nig type (see Remark~\ref{grid_poly}).
 
  The paper is organized as follows. In Section~\ref{Section: Introduction}, we present a detailed introduction to polyominoes and polyomino ideals, and in Lemma~\ref{Lemma: Closed path number of vertices and cells} we prove that, if $\cP$ is a closed path polyomino, then its number of vertices is twice the number of its cells, a fact that will be useful in the next sections. In order to study closed path polyominoes of K\H{o}nig type, a combinatorial formula to compute the height of $I_{\cP}$ is needed. Section~\ref{Krull} is devoted to this scope. In Theorem~\ref{Thm: Dimension closed path}, we give a combinatorial formula for the Krull dimension of $K[\cP]$ and prove it using the simplicial complexes theory. The fact that $\cP$ contains some specific configurations as shown in \cite[Section 6]{Cisto_Navarra_closed_path} plays a crucial role in our proof. Consequently, in Corollary~\ref{Coro: height of P}, we prove that the height of $I_{\cP}$ is the number of cells of the closed path polyomino. We conjecture that this formula holds for any non-simple polyomino. Furthermore, the polyominoes discussed in \cite{Navarra_Dinu_grid} and \cite{Frame} affirmatively support this conjecture. Additionally, in \cite[Theorem 1.1]{Moradi}, an intriguing sufficient condition for this conjecture is provided. Section~\ref{Konig} is devoted to the proof of the K\H{o}nig type property of any closed path polyomino (see Theorem~\ref{konigfinal}). Unfortunately, the monomial order used to prove Theorem \ref{Thm: Dimension closed path} does not guarantee the K\H{o}nig type property for all closed paths as shown in Remark \ref{grid_poly_2}. In Definition~\ref{Procedure: to define Y}, we define a suitable order on the vertices of the closed path polyomino $\cP$ for whom the desired property holds. There are two cases to be examined: either $\cP$ contains a configuration of four cells (treated in Proposition~\ref{Lemma: A closed path with a tetromino is Konig type}) or $\cP$ has an $L$-configuration in every change of direction (treated in Proposition~\ref{Lemma: A closed path with a L-conf is Konig type}). In addition, we present a concrete example to illustrate our procedure. In Section~\ref{canonical}, we study the canonical module of the coordinate ring for a sub-class of closed path polyominoes, called \textit{circle closed path polyominoes} (see Definition~\ref{circle}). Actually, the canonical module of coordinate rings of polyominoes has never been studied; just recently, in \cite{Level}, the authors have studied the levelness property for the special class of simple paths. In our case, we show that the canonical module can be obtained from two ideals: a binomial ideal $J(\cP)$, which is given by the K\H{o}nig type property (see Proposition~\ref{Lemma: A closed path with a L-conf is Konig type}), and a monomial ideal $K(\cP)$, which is intimately related to the combinatorics of the polyomino. The binomial ideal coming from the K\H{o}nig type property will play an important role: $J(\cP)\subset I_{\cP}$ is a complete intersection ideal, radical and it has the same height as the polyomino ideal associated to a closed path, by Lemma~\ref{Lemma: property of J(P) compared to I_P}. These properties allow us to use a result from linkage theory, Proposition~\ref{prop5.5}, which was first observed in \cite{P_S_Canonical_module}. To compute the colon ideal $J(\cP): I_{\cP}$, we use another result, namely Proposition~\ref{prop5.6}, determining all the minimal prime ideals of $J(\cP)$. Finally we give our main result from this section (Theorem~\ref{Thm: description canonical module}), where we determine explicitly the canonical module of $K[\cP]$ for any circle closed path polyomino $\cP$. As a consequence, we compute the Cohen-Macaulay type of $K[\cP]$ in Corollary~\ref{CMtype}, and we show that $K[\cP]$ is a level ring. 
	
	\section{Basics on polyominoes and polyomino ideals}\label{Section: Introduction}
	
	\noindent Let $(i,j),(k,l)\in \Z^2$. We say that $(i,j)\leq(k,l)$ if $i\leq k$ and $j\leq l$. Consider $a=(i,j)$ and $b=(k,l)$ in $\Z^2$ with $a\leq b$. The set $[a,b]=\{(m,n)\in \Z^2: i\leq m\leq k,\ j\leq n\leq l \}$ is called an \textit{interval} of $\Z^2$. 
	Moreover, if $i< k$ and $j<l$, then $[a,b]$ is a \textit{proper} interval. In this case, we say $a$ and $b$ are the \textit{diagonal corners} of $[a,b]$, and $c=(i,l)$ and $d=(k,j)$ are the \textit{anti-diagonal corners} of $[a,b]$. If $j=l$ (or $i=k$), then $a$ and $b$ are in \textit{horizontal} (or \textit{vertical}) \textit{position}. We denote by $]a,b[$ the set $\{(m,n)\in \Z^2: i< m< k,\ j< n< l\}$. A proper interval $C=[a,b]$ with $b=a+(1,1)$ is called a \textit{cell} of $\ZZ^2$; moreover, the elements $a$, $b$, $c$ and $d$ are called respectively the \textit{lower left}, \textit{upper right}, \textit{upper left} and \textit{lower right} \textit{corners} of $C$. The set of vertices of $C$ is $V(C)=\{a,b,c,d\}$ and the set of edges of $C$ is $E(C)=\{\{a,c\},\{c,b\},\{b,d\},\{a,d\}\}$. Let $\cS$ be a non-empty collection of cells in $\Z^2$. Then $V(\cS)=\bigcup_{C\in \cS}V(C)$ and $E(\cS)=\bigcup_{C\in \cS}E(C)$, while the rank of $\cS$ is the number of cells belonging to $\cS$. If $C$ and $D$ are two distinct cells of $\cS$, then a \textit{walk} from $C$ to $D$ in $\cS$ is a sequence $\cC:C=C_1,\dots,C_m=D$ of cells of $\ZZ^2$ such that $C_i \cap C_{i+1}$ is an edge of $C_i$ and $C_{i+1}$ for $i=1,\dots,m-1$. Moreover, if $C_i \neq C_j$ for all $i\neq j$, then $\cC$ is called a \textit{path} from $C$ to $D$. Moreover, if we denote by $(a_i,b_i)$ the lower left corner of $C_i$ for all $i=1,\dots,m$, then $\cC$ has a \textit{change of direction} at $C_k$ for some $2\leq k \leq m-1$ if $a_{k-1} \neq a_{k+1}$ and $b_{k-1} \neq b_{k+1}$. In addition, a path can change direction in one of the following ways:
	\begin{enumerate}
		\item North, if $(a_{i+1}-a_i,b_{i+1}-b_i)=(0,1)$ for some $i=1,\dots,m-1$;
		\item South, if $(a_{i+1}-a_i,b_{i+1}-b_i)=(0,-1)$ for some $i=1,\dots,m-1$;
		\item East, if $(a_{i+1}-a_i,b_{i+1}-b_i)=(1,0)$ for some $i=1,\dots,m-1$;
		\item West, if $(a_{i+1}-a_i,b_{i+1}-b_i)=(-1,0)$ for some $i=1,\dots,m-1$.
	\end{enumerate}
 
    \noindent We say that $C$ and $D$ are \textit{connected} in $\cS$ if there exists a path of cells in $\cS$ from $C$ to $D$. A \textit{polyomino} $\cP$ is a non-empty, finite collection of cells in $\Z^2$ where any two cells of $\cP$ are connected in $\cP$. For instance, see Figure \ref{Figure: Polyomino introduction}.
	\begin{figure}[h]
		\centering
		\includegraphics[scale=0.6]{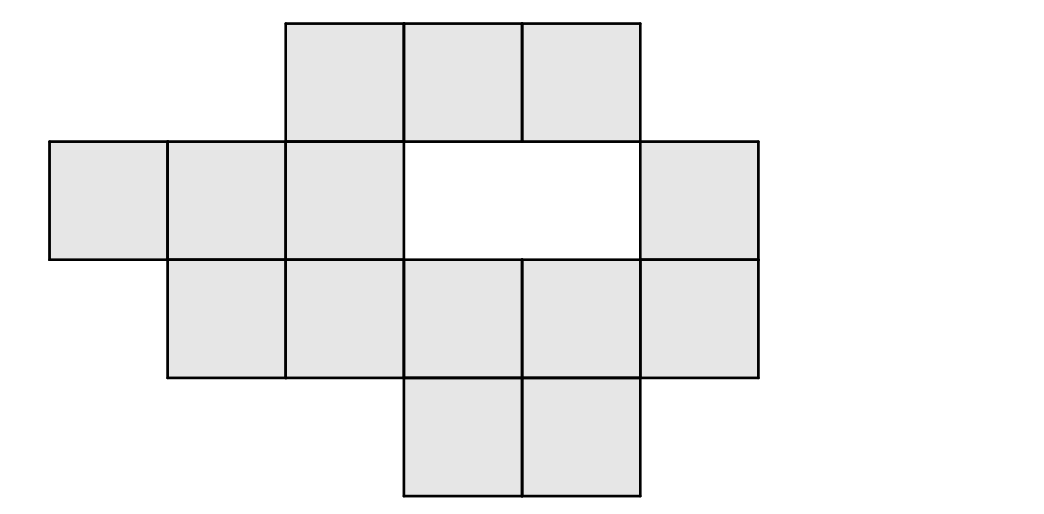}
		\caption{A polyomino.}
		\label{Figure: Polyomino introduction}
	\end{figure}
	
	\noindent Let $\cP$ be a polyomino. A \textit{sub-polyomino} of $\cP$ is a polyomino whose cells belong to $\cP$. We say that $\cP$ is \textit{simple} if for any two cells $C$ and $D$ not in $\cP$ there exists a path of cells not in $\cP$ from $C$ to $D$. A finite collection of cells $\cH$ not in $\cP$ is a \textit{hole} of $\cP$ if any two cells of $\cH$ are connected in $\cH$ and $\cH$ is maximal with respect to set inclusion. For example, the polyomino in Figure \ref{Figure: Polyomino introduction} is not simple. Obviously, each hole of $\cP$ is a simple polyomino, and $\cP$ is simple if and only if it has no hole. A polyomino is said to be \textit{thin} if it does not contain the square tetromino, which is a square obtained as a union of four distinct cells. The \textit{rank} of a polyomino, $\rank(\cP)$, is given by the number of its cells. Consider two cells $A$ and $B$ of $\Z^2$ with $a=(i,j)$ and $b=(k,l)$ as the lower left corners of $A$ and $B$ with $a\leq b$. A \textit{cell interval} $[A,B]$ is the set of the cells of $\Z^2$ with lower left corner $(r,s)$ such that $i\leqslant r\leqslant k$ and $j\leqslant s\leqslant l$. If $(i,j)$ and $(k,l)$ are in horizontal (or vertical) position, we say that the cells $A$ and $B$ are in \textit{horizontal} (or \textit{vertical}) \textit{position}.\\
	Let $\cP$ be a polyomino. Consider  two cells $A$ and $B$ of $\cP$ in vertical or horizontal position. 	 
	The cell interval $[A,B]$, containing $n>1$ cells, is called a
	\textit{block of $\cP$ of rank n} if all cells of $[A,B]$ belong to $\cP$. The cells $A$ and $B$ are called \textit{extremal} cells of $[A,B]$. Moreover, a block $\cB$ of $\cP$ is \textit{maximal} if there does not exist any block of $\cP$ which properly contains $\cB$. It is clear that an interval of $\ZZ^2$ identifies a cell interval of $\ZZ^2$ and vice versa, hence we can associate to an interval $I$ of $\ZZ^2$ the corresponding cell interval denoted by $\cP_{I}$. A proper interval $[a,b]$ is called an \textit{inner interval} of $\cP$ if all cells of $\cP_{[a,b]}$ belong to $\cP$. We denote by $\cI(\cP)$ the set of all inner intervals of $\cP$. An interval $[a,b]$ with $a=(i,j)$, $b=(k,j)$ and $i<k$ is called a \textit{horizontal edge interval} of $\cP$ if the sets $\{(\ell,j),(\ell+1,j)\}$ are edges of cells of $\cP$ for all $\ell=i,\dots,k-1$. In addition, if $\{(i-1,j),(i,j)\}$ and $\{(k,j),(k+1,j)\}$ do not belong to $E(\cP)$, then $[a,b]$ is called a \textit{maximal} horizontal edge interval of $\cP$. We define similarly a \textit{vertical edge interval} and a \textit{maximal} vertical edge interval. \\
	\noindent Following \cite{Trento}, we call a \textit{zig-zag walk} of $\cP$ a sequence $\cW:I_1,\dots,I_\ell$ of distinct inner intervals of $\cP$ where, for all $i=1,\dots,\ell$, the interval $I_i$ has either diagonal corners $v_i$, $z_i$ and anti-diagonal corners $u_i$, $v_{i+1}$, or anti-diagonal corners $v_i$, $z_i$ and diagonal corners $u_i$, $v_{i+1}$, such that:
	\begin{enumerate}
		\item $I_1\cap I_\ell=\{v_1=v_{\ell+1}\}$ and $I_i\cap I_{i+1}=\{v_{i+1}\}$, for all $i=1,\dots,\ell-1$;
		\item $v_i$ and $v_{i+1}$ are on the same edge interval of $\cP$, for all $i=1,\dots,\ell$;
		\item for all $i,j\in \{1,\dots,\ell\}$ with $i\neq j$, there exists no inner interval $J$ of $\cP$ such that $z_i$, $z_j$ belong to $J$.
	\end{enumerate}
	\noindent According to \cite{Cisto_Navarra_closed_path}, we recall the definition of a \textit{closed path polyomino}, and the configuration of cells characterizing its primality. We say that a polyomino $\cP$ is a \textit{closed path} if there exists a sequence of cells $A_1,\dots,A_n, A_{n+1}$, $n>5$, such that:
	\begin{enumerate}
		\item $A_1=A_{n+1}$;
		\item $A_i\cap A_{i+1}$ is a common edge, for all $i=1,\dots,n$;
		\item $A_i\neq A_j$, for all $i\neq j$ and $i,j\in \{1,\dots,n\}$;
		\item For all $i\in\{1,\dots,n\}$ and for all $j\notin\{i-2,i-1,i,i+1,i+2\}$, we have $V(A_i)\cap V(A_j)=\emptyset$, where $A_{-1}=A_{n-1}$, $A_0=A_n$, $A_{n+1}=A_1$ and $A_{n+2}=A_2$. 
	\end{enumerate}
 
	\begin{figure}[h]
		\centering
        \subfloat{\includegraphics[scale=0.5]{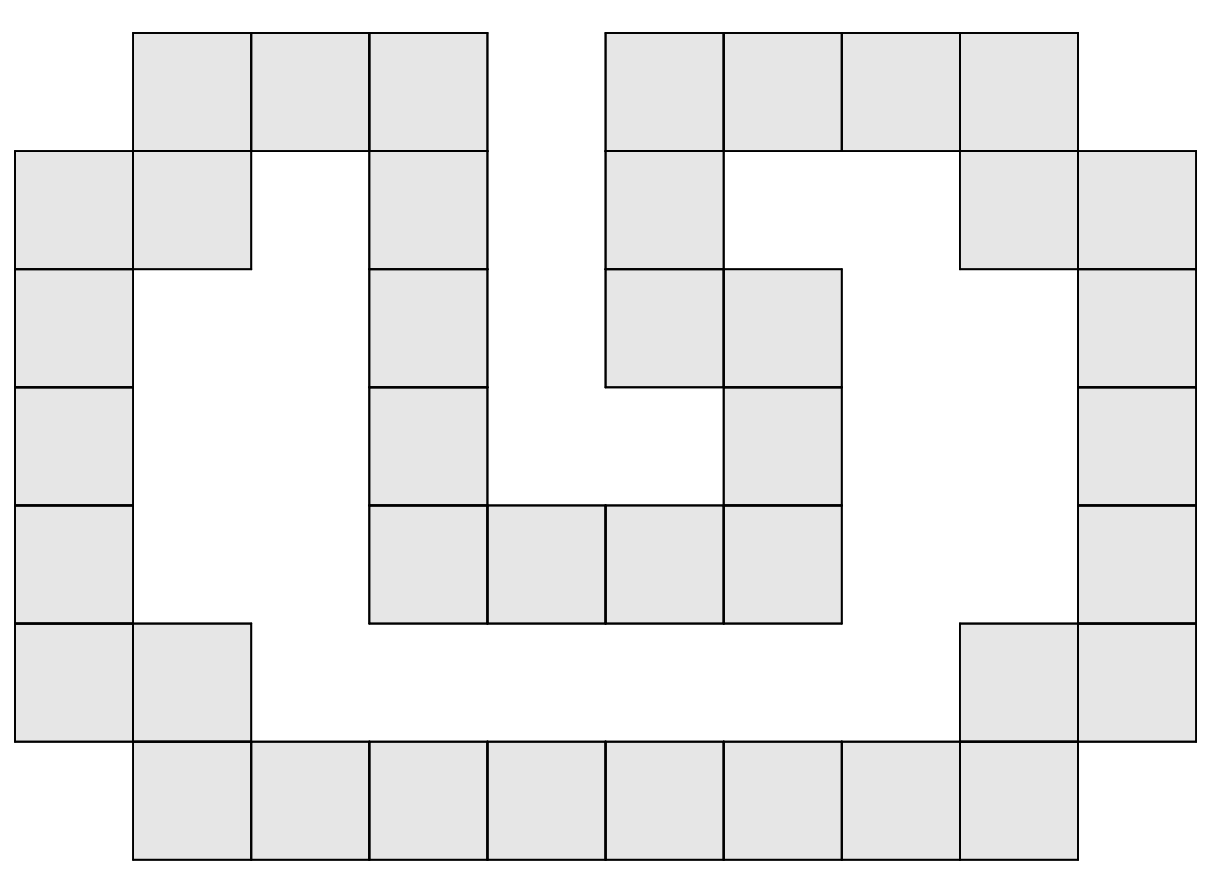}}\qquad\qquad
        \subfloat{\includegraphics[scale=0.5]{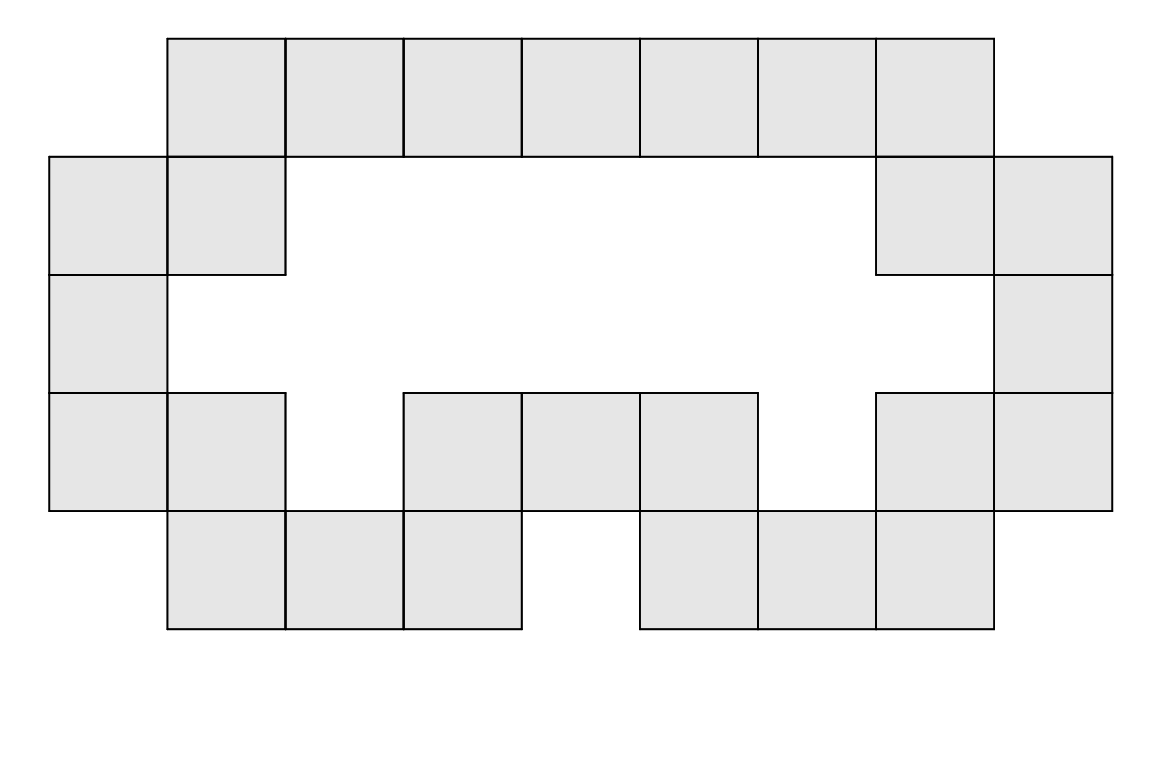}}
		\caption{An example of two closed paths.}
        \label{Figure: Example closed paths}
	\end{figure}
 
    \noindent A path of five cells $C_1, C_2, C_3, C_4, C_5$ of $\cP$ is called an \textit{L-configuration} if the two sequences $C_1, C_2, C_3$ and $C_3, C_4, C_5$ go in two orthogonal directions. A set $\cB=\{\cB_i\}_{i=1,\dots,n}$ of maximal horizontal (or vertical) blocks of rank at least two, with $V(\cB_i)\cap V(\cB_{i+1})=\{a_i,b_i\}$ and $a_i\neq b_i$ for all $i=1,\dots,n-1$, is called a \textit{ladder of $n$ steps} if $[a_i,b_i]$ is not on the same edge interval of $[a_{i+1},b_{i+1}]$ for all $i=1,\dots,n-2$. We recall that a closed path has no zig-zag walks if and only if it contains an $L$-configuration or a ladder of at least three steps (see \cite[Section 6]{Cisto_Navarra_closed_path}). For instance, in Figure \ref{Figure: Example closed paths} there is presented on the left side a closed path whose polyomino ideal is prime (so it does not contain zig-zag walks), and on the right side a closed path with zig-zag walks.

	\begin{lemma}\label{Lemma: Closed path number of vertices and cells}
		Let $\cP$ be a closed path polyomino. Then $|V(\cP)|=2\rank(\cP)$.
	\end{lemma}
	\begin{proof}
			From \cite[Lemma 3.3]{Cisto_Navarra_closed_path}, $\cP$ contains a block $\cB$ of rank at least three. We may assume that $\cB$ is in horizontal position and we may label the cells of $\cP$ taking $A_{n-1},A_n$ and $A_1$ in $\cB$ as shown in Figure~\ref{Figure:sub-polyomino for prove number vertices are double of number cells} (A). We want to inductively assign a pair of vertices of $\cP$ to every cell of $\cP$. Let us start to attach to $A_1$ the lower and the upper right corners of $A_1$. Let $i\geq 1$ and consider the set $\cB_i=\{A_{i-1},A_i,A_{i+1}\}$. If $\cB_i$ is as in Figure \ref{Figure:sub-polyomino for prove number vertices are double of number cells} (B), up to rotations, then we attach to $A_{i+1}$ the two vertices of $\cP$ marked in red. Otherwise, if $\cB_i$ has the shape as in Figure~\ref{Figure:sub-polyomino for prove number vertices are double of number cells} (C), up to rotations or reflections, then we attach to $A_{i+1}$ the two blue vertices of $\cP$. This procedure ends after $n-1$ steps, considering the set $\cB_{n-1}=\{A_{n-2},A_{n-1},A_{n}\}$ and attaching to $A_n$ the lower and the upper right corners of $A_n$. Therefore, we can attach to every cell of $\cP$ two distinct vertices as defined before, and in conclusion, we get that $|V(\cP)|=2\rank(\cP)$.

	\begin{figure}[h!]
	\subfloat[]{\includegraphics[scale=0.8]{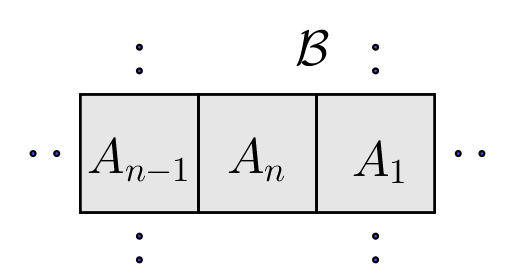}}\qquad
	\subfloat[]{\includegraphics[scale=0.8]{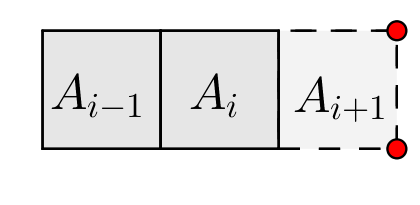}}\qquad
	\subfloat[]{\includegraphics[scale=0.8]{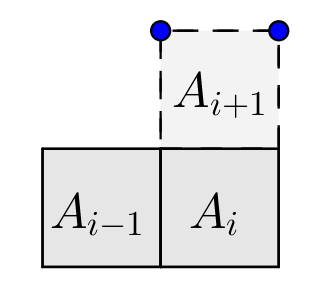}}\qquad
	\caption{Changes of directions in a closed path.}
	\label{Figure:sub-polyomino for prove number vertices are double of number cells}
\end{figure}
		
	\end{proof}
	\noindent Let $\cP$ be a polyomino. We set $S_\cP=K[x_v\:\; v\in V(\cP)]$, where $K$ is a field. If $[a,b]$ is an inner interval of $\cP$, with $a$,$b$ and $c$,$d$ respectively diagonal and anti-diagonal corners, then the binomial $x_ax_b-x_cx_d$ is called an \textit{inner 2-minor} of $\cP$. We define $I_{\cP}$ as the ideal in $S_\cP$ generated by all the inner 2-minors of $\cP$ and we call it the \textit{polyomino ideal} of $\cP$. We set also $K[\cP] = S_\cP/I_{\cP}$, which is the \textit{coordinate ring} of $\cP$. \\
	
	\section{Krull dimension of closed path polyominoes}\label{Krull}
	
	\noindent In this section we compute the Krull dimension of the coordinate ring attached to a closed path polyomino. We recall some basic facts on simplicial complexes. A \textit{finite simplicial complex} $\Delta$ on the vertex set $[n]=\{1,\dots,n\}$ is a
	collection of subsets of $[n]$ satisfying the following two conditions:
	\begin{enumerate}
		\item if $F'\in \Delta$ and $F \subseteq F'$ then $F \in \Delta$;
		\item $\{i\}\in \Delta$ for all $i=1,\dots,n$.
	\end{enumerate}  The elements of $\Delta$ are called \textit{faces}, and the dimension of a face is one less than its cardinality. An \textit{edge} of $\Delta$ is a face of dimension 1, while a \textit{vertex} of $\Delta$ is a face of dimension 0. The maximal faces of $\Delta$ with respect to the set inclusion are called \textit{facets}. The dimension of $\Delta$ is given by $\sup\{\dim(F):F\in \Delta\}$. We say that a simplicial complex $\Delta$ is \textit{pure} if all the facets have the same dimension. If $\Delta$ is pure, then the dimension of $\Delta$ is given trivially by the dimension of a facet of $\Delta$. Let $\Delta$ be a simplicial complex on $[n]$ and $R=K[x_1,\dots,x_n]$ be the polynomial ring in $n$ variables over a field $K$. To every collection $F=\{i_1,\dots,i_r\}$ of $r$ distinct vertices of $\Delta$, we associate a monomial $x^F$ in $R$ where	$x^F=x_{i_1}\dots x_{i_r}.$ The monomial ideal generated by all monomials $x^F$ such that $F$ is not a face of $\Delta$ is called \textit{Stanley-Reisner ideal} and it is denoted by $I_{\Delta}$. The \textit{face ring} of $\Delta$, denoted by $K[\Delta]$, is defined to be the quotient ring $R/I_{\Delta}$. It follows from \cite[Corollary 6.3.5]{Villareal}, if $\Delta$ is a simplicial complex on $[n]$ of dimension $d$, then $\dim K[\Delta]=d+1=\max\{s\:\; x_{i_1}\dots x_{i_s}\notin I_{\Delta},i_1<\dots<i_s\}$.
	
	\begin{defn}\rm \label{Definition: Gamma-path...}
		A polyomino $\cR$ is called \textit{$\Gamma$-path} with middle cell $D$ and hooking vertex $w$ if it consists of a maximal horizontal block $\cB_1=[C_1,B_1]$ of rank greater than three and a maximal vertical block $\cB_2=[B_2,C_2]$ of rank greater than three such that there is a cell $D$ not belonging to $\cB_{1}\cup\cB_2$ and $V(\cB_1)\cap V(\cB_2)=\{w\}$ where $w$ is the upper left corner of $D$. This is illustrated in Figure~ \ref{Figure:configurations of W-paths} (A). With reference to Figure \ref{Figure:configurations of W-paths}, we define $\tau$-paths, $W$-paths and $\zeta$-paths in a similar way.
		
		\begin{figure}[H]
			\subfloat[$\Gamma$-path]{\includegraphics[scale=0.7]{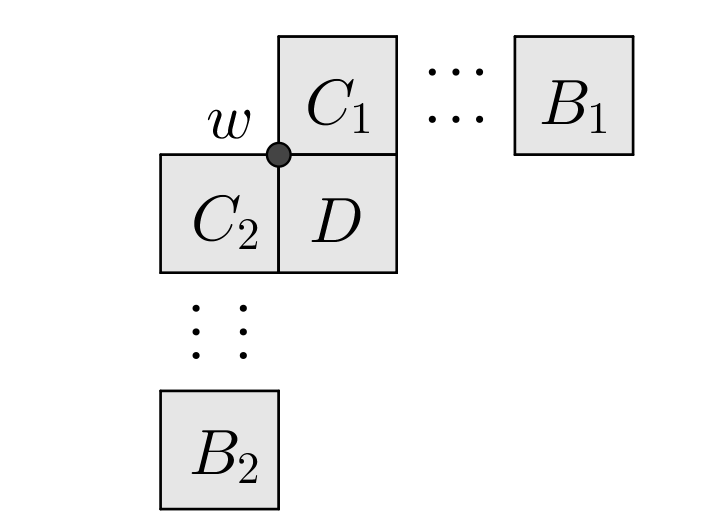}}
			\subfloat[$\tau$-path]{\includegraphics[scale=0.7]{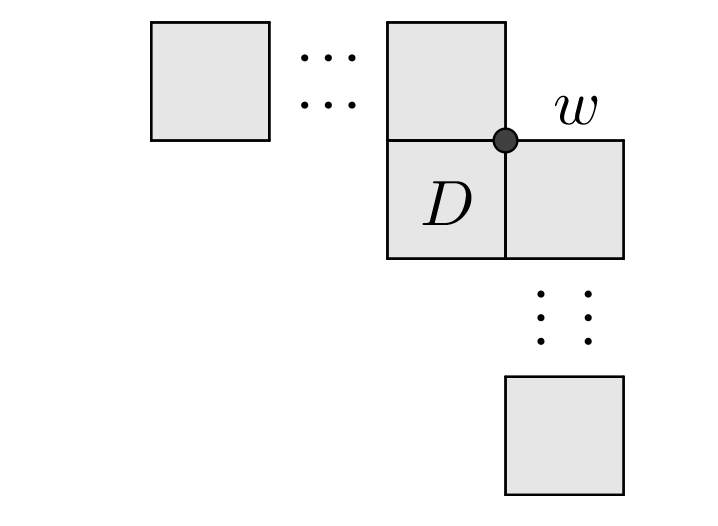}}
			\subfloat[$W$-path]{\includegraphics[scale=0.7]{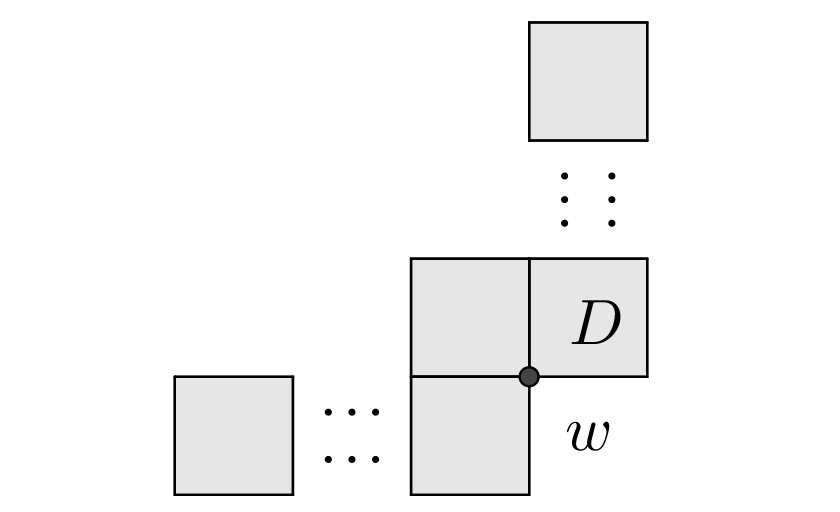}}
			\subfloat[$\zeta$-path]{\includegraphics[scale=0.7]{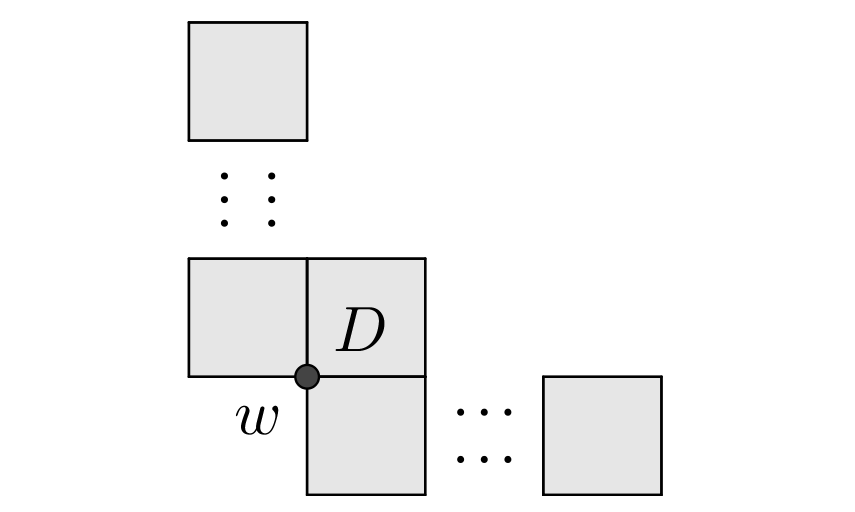}}
			\caption{Some configurations in a closed path with zig-zag walks.}
			\label{Figure:configurations of W-paths}
		\end{figure} 
		
		\noindent Moreover, we say that a pair $(\cF,\cG)$ of the previous polyominoes is \textit{compatible} if $\cF$ is a $\Gamma$-path ($\tau$-path or $W$-path or $\zeta$-path) and $\cG$ is one of the other paths.
	\end{defn}
	
	\begin{defn}\rm \label{Definition: Skew path...}
		A polyomino $\cS$ is called a \textit{LU-skew path} with hooking vertices $a,b$ if it is made up of two maximal blocks $\cB_1=[C_1,D_1]$ and $\cB_2=[C_2,D_2]$, both of them having rank greater than three, with $V(\cB_1)\cap V(\cB_2)=\{a,b\}$ where $a,b$ are the left and right upper corners of $D_1$, respectively. For instance, see Figure~\ref{Figure:configurations of Skew} (A). With reference to Figure~\ref{Figure:configurations of Skew}, we can define $LD$-skew, $DU$-skew and $UD$-skew paths in a similar way.
		
		\begin{figure}[h!]
            \flushleft\qquad\qquad\qquad
			\begin{minipage}{.3\textwidth}
				\subfloat[$LU$-skew]{\includegraphics[scale=0.7]{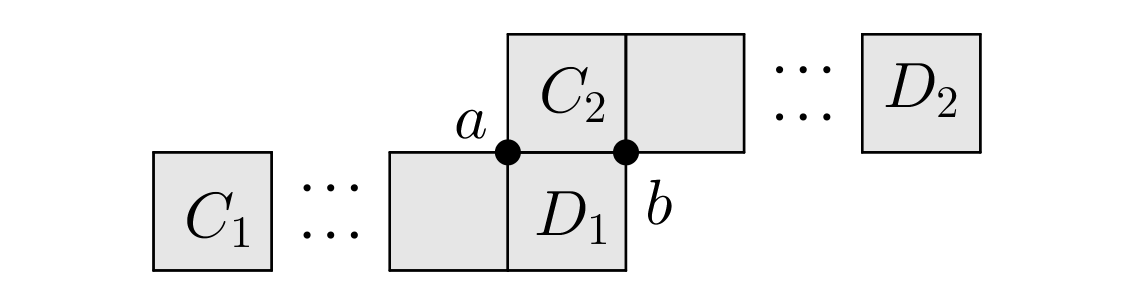}}
				\vspace{0.3cm}
				\subfloat[$LD$-skew]{\includegraphics[scale=0.7]{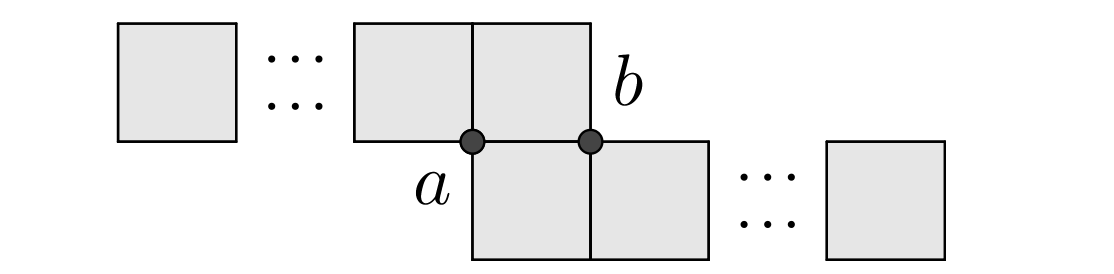}}
			\end{minipage}
			\begin{minipage}{.3\textwidth}
				\subfloat[$DU$-skew]{\includegraphics[scale=0.7]{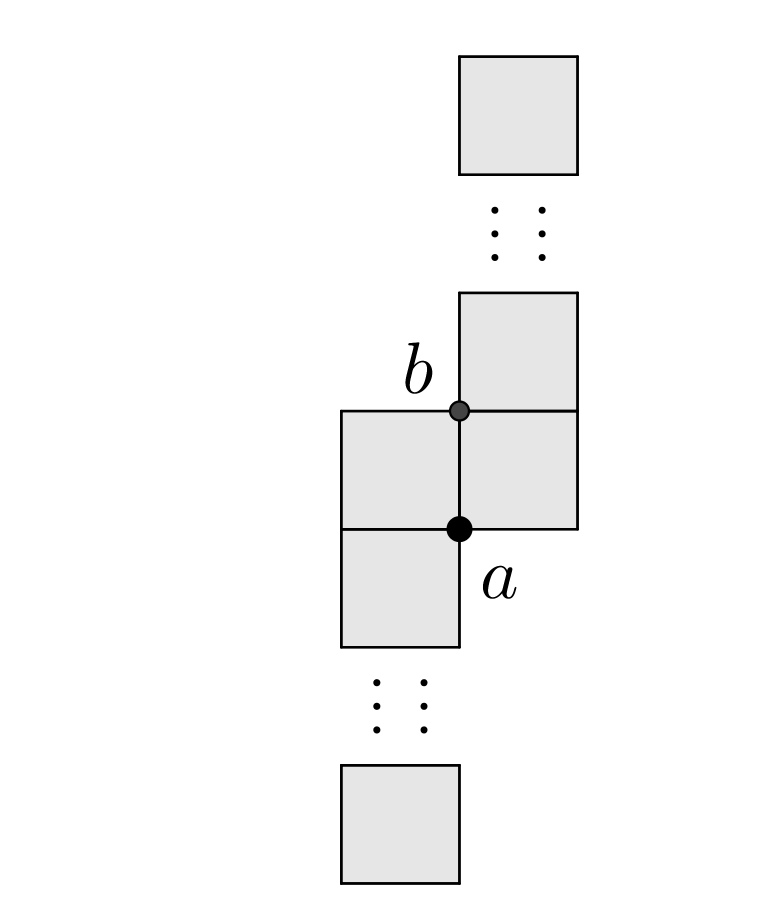}}
				\subfloat[$UD$-skew]{\includegraphics[scale=0.7]{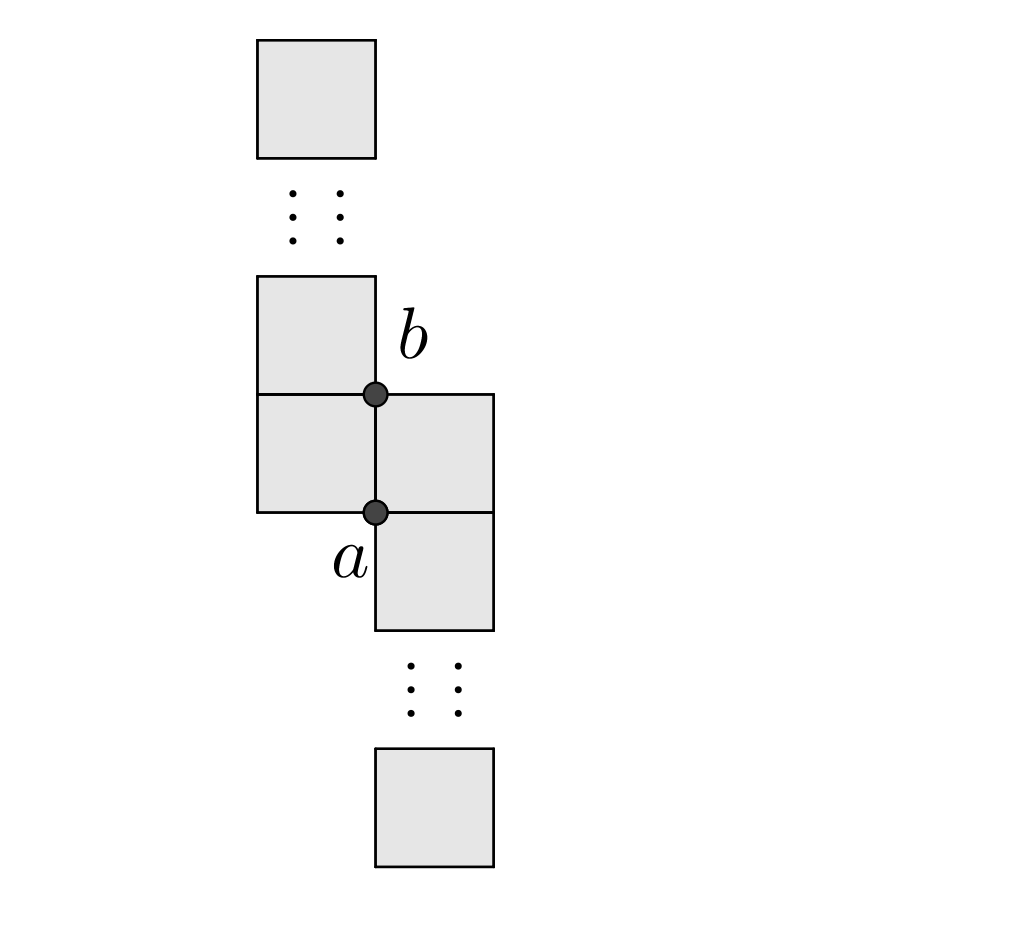}}
			\end{minipage}
			\caption{Some configurations in a closed path with zig-zag walks.}
			\label{Figure:configurations of Skew}
		\end{figure} 
	\end{defn}
	
	\noindent Let $\cP$ be a closed path and $(\cF,\cG)$ be a pair of two sub-polyominoes of $\cP$ as in Definition \ref{Definition: Gamma-path...}. Without loss of generality, we may assume that the middle cell of $\cF$ is $A_1$ and the middle cell of $\cG$ is $A_k$ with $k\in [n-1]$. Then a sub-polyomino $\cQ$ of $\cP$ as in Definition \ref{Definition: Skew path...} is said to be \textit{between $\cF$ and $\cG$} if $\cQ$ is contained in $\{A_i:1<i<k\}$. \\
	\noindent In the following remark we point out the structure of a closed path $\cP$ having a zig-zag walk, showing that $\cP$ consists of the configurations defined in Definitions \ref{Definition: Gamma-path...} and \ref{Definition: Skew path...}, arranged in a suitable way. The latter is essentially a technical consequence of the characterization given in \cite[Section 6]{Cisto_Navarra_closed_path} which states that a closed path has a zig-zag walk if and only if it has neither an $L$-configuration nor a ladder of at least three steps.
	
	\begin{rmk}\rm \label{Remark: for the arrangements of configurations}
	Let $\cP$ be a closed path with a zig-zag walk. We denote by $\cB_1$ a maximal block of $\cP$ having its length of at least three. It is not restrictive to assume that $\cB_{1}=[A_1, A_k]$, where $k\geq 3$, is in horizontal position and that $A_n$ is at North of $A_1$; otherwise we can rotate $\cP$ suitably or relabel the cells of $\cP$.	Observe that $A_{n-1}$ is necessarily at West of $A_n$, otherwise if it is at North then $\cB_{1}\cup \{A_{n-1},A_{n}\}$ contains an $L$-configuration. Consider now $A_{k+1}$ and note that it is at North of $A_k$ because it cannot be at East from the maximality of $\cB_1$ and it cannot be at South otherwise either $\cB_1\cup \{A_{k+1},A_{k+2}\}$ contains an $L$-configuration if $A_{k+2}$ is at South of $A_{k+1}$ or $\{A_{n-1},A_n\}\cup \cB_1\cup \{A_{k+1},A_{k+2}\}$ is contained in a ladder of three steps if $A_{k+2}$ is at East of $A_{k+1}$. For similar arguments $A_{k+2}$ is necessarily at East of $A_{k+1}$. Now, we want to define inductively the configurations of cells that appear in $\cP$, starting from $\cB_1$. Set $h\geq 1$. Let $\cB_{h}$ be a maximal block with at least three cells and we may assume that $\cB_{h}=[A_{j_h},A_{j_{h+1}}]$, where $k<j_{h}<j_{h+1}<n$, is in horizontal position and that $A_{j_h-1}$ is at North of $A_{j_h}$, otherwise it is sufficient to reflect $\cP$. For arguments similar to the ones done before, we can define the maximal block $\cB_{h+1}=[A_{j_{h+2}},A_{j_{h+3}}]$ of $\cP$ having at least three cells, arranged to $\cB_h$ following Figure \ref{Figure: Arrangements}. 
	
	\begin{figure}[h]
		\subfloat{\includegraphics[scale=0.7]{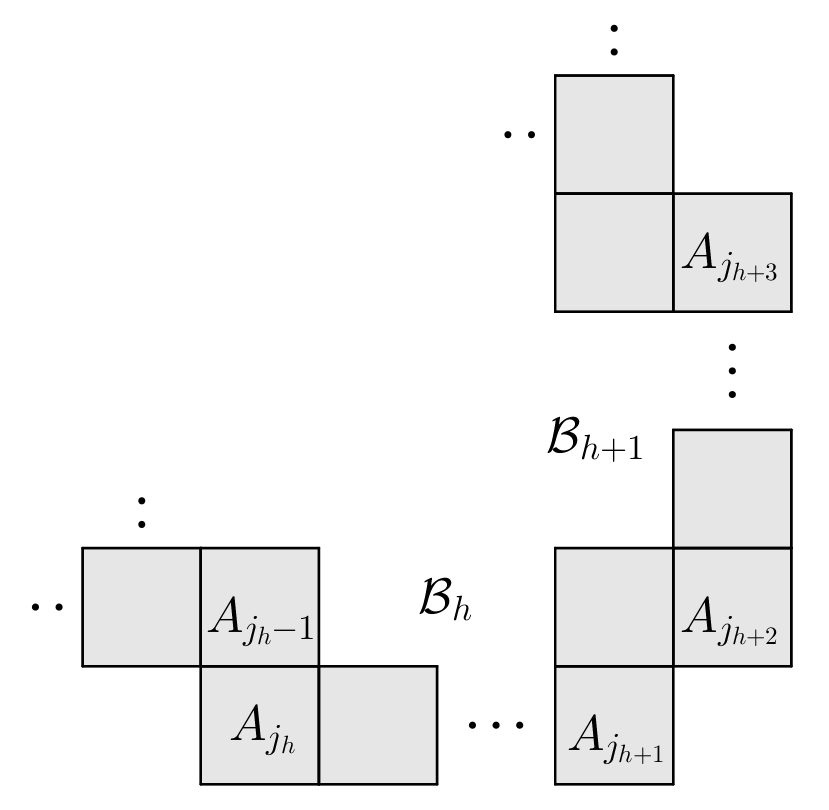}}\qquad\qquad
		\subfloat{\includegraphics[scale=0.7]{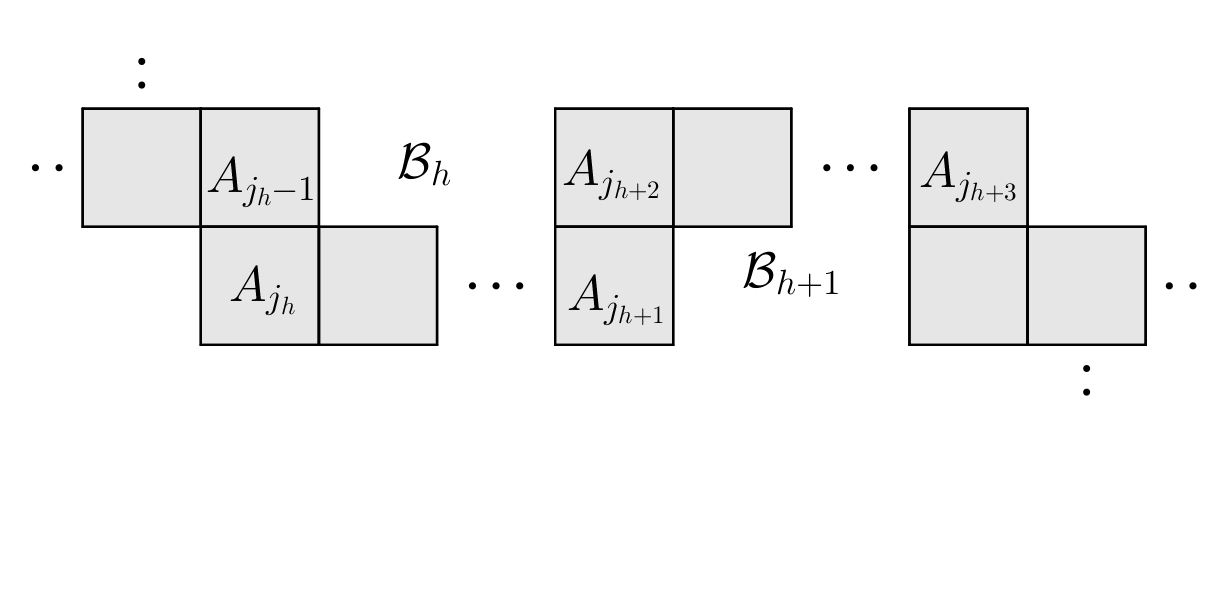}}
		\caption{Arrangements of $\cB_{h}$ and $\cB_{h+1}$.}
		\label{Figure: Arrangements}
	\end{figure} 

	\noindent Since $\cP$ is a closed path, this procedure is finite so there exists $t\in \NN$ such that $\cB_{t+1}=\cB_1$ and $\cB_{t}$ and $\cB_1$ are arranged as in Figure \ref{Figure: Arrangements at the end}.
	
	\begin{figure}[h]
		\subfloat{\includegraphics[scale=0.7]{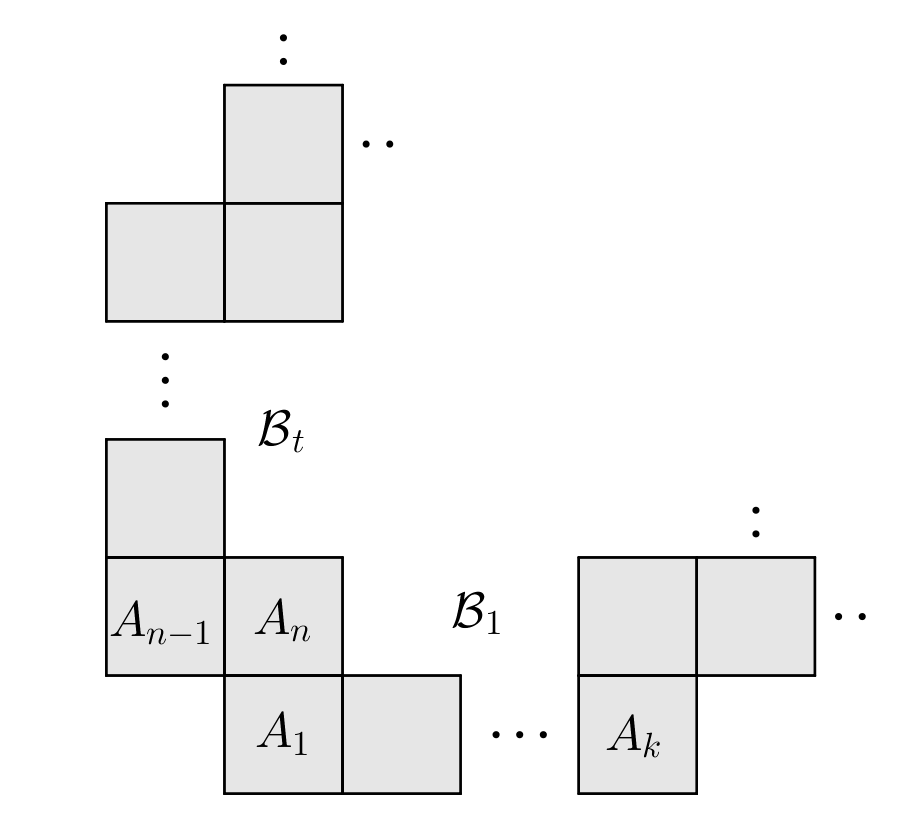}}\qquad
		\subfloat{\includegraphics[scale=0.7]{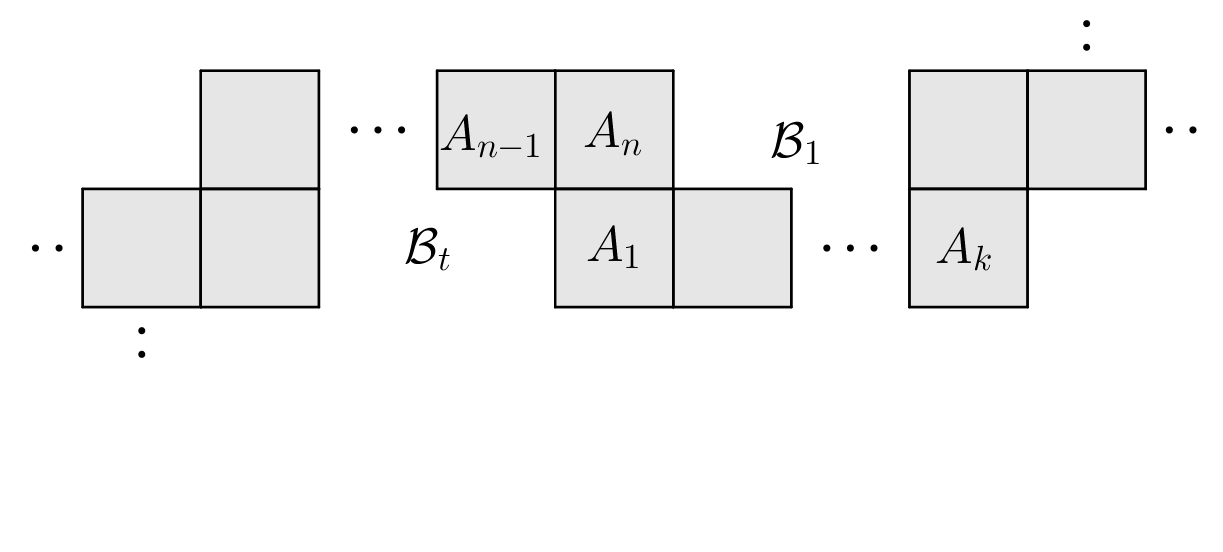}}
		\caption{Arrangements of $\cB_{t}$ and $\cB_{1}$.}
		\label{Figure: Arrangements at the end}
	\end{figure}

	\end{rmk}

 \begin{defn}\rm \label{Defn:order}
     Let $\cP$ be a closed path polyomino. Let $<^1$ be the total order on $V(\cP)$ defined as $u<^1 v$ if and only if, for $u = (i,j)$ and $v = (k,l)$, $i < k$, or $i = k$ and $j < l$. Let $Y\subset V(\cP)$ and consider $<^Y_{\mathrm{lex}}$ be the lexicographical order in $S_\cP$ induced by the following order on the variables of $S_\cP$:
	\[ \mbox{for}\ u,v \in V(\cP)\qquad
	x_u<^Y_{\mathrm{lex}} x_v \Leftrightarrow
	\left\{
	\begin{array}{l}
		u\notin Y\ \mbox{and}\ v\in Y \\
		u,v\notin Y\ \mbox{and}\ u<^1 v \\
		u,v\in Y\ \mbox{and}\ u<^1 v
	\end{array}
	\right.
	\]
	From \cite[Theorem 4.9] {Cisto_Navarra_CM_closed_path}, we know that there exists a suitable set $Y\subset V(\cP)$ such that the set of generators of $I_{\cP}$ forms the reduced Gr\"obner basis of $I_\cP$ with respect to $<^Y_{\mathrm{lex}}$, defined in \cite[Algorithm 4.7, Definition 4.8]{Cisto_Navarra_CM_closed_path}. We denote by $\Delta(\cP)$ the simplicial complex on the vertex set of $\cP$ having the square-free monomial ideal $J_{\cP}=\mathrm{in}_{<^Y_{\mathrm{lex}}}(I_{\cP})$ as the Stanley-Reisner ideal. 
 \end{defn}
 
	\begin{thm}\label{Thm: Dimension closed path}
		Let $\cP$ be a closed path and $\Delta(\cP)$ be the associated simplicial complex. Then $\Delta(\cP)$ is pure and the Krull dimension of $K[\cP]$ is $\vert V(\cP)\vert-\mathrm{rank}(\cP)$.
	\end{thm}
	
	\begin{proof}
		Firstly, we assume that $\cP$ does not contain any zig-zag walk. We know that $I_{\cP}$ is a toric ideal from \cite[Theorem 6.2]{Cisto_Navarra_closed_path}, so $\Delta_{\cP}$ is pure from \cite[Theorem 9.6.1]{Villareal}. Moreover, from \cite[Theorem 5.5]{Cisto_Navarra_Hilbert_series}, the Krull dimension is given by $\vert V(\cP)\vert-\mathrm{rank}(\cP)$. \\ 
		We need to examine only the case when $\cP$ contains a zig-zag walk. Suppose that we are in this case. Hence, in what follows, our aim is to define a suitable facet of $\Delta(\cP)$. We have that $\cP$ contains just the configurations defined in Definitions \ref{Definition: Gamma-path...} and \ref{Definition: Skew path...}, arranged as described in Remark \ref{Remark: for the arrangements of configurations}. Let $\cS_1$ be a $\Gamma$-path. Referring to Figure \ref{Figure:configurations of W-paths} (A), we label the cells of $\cP$ setting $D=A_1$, $C_1=A_2$ and so on. Let $k>1$ be the minimum integer such that $\cS_2$ is either a $W$-path or a $\tau$-path with middle cell $A_k$. The hooking vertices of $\cS_1$ and $\cS_2$ are on the same maximal horizontal edge interval $V$ of $\cP$ and, if there exist, the hooking vertices of the $LU$-skew or $LD$-skew paths between $\cS_1$ and $\cS_2$ belong to $V$. Moreover, by the minimality of $k$, there does not exist any $DU$-skew and $UD$-skew path between $\cS_1$ and $\cS_2$ belonging to $V$. We want to define a suitable set of vertices of $\cP$ in order to find a facet of $\Delta(\cP)$. We distinguish two cases, depending on if $\cS_2$ is a $W$-path or a $\tau$-path. We may consider $z_T,x_T\in Y$ and $y_T\notin Y$, with reference to Figure \ref{Figure: I case of I case, proof dimension}.\\
		\textsl{Case I:} Assume that $\cS_2$ is a $W$-path.
		
		\begin{enumerate}
			\item Suppose that there does not exist any $LU$-skew or $LD$-skew paths between $\cS_1$ and $\cS_2$. Let $\cB_1$ be the maximal horizontal block $[A_2,A_{k-1}]$. Denote the upper right corner of $A_{i}$ by $a_i^1$ for all $i=2,\dots,k-1$ and the upper left corner of $A_{2}$ by $a_{1}^1$. We set 
			$$V(\cS_1,\cS_2)=\{a_i^1:i\in [\mathrm{rank}(\cB_1)]\},$$ as in Figure \ref*{Figure: I case of I case, proof dimension}. 
			
			\begin{figure}[h]
				\centering
				\includegraphics[scale=0.7]{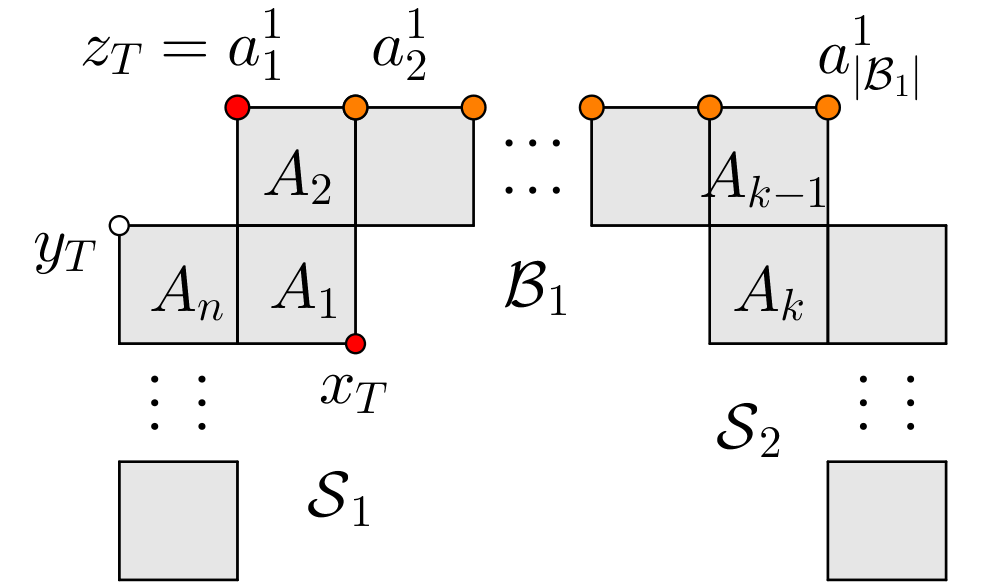}
				\caption{Arrangements of cells in \textsl{Case I}-(1)}
				\label{Figure: I case of I case, proof dimension}
			\end{figure}
			
			\item Suppose that there exist $LD$-skew or $LU$-skew paths between $\cS_1$ and $\cS_2$ and, in particular, that there are $m$ maximal horizontal blocks $\cB_i$ between $\cS_1$ and $\cS_2$. Set $\cB_j=[A_{k_j},A_{k_{j+1}}]$, where $k_1=2$ and $k_j<k_{j+1}$ for all $j\in [m]$. For all $j\in [m]$ odd and for all $i=1,\dots,\mathrm{rank}(\cB_j)$, we denote by $a_1^j$ the upper left corner of $A_{k_j}$ and by $a_i^j$ the right upper corner of $A_{k_j+i}$ in $\cB_j$.
			For all $j\in [m]$ even and for all $i=1,\dots,\mathrm{rank}(\cB_j)$, we denote by $a_i^j$ the lower right corner of $A_{k_j+i}$ in $\cB_j$. See Figure~\ref{Figure: II case of I case, proof dimension}.
			
			\begin{figure}[h]
				\centering
				\includegraphics[scale=0.7]{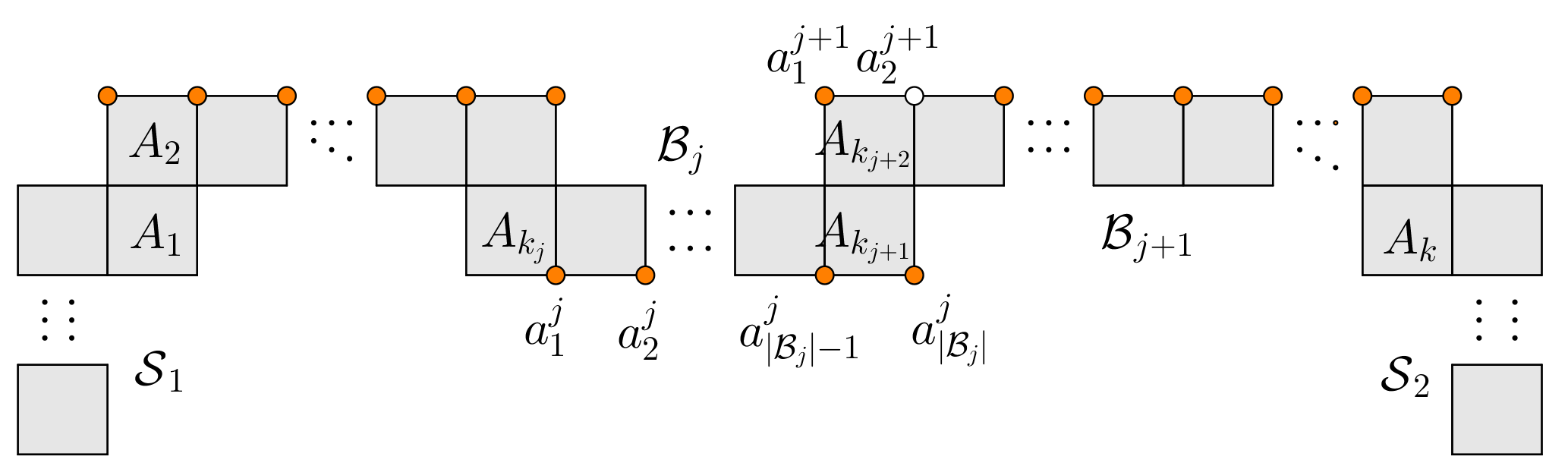}
				\caption{Arrangements of cells in \textsl{Case I}-(2)}
				\label{Figure: II case of I case, proof dimension}
			\end{figure}
\noindent Recall that for all $j\in [m]$, $a^j_{\vert\cB_j\vert-2},a^j_{\vert\cB_j\vert-1}\notin Y$ and $a^{j+1}_{1},a^{j+1}_{2}\in Y$ with $a^{j+1}_{2}>a^{j+1}_{1}$. Then we set
			$$ V(\cS_1,\cS_2)=\bigcup_{j\in [m]\atop even} \{a_i^j:i\in [\mathrm{rank}(\cB_j)]\} \cup \biggl(\bigcup_{j\in [m]\atop odd} \{a_i^j:i\in [\mathrm{rank}(\cB_j)+1]\}\biggr)\backslash \bigcup_{j\in [m]\atop odd}\{a_2^j\}. $$ 
		\end{enumerate}
		
		\textsl{Case II:} Assume that $\cS_2$ is a $\tau$-path.
		\begin{enumerate}
			\item Suppose that there exists just an $LD$-skew path between $\cS_1$ and $\cS_2$. For all $i=0,\dots,\mathrm{rank}(\cB_1)-1$, we denote by $a_1^1$ the upper left corner of $A_2$ and by $a_i^1$ the right upper corner of $A_{k_j+i}$ in $\cB_j$. 
			For all $i=1,\dots,\mathrm{rank}(\cB_2)-1$, we denote by $a_i^2$ the lower right corner of $A_{t+i}$ in $\cB_2$. See Figure \ref{Figure: I case of II case, proof dimension}. In this case, we set 
			$$ V(\cS_1,\cS_2)=\{a_i^1:i\in [\mathrm{rank}(\cB_1)]\} \cup\{a_i^2:i\in [\mathrm{rank}(\cB_2)]\}.$$
			
			\begin{figure}[h]
				\centering
				\includegraphics[scale=0.7]{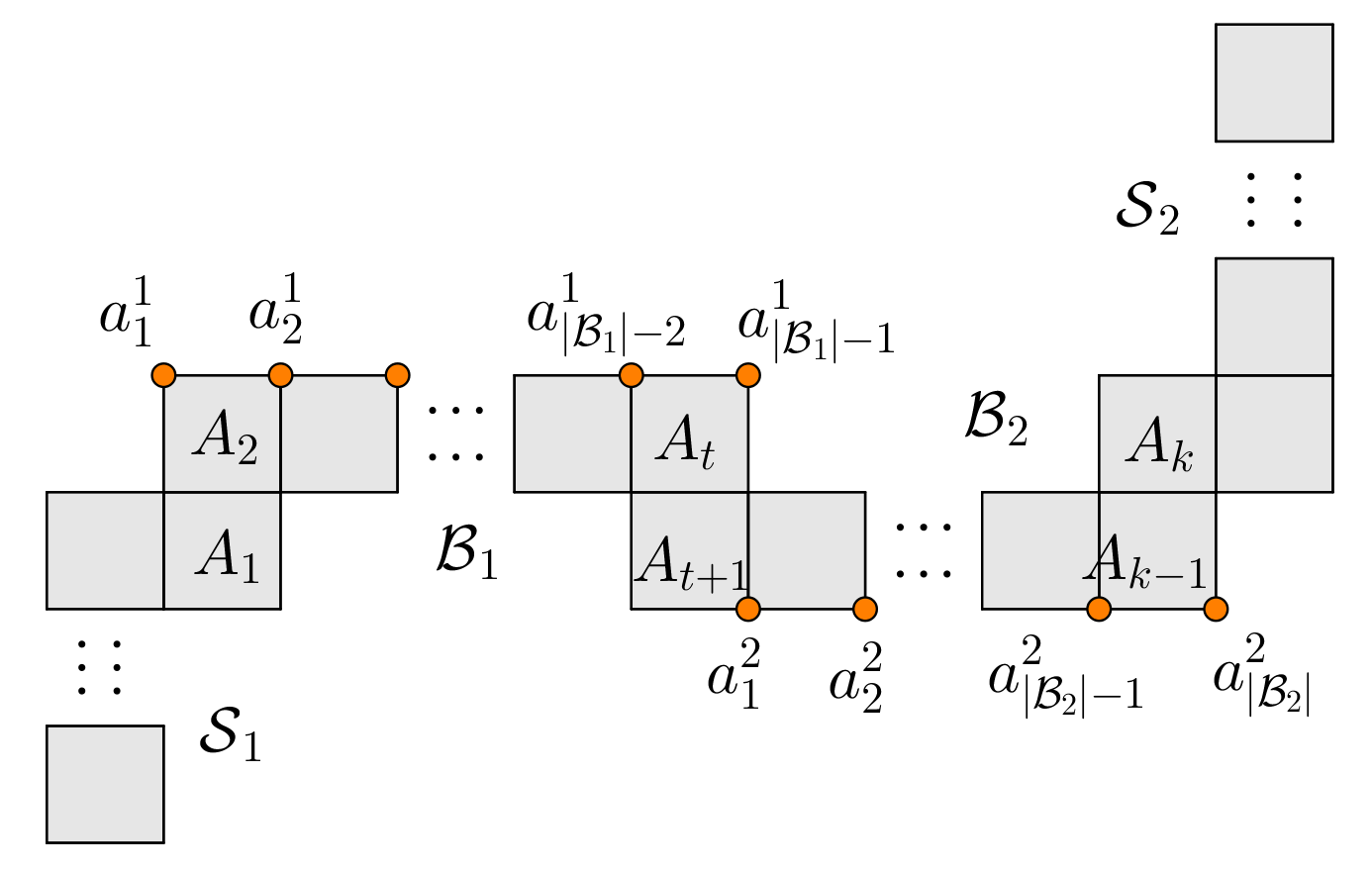}
				\caption{Arrangements of cells in \textsl{Case II}-(1)}
				\label{Figure: I case of II case, proof dimension}
			\end{figure}
			
			\item Suppose that there exist $LD$-skew or $LU$-skew paths between $\cS_1$ and $\cS_2$. With the same notations as in (2) of \textsl{Case I} (see Figure \ref{Figure: II case of II case, proof dimension}), we set
			$$ V(\cS_1,\cS_2)=\bigcup_{j\in [m]\atop even} \{a_i^j:i\in [\mathrm{rank}(\cB_j)]\} \cup \biggl(\bigcup_{j\in [m]\atop odd} \{a_i^j:i\in [\mathrm{rank}(\cB_j)]\}\biggr)\backslash \bigcup_{j\in [m]\atop odd}\{a_2^j\}. $$ 
			
			\begin{figure}[h]
				\centering
				\includegraphics[scale=0.7]{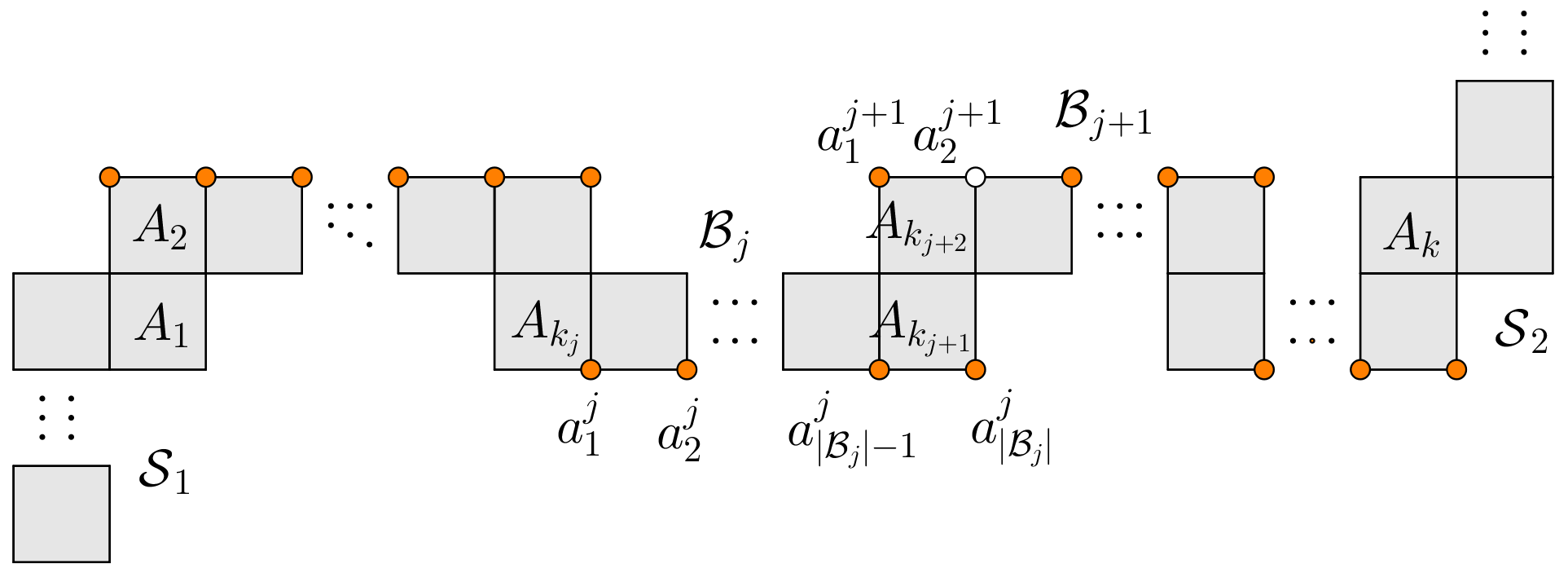}
				\caption{Arrangements of cells in \textsl{Case 2}-(2)}
				\label{Figure: II case of II case, proof dimension}
			\end{figure}
			\end{enumerate}
		
		\noindent In each case, we can define the following bijective correspondence $\phi_{1,2}:V(\cS_1,\cS_2)\to \{A_1\}\cup \bigcup_{i=1}^m \cB_i$ with $\phi_{1,2}(a_1^1)=A_1$ and $\phi_{1,2}(a_i^j)=A_{k_j+i}$ for all $j\in [m]$ and for all $i\in[\mathrm{rank}(\cB_j)]$; in particular, $\vert V(\cS_1,\cS_2)\vert=\sum_{i=1}^{m}\mathrm{rank}(\cB_i)+1$. Once we have defined the set $V(\cS_1,\cS_2)$, we can use the same arguments for all pairs $(\cS_i,\cS_{i+1})$ of compatible paths, where $\cS_{i+1}$ is taken by minimality with respect $\cS_i$, as done for $(\cS_1,\cS_2)$. Assume that there exist $t$ such pairs, where $\cS_{t+1}=\cS_1$. So as done before, for all $i\in [t]$, we can define $V(\cS_i,\cS_{i+1})$ by similar arguments as for $V(\cS_1,\cS_2)$. Hence, $V=\cup_{i=1}^t V(\cS_i,\cS_{i+1})$.
		Observe that $\vert V\vert=\mathrm{rank} (\cP)$. We want to prove that $F=\{x_v:v\in V\}$ is a facet of $\Delta(\cP)$. Obviously, $F$ is a face of $\Delta(\cP)$, so we need to prove the maximality with respect to the set inclusion. Suppose, by contradiction, that $F$ is not maximal. Then there exists a $p\in V(\cP)\backslash V$ such that $F\subset F\cup \{x_p\}$. Without loss of generality, we may assume that $p$ is a vertex of the sub-polyomino $\{A_1\}\cup \bigcup_{i=1}^m \cB_i$ between $\cS_1$ and $\cS_2$. 
		\begin{enumerate}
			\item Assume that we are in $(1)$ of \textsl{Case I}. Suppose $p\in V(\cB_1)$. If $p$ is the lower left corner of $A_2$, then since $y_T\notin Y$, we get the contradiction that $x_px_q\in J_{\cP}$, where $q$ is the lower left corner of $A_n$. In the other cases, we obtain similarly a contradiction considering $\{x_{a_1^1},x_p\}$, as well as when $p$ is the lower left or right corner of $A_1$.  
			\item Assume that we are in $(2)$ of \textsl{Case I}. The only case which we discuss is when $p=a_{2}^{j+1}$ for some $j\in[m]$. In this case, since $a_{2}^{j+1}>a_{1}^{j+1}$ we have $x_{a^j_{\vert\cB_j\vert-1}}x_{a_2^{j+1}}\in J_{\cP}$. But $\{x_{a^j_{\vert\cB_j\vert-1}},x_{a_2^{j+1}}\}\in \Delta(\cP)$ since $p=a_{2}^{j+1}$, so we get a contradiction.
			\item In the subcases $(1)$ and $(2)$ of \textsl{Case II}, we can argue in a similar way, and we obtain a contradiction.
		\end{enumerate}
		
	\noindent Therefore, $F$ is a facet of $\Delta(\cP)$. Now, we prove that all other facets of $\Delta(\cP)$ can be obtained from $F$ replacing some vertices of $F$ with other ones in the same number. Let $G$ be a facet of $\Delta(\cP)$. Since $\cP$ is union of the configurations defined in Definitions \ref{Definition: Gamma-path...} and \ref{Definition: Skew path...}, arranged suitably as explained in Remark \ref{Remark: for the arrangements of configurations}, it is sufficient to study the behavior of $G$ in the cases described in Figures \ref{Figure: I case of I case, proof dimension}, \ref{Figure: II case of I case, proof dimension}, \ref{Figure: I case of II case, proof dimension}, \ref{Figure: II case of II case, proof dimension}. Without losing of generality, we may examine the case in Figure \ref{Figure: I case of I case, proof dimension} and, with reference to Figure \ref{Figure: Proof for pure simplicial 1}, it is sufficient to focus on the vertices $\cL=\{a_i,b_i\}_{i\in [n]}\cup \{c_1,c_2, e,w, f,g, d_1,d_2\}$ (where $[n]=\{1,\dots,n\}$). Note that the orange vertices are the points of $F$ belonging to the vertical cell intervals $\cB_2$ and $\cB_3$. Moreover, we recall that $c_1$, $d_{m-1}$ and $d_m$ cannot belong to $Y$, because $a_1=z_T$ and $w=x_T$ are in $Y$.    
	
	\begin{figure}[h]
		\centering
		\includegraphics[scale=0.7]{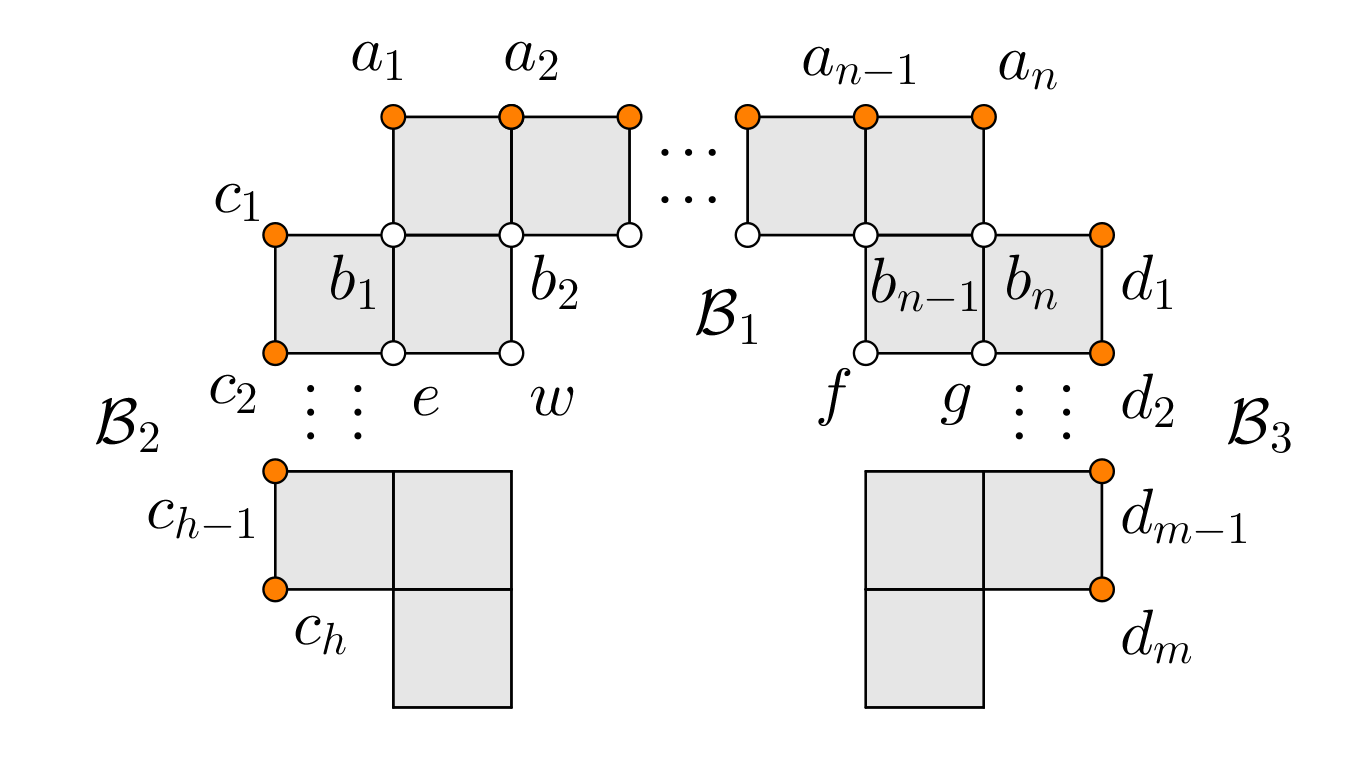}
		\caption{Arrangements of $\cB_1$, $\cB_2$ and $\cB_3$.}
		\label{Figure: Proof for pure simplicial 1}
	\end{figure}
	
	\noindent Observe that we cannot replace $a_1$ in $F$ with $b_i$ for $i\in\{2,\dots,n-1\}$, otherwise $\{b_i,a_{i+1}\}\in \Delta(\cP)$, a contradiction. A similar conclusion, that is $\{c_2,b_1\}\in \Delta(\cP)$, arises if we replace $a_1$ with $b_1$ in $F$. Moreover, note that if $a_1\in G$, then we cannot replace any point in $\cL\setminus\{a_1\}$ with another one, otherwise we get a contradiction similar to the previous one. The only possibilities for $G$ are described in the following:
	\begin{enumerate}
		\item $G=F\setminus\{a_1\}\cup \{b_n\}$;
		\item $G\in \{H_i:i\in \{2,\dots,n\}\}$, where $$H_i=F\setminus \big(\{a_1\}\cup \{a_k:k\in\{i+1,\dots,n\}\}\big) \cup\big(\{b_n\}\cup \{b_k:k\in\{i,\dots,n-1\}\}\big);$$ 
		\item $G\in \{K_i:i\in \{2,\dots,n\}\}$, where $K_i=H_i\setminus \{b_n,d_1\}\cup \{f,g\};$
		\item $G=F\setminus \{a_1,c_1\}\cup \{b_n,w\}$;
		\item $G\in \{P_i:i\in \{2,\dots,n\}\}$, where $$P_i=F\setminus \big(\{a_1,c_1\}\cup \{a_k:k\in\{i+1,\dots,n\}\}\big) \cup\big(\{w,b_n\}\cup \{b_k:k\in\{i,\dots,n-1\}\}\big);$$ 
		\item $G\in \{Q_i:i\in \{2,\dots,n\}\}$, where $Q_i=P_i\setminus \{b_n,d_1\}\cup \{f,g\}.$
	\end{enumerate} 
	
	\noindent In each of the presented cases, we have $\vert G\vert =\vert F\vert$. In general, we can extend the previous arguments to each part of $\cP$, and we always obtain that $\vert G\vert =\vert F\vert$. Therefore, there does not exist any facet $G$ in $\Delta(\cP)$ such that $\vert G\vert > \mathrm{rank}(\cP)$ or $\vert G\vert <\mathrm{rank}(\cP)$. Hence, all facets in $\Delta(\cP)$ have the same dimension, so $\Delta(\cP)$ is pure. Moreover, $\dim\Delta(\cP)=\mathrm{rank}(\cP)-1$, so $\dim K[\cP]=\mathrm{rank}(\cP)$. From Lemma \ref*{Lemma: Closed path number of vertices and cells}, we obtain that $\dim K[\cP]~=\vert V(\cP)\vert-\mathrm{rank}(\cP)$.	
	\end{proof}

	\begin{coro}\label{Coro: height of P}
	Let $\cP$ be a closed path polyomino. Then $\mathrm{ht}(I_{\cP})=\rank(\cP)$.
	\end{coro}

	\begin{proof}
	It follows from Theorem~\ref{Thm: Dimension closed path} and \cite[Corollary 3.1.7]{Villareal}.
	\end{proof}
	
	\noindent It is known that if $\cP$ is a simple polyomino, then $\dim(K[\cP])=|V(\cP)|-\rank(\cP)$, so $\mathrm{ht}(I_{\cP})=\rank(\cP)$. The following conjecture arises naturally.
	
	\begin{conj}
		Let $\cP$ be a non-simple polyomino. Then $\mathrm{ht}(I_{\cP})=\rank(\cP)$.
	\end{conj}

	\section{K\H{o}nig type property of closed path polyominoes}\label{Konig}
	
	\noindent Let $R=K[x_1,\dots,x_n]$ and $I$ be a graded ideal in $R$ of height $h$. In according to \cite{Def. Konig type}, we say that $I$ is of \textit{K\H{o}nig type} if there exist a sequence $f_1,\dots, f_h$ of homogeneous polynomials forming part of a minimal system of homogeneous generators of $I$ and a monomial order $<$ on $R$ such that $\lt_<(f_1),\dots,\lt_<(f_h)$ is a regular sequence.

	\noindent If $\cP$ is a polyomino and $I_{\cP}$ is its polyomino ideal, then we say that $\cP$ is of \textit{K\H{o}nig type} if $I_{\cP}$ is an ideal of K\H{o}nig type. In \cite{Hibi - Herzog Konig type polyomino}, it is shown by Herzog and Hibi that the polyomino ideal of a simple thin polyomino is of K\H{o}nig type. However, the polyomino ideal of a thin polyomino is not always of K\H{o}nig type, as shown in the following remark. 

    \begin{rmk}\label{grid_poly}\rm 
Firstly, we provide the following necessary condition in order to $I_{\cP}$ is of K\"onig type. If $\cP$ is a polyomino such that $\mathrm{ht}(I_{\cP})=\rank(\cP)$ and $I_{\cP}$ is of K\H{o}nig type with respect to a monomial order $<$, then $|V(\cP)|\geq 2\rank(\cP)$. In fact, assuming that $\cP$ has $n$ distinct cells, then there exist $n$ generators $f_1, \dots, f_n$ of $I_{\cP}$ such that $\lt_<(f_1), \dots, \lt_<(f_n)$ is a regular sequence. Suppose by contradiction that $|V(\cP)|<2\rank(P)$. Since $|V(\cP)|<2n$ and the initial monomial $\lt_<(f_i)$ is a product of two variables for all $i\in[n]$, there exist $j,k\in[n]$ such that $\lt_<(f_j)$ and $\lt_<(f_k)$ do not have disjoint supports. Hence $\lt_<(f_1), \dots, \lt_<(f_n)$ cannot form a regular sequence, so we get a contradiction and our claim holds. Now, consider the polyomino $\cP_1$ from Figure~ \ref{Figure: example of non Konig type} (A). Using the Algebra software \texttt{Macaulay2} (\cite{Package_M2, macaulay2}) we compute the height of $I_{\cP_1}$, and we obtain that $\mathrm{ht}(I_{\cP_1})=13=\rank(\cP_1)$, as expected. Note that $\vert V(\cP_1)\vert= 24$, so $\vert V(\cP_1)\vert < 2\rank(\cP_1)$. Hence $I_{\cP_1}$ cannot be of K\H{o}nig type. 
\end{rmk}

Closed path polyominoes lie in the class of thin polyominoes but they are more complicated
because they are non-simple. The aim of this section is to show that this class of polyominoes satisfies the K\H{o}nig type property. Motivated by the following remark, we need to define a suitable monomial order to reach this goal.

\begin{rmk}\label{grid_poly_2}\rm
 Consider the closed path $\cP_2$ in Figure \ref{Figure: example of non Konig type} (B) and the lexicographic order $<_{\mathrm{lex}}^Y$ defined in Definition \ref{Defn:order}, where $Y=\{d <^1 c <^1 b <^1 a\}$. In Figure \ref{Figure: example of non Konig type} (B) we figure out the simple graph $G$ where $V(G)=V(\cP_2)$ and $\{u,v\}\in E(G)$ if $x_ux_v\in \mathrm{in}_{<_{\mathrm{lex}}^Y}(I_{\cP_2})$. Observe that $p$ and $q$ are isolated vertices in $G$, that is, $\{p,v\},\{q,v\}\notin E(G)$ for all $v\in V(G)$. The order $<_{\mathrm{lex}}^Y$ is not sufficient to say that $I_{\cP_2}$ is of K\H{o}nig type. In fact, since there does not exist any generator $f$ of $I_{\cP_2}$ such that $x_p$ or $x_q$ divides $\mathrm{in}_{<_{\mathrm{lex}}^Y}(f)$ and, moreover, $\mathrm{ht}(I_{\cP_2})=16=|V(\cP_2)|/2$ due to Lemma \ref{Lemma: Closed path number of vertices and cells} and Corollary \ref{Coro: height of P}, then we cannot find $16$ generators of $I_{\cP_2}$ whose initial terms form a regular sequence, so $I_{\cP_2}$ is not of K\H{o}nig type with respect to $<_{\mathrm{lex}}^Y$.
 
  	\begin{figure}[h!]
  		\centering
  		\subfloat[]{\includegraphics[scale=0.6]{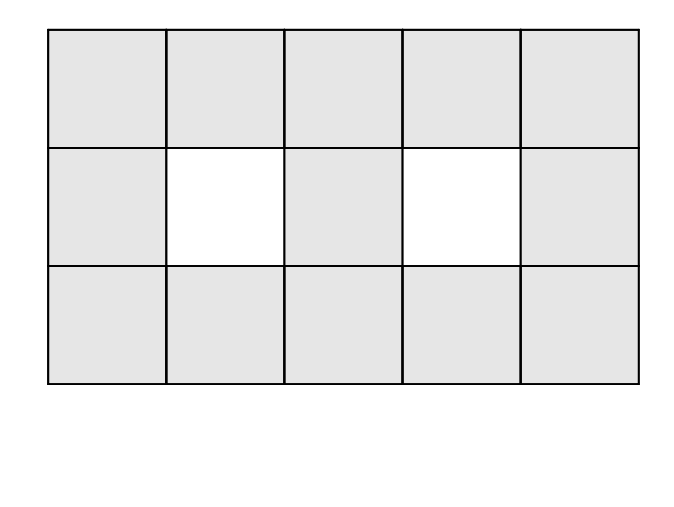}}\qquad\qquad
        \subfloat[]{\includegraphics[scale=0.6]{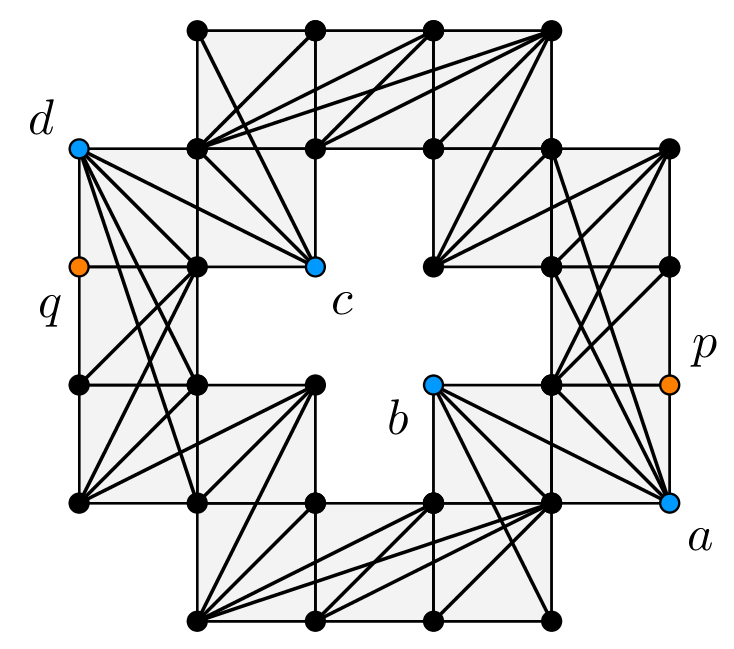}}
  		\caption{Two non-simple and thin polyominoes.}
  		\label{Figure: example of non Konig type}
  	\end{figure}

    \end{rmk}

	\begin{rmk}\rm \label{Remark: For the Konig Type}
	Let $\cP$ be a closed path polyomino having $n$ distinct cells. Let $<_{\mathrm{lex}}$ be the lexicographic order induced by a total order on $\{x_v:v\in V(\cP)\}$. Suppose that there exist $n$ generators $f_1,\dots,f_n$ of $I_{\cP}$ whose initial terms do not have any variable in common. Then, from Corollary~\ref{Coro: height of P}, we know that $\mathrm{ht}(I_\cP)=n$. Moreover, $\gcd(\lt_{<_{\mathrm{lex}}}(f_i),\lt_{<_{\mathrm{lex}}}(f_j))=1$ for all $i\neq j$, so $\lt_{<_{\mathrm{lex}}}(f_1),\dots,\lt_{<_{\mathrm{lex}}}(f_n)$ forms a regular sequence. Hence, $I_{\cP}$ is of K\H{o}nig type.
	\end{rmk}

\begin{table}[h]
    \centering
    \begin{minipage}{0.45\textwidth}
        \centering
        \renewcommand\arraystretch{0.8}
        \begin{tabular}{c|c|c}
            & \textbf{IF} it occurs & \textbf{THEN} we refer to \\
            \hline
            I & \begin{minipage}{0.36\textwidth}
                \includegraphics[scale=0.52]{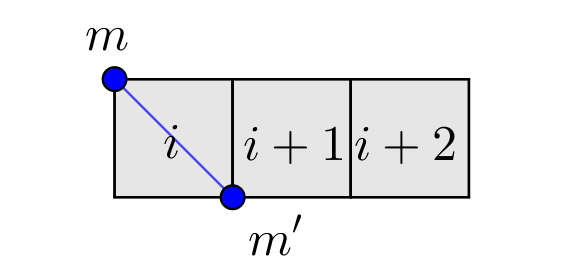}
            \end{minipage} & \begin{minipage}{0.36\textwidth}
                \includegraphics[scale=0.52]{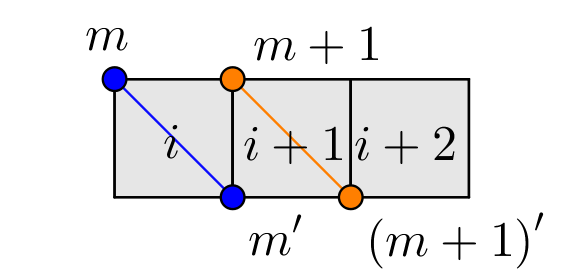}
            \end{minipage} \\
            \hline
            II & \begin{minipage}{0.36\textwidth}
                \includegraphics[scale=0.52]{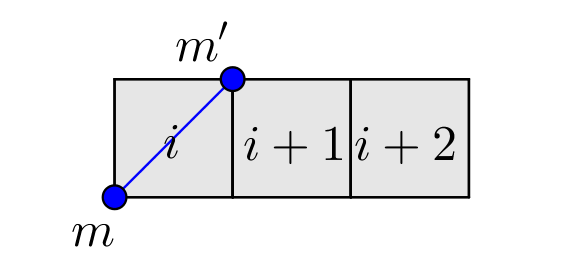}
            \end{minipage} & \begin{minipage}{0.36\textwidth}
                \includegraphics[scale=0.52]{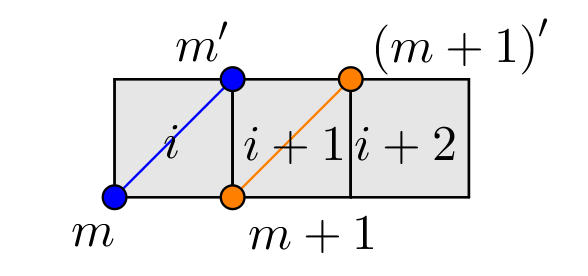}
            \end{minipage} \\
            \hline
            III & \begin{minipage}{0.36\textwidth}
                \includegraphics[scale=0.52]{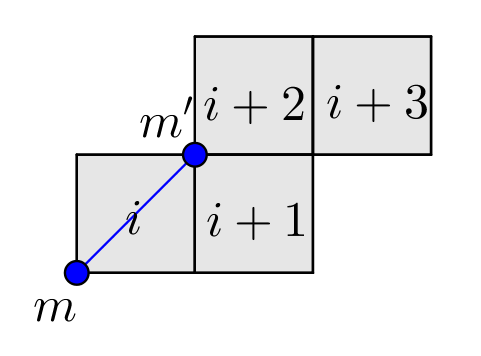}
            \end{minipage} & \begin{minipage}{0.36\textwidth}
                \includegraphics[scale=0.52]{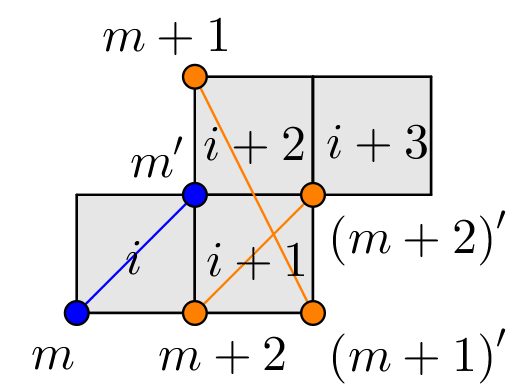}
            \end{minipage} \\
            \hline
            IV & \begin{minipage}{0.36\textwidth}
                \includegraphics[scale=0.52]{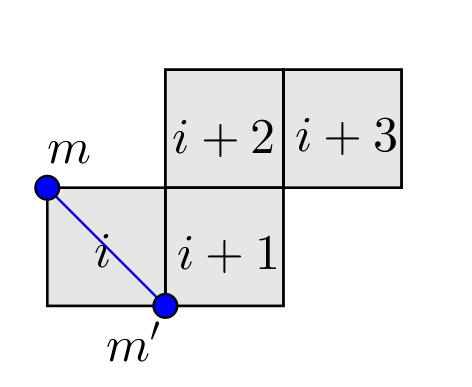}
            \end{minipage} & \begin{minipage}{0.36\textwidth}
                \includegraphics[scale=0.52]{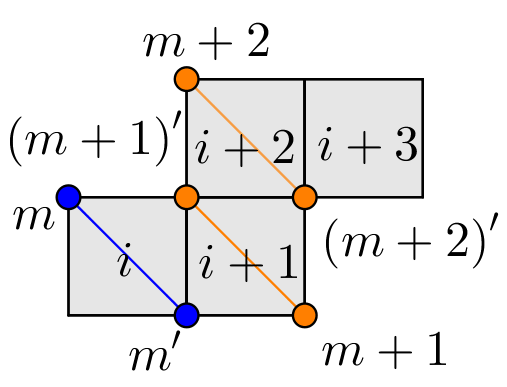}
            \end{minipage} \\
            \hline
            V & \begin{minipage}{0.36\textwidth}
                \includegraphics[scale=0.52]{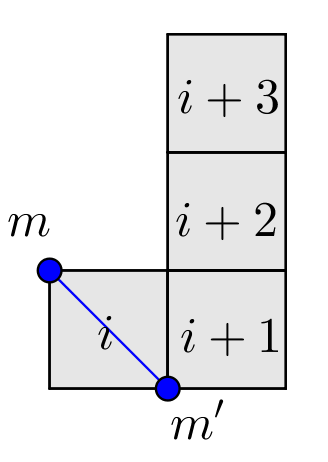}
            \end{minipage} & \begin{minipage}{0.36\textwidth}
                \includegraphics[scale=0.52]{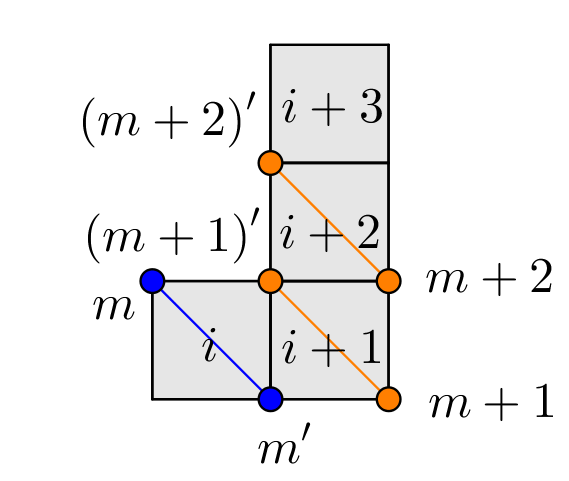}
            \end{minipage} \\
            \hline
            VI & \begin{minipage}{0.36\textwidth}
                \includegraphics[scale=0.52]{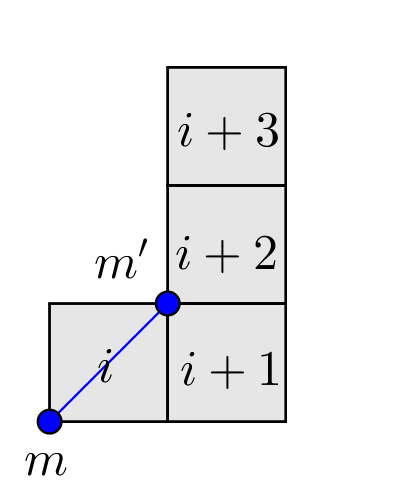}
            \end{minipage} & \begin{minipage}{0.36\textwidth}
                \includegraphics[scale=0.52]{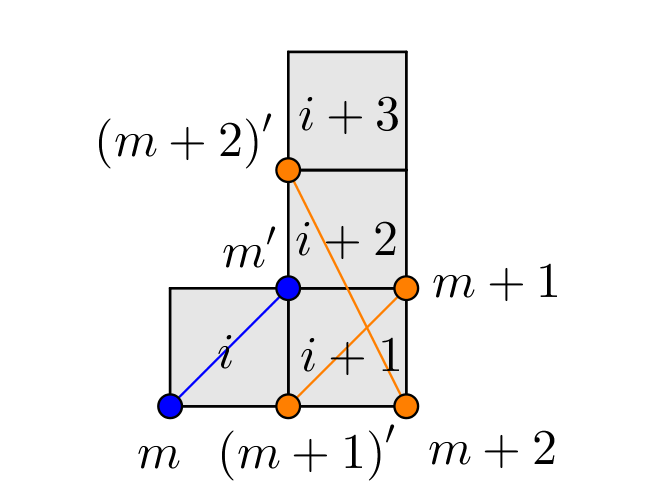}
            \end{minipage} 
        \end{tabular}
    \end{minipage}
    \hfill
    \begin{minipage}{0.48\textwidth}
        \centering
        \renewcommand\arraystretch{0.8}
        \begin{tabular}{c|c|c}
            & \textbf{IF} it occurs & \textbf{THEN} we refer to \\
            \hline
            VII & \begin{minipage}{0.34\textwidth}
                \includegraphics[scale=0.52]{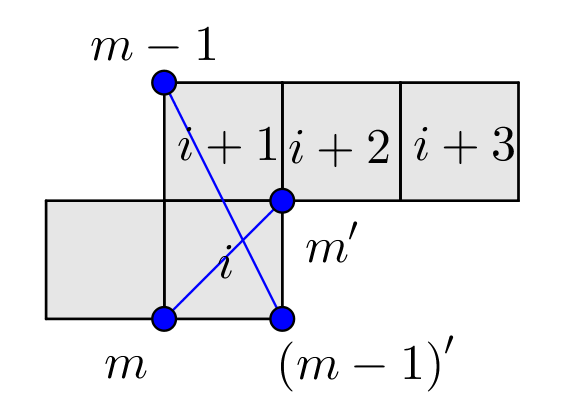}
            \end{minipage} & \begin{minipage}{0.36\textwidth}
                \includegraphics[scale=0.52]{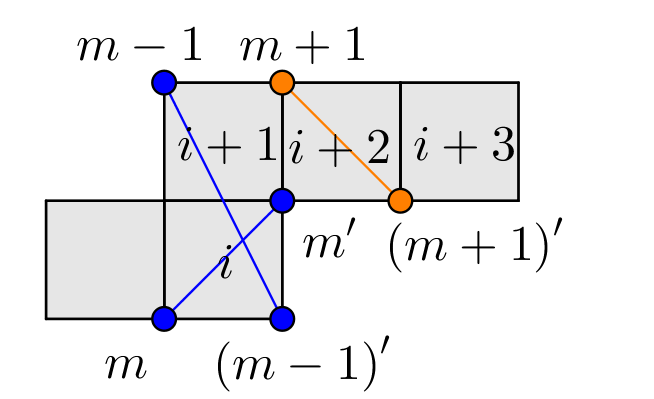}
            \end{minipage} \\
            \hline
            VIII & \begin{minipage}{0.36\textwidth}
                \includegraphics[scale=0.52]{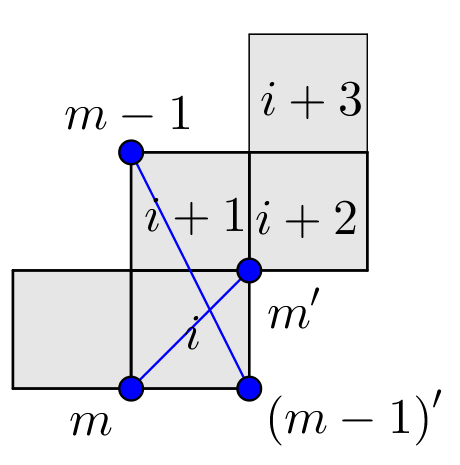}
            \end{minipage} & \begin{minipage}{0.36\textwidth}
                \includegraphics[scale=0.52]{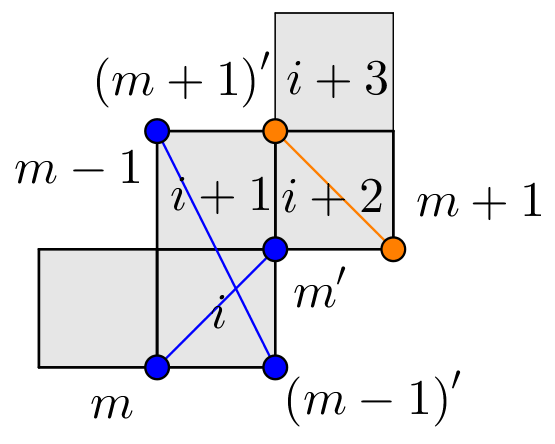}
            \end{minipage} \\
            \hline
            IX & \begin{minipage}{0.36\textwidth}
                \includegraphics[scale=0.52]{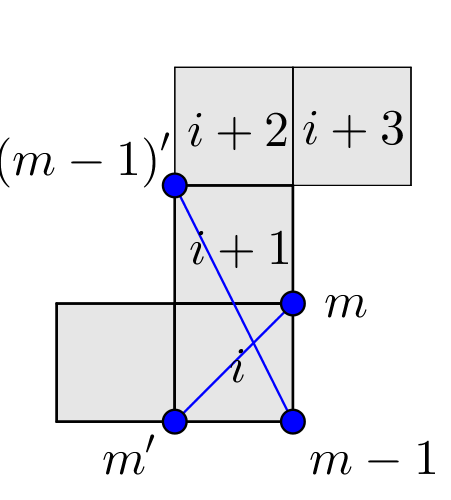}
            \end{minipage} & \begin{minipage}{0.36\textwidth}
                \includegraphics[scale=0.52]{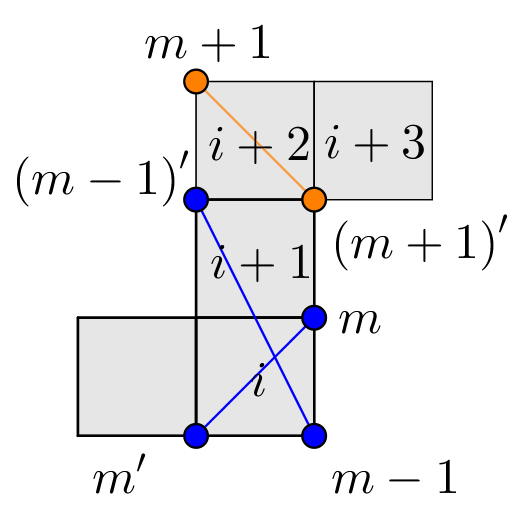}
            \end{minipage} \\
            \hline
            X & \begin{minipage}{0.20\textwidth}
                \includegraphics[scale=0.52]{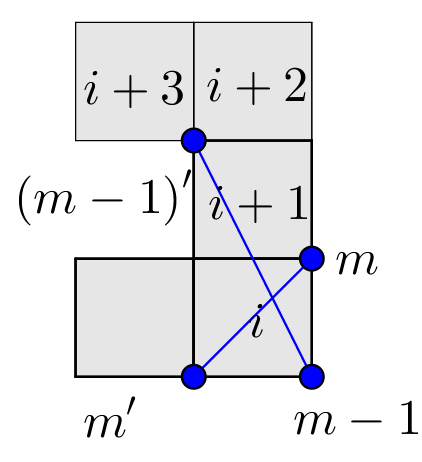}
            \end{minipage} & \begin{minipage}{0.36\textwidth}
                \includegraphics[scale=0.52]{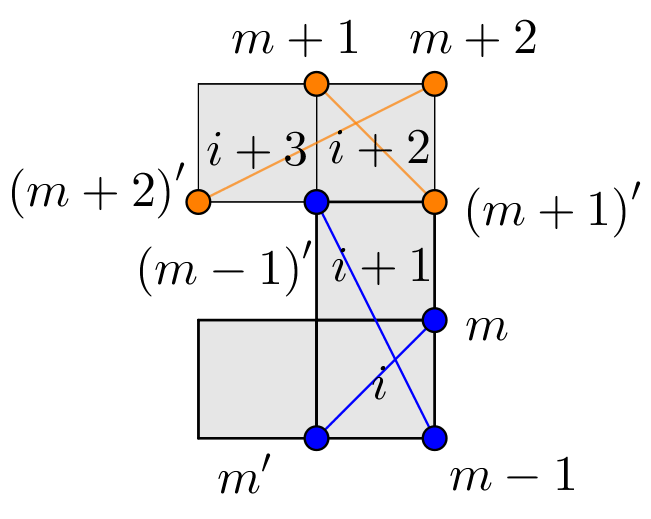}
            \end{minipage} \\
            \hline
        \end{tabular}
    \end{minipage}
    \caption{}
    \label{Table2}
\end{table}

\begin{defn}\rm \label{Procedure: to define Y}
 	Let $\cP:A_1,\dots,A_n$ be a closed path polyomino. In order to define a suitable total order on $\{x_v:v\in V(\cP)\}$, Table~\ref{Table2} will be very useful. Let $Y^{(1)}\subset V(\cP)$. Let $j\geq 2$ and assume that $Y^{(j-1)}$ is known. We want to define $Y^{(j)}$. We refer to Table~\ref{Table2} up to suitable rotations and reflections of $\cP$. If one of the configurations in the left column of Table~\ref{Table2} occurs, where the blue vertices are in $Y^{(j-1)}$, then we denote by $k$ the maximum integer such that $m+k$ is an orange vertex in the picture displayed in the corresponding right column. Hence, we set $Z^{(j)}_1=\{x_{m},\dots,x_{m+k}\}$ and $Z^{(j)}_2=\{x_{m'},\dots,x_{(m+k)'}\}$, where for all $a\in Y^{(j-1)}_1$ we put $x_a>x_{h_1}>x_{h_2}>x_{1'}$ with $m<h_1<h_2\leq m+k$ and for all $b\in Y^{(j-1)}_2$ we put $x_b>x_{t_1}>x_{t_2}$ with $m'<t_1<t_2\leq (m+k)'$. Therefore, we define $Y^{(j)}=Y^{(j)}_1\sqcup Y^{(j)}_2$ where $Y^{(j)}_1=Y^{(j-1)}_1\sqcup Z^{(j)}_1$ and $Y^{(j)}_2=Y^{(j-1)}_2\sqcup Z^{(j)}_2$.
 \end{defn}
	
	\noindent We need to distinguish just two cases depending on the changes of direction in $\cP$, so we have the following two results.

	\begin{prop}\label{Lemma: A closed path with a tetromino is Konig type}
	Let $\cP:A_1,\dots,A_n$ be a closed path polyomino. Suppose that $\cP$ contains a configuration of four cells as in Figure \ref*{Figure: particular tetromino} (A), up to reflections or rotations of $\cP$ or up to relabeling of the cells of $\cP$. Then $I_{\cP}$ is of K\H{o}nig type.
	\end{prop}
	
	\begin{proof}
	We distinguish two cases depending on the position of $A_3$ with respect to $A_2$.
    
     \textsl{Case I:} We assume that $A_3$ is at North of $A_2$. 
        
	\begin{figure}[h]
		\centering
		\subfloat[]{\includegraphics[scale=0.7]{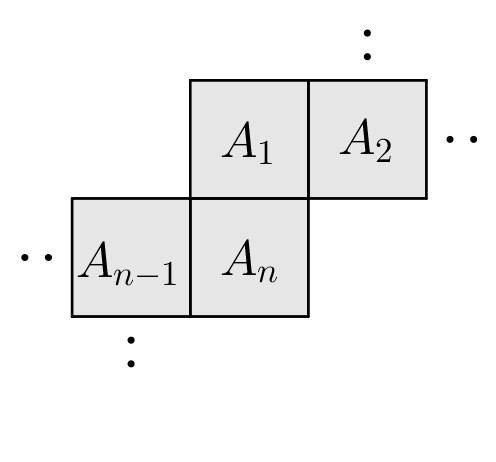}}\qquad
		\subfloat[]{\includegraphics[scale=0.7]{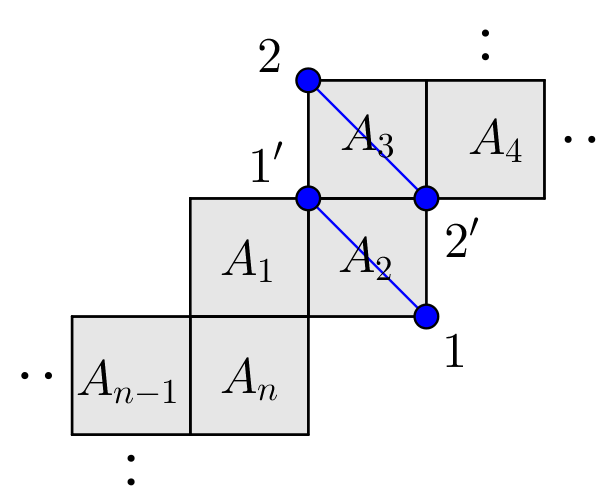}}
		\subfloat[]{\includegraphics[scale=0.7]{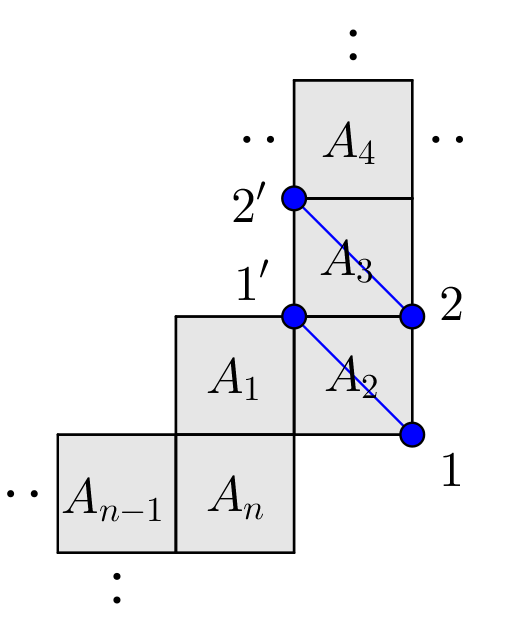}}
		\caption{Arrangements of $A_{n-1}$, $A_n$, $A_1$, $A_2$, $A_3$, $A_4$ in \textsl{Case I} of Proposition \ref{Lemma: A closed path with a tetromino is Konig type}.}
		\label{Figure: particular tetromino}
	\end{figure}

\noindent We set $Y^{(1)}=Y^{(1)}_1 \sqcup Y^{(1)}_2$ where $Y^{(1)}_1=\{x_1,x_2\}$ and $Y^{(1)}_2=\{x_{1'},x_{2'}\}$ with $x_1>x_2>x_{1'}>x_{2'}$, with reference to Figure~\ref{Figure: particular tetromino} (B) if $A_4$ is at East of $A_3$ or to Figure~ \ref{Figure: particular tetromino} (C) if $A_4$ is at North of $A_3$. Starting with this position for $Y^{(1)}$, we apply the procedure described in Definition~\ref{Procedure: to define Y}. 
 
 Since $\cP$ has a finite number of cells and stops in the configuration $\{A_{n-1}, A_n, A_1, A_2\}$, the previous procedure consists of a finite number of steps, let us say $p$ steps. In particular, in Figure~\ref{Figure: tetromino in the last part CASE 1}, we summarize all cases which may appear in the last step, where the blue vertices represent the points which are in $Y^{(p-1)}$ in the penultimate step. Let $Y=Y^{(p)}$ be the order set of variables obtained by using the previous arguments and let $|Y|=2r$ with $r\in \NN$. We have $x_1>x_2>\dots>x_r>x_{1'}>x_{2'}>\dots>x_{r'}$ and we set $Y_1=\{x_1,x_2,\dots,x_r\}$ and $Y_2=\{x_{1'},x_{2'},\dots,x_{r'}\}$. Moreover, observe that all vertices of $\cP$ are covered two by two, so $r=n$ by Lemma~\ref{Lemma: Closed path number of vertices and cells}. Hence, we obtain $n$ generators of $I_{\cP}$ whose initial terms do not have any variables in common, hence, by Remark~\ref*{Remark: For the Konig Type}, it follows that $I_{\cP}$ is of K\H{o}nig type.
	
	\begin{figure}[h]
		\centering
		\subfloat[$a\in Y_2$]{\includegraphics[scale=0.75]{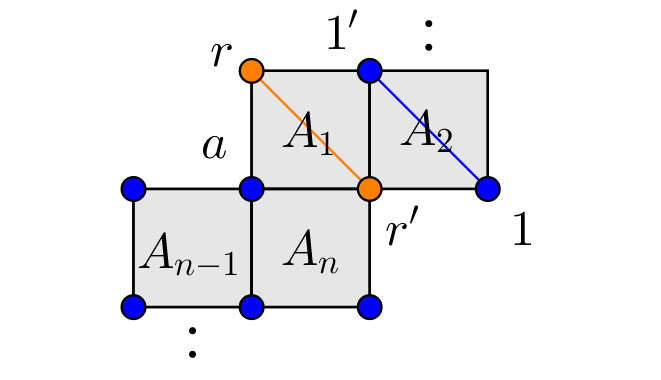}}
		\subfloat[$a\in Y_2$]{\includegraphics[scale=0.75]{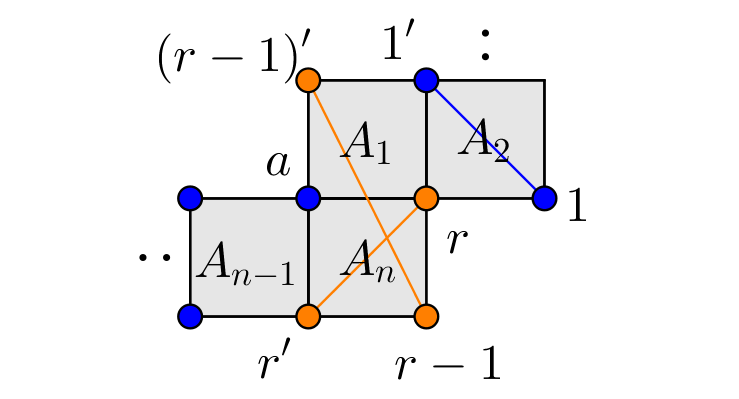}}
		\caption{The last step in \textsl{Case I} of Proposition \ref{Lemma: A closed path with a tetromino is Konig type}.}
		\label{Figure: tetromino in the last part CASE 1}
	\end{figure}

     \textsl{Case II:} We assume that $A_3$ is at East of $A_2$. We set $Y^{(1)}=Y^{(1)}_1 \sqcup Y^{(1)}_2$ where $Y^{(1)}_1=\{x_1\}$ and $Y^{(1)}_2=\{x_{1'}\}$ with $x_1>x_{1'}$, with reference to Figure~ \ref{Figure:tetromino A_3 is at East of A_2} (A). As done before, we start with this position for $Y^{(1)}$ and we apply the procedure described in Definition~\ref{Procedure: to define Y}. Let $q$ be the number of the steps until $\{A_{n-1},A_n,A_1,A_2\}$. In Figure~\ref{Figure:tetromino A_3 is at East of A_2} (A), (B) and (C) we show all cases in the last step and we point out that we set $x_0>x_1$. Hence, with the same arguments as before we get the desired conclusion.
	
	\begin{figure}[h]
		\centering
		\subfloat[]{\includegraphics[scale=0.75]{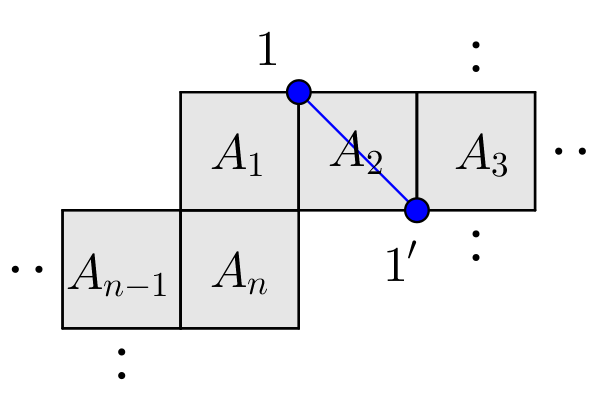}}
		\subfloat[]{\includegraphics[scale=0.75]{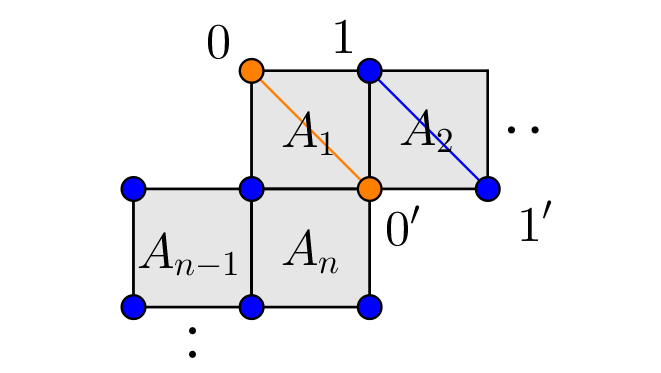}}
		\subfloat[$a\in Y_2$]{\includegraphics[scale=0.75]{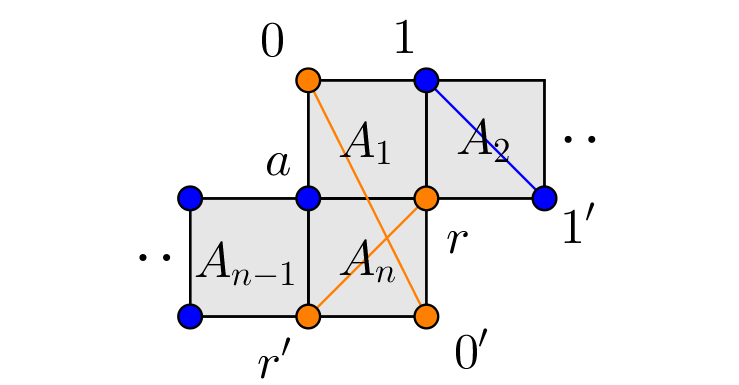}}
		\subfloat[$b\in Y_2$]{\includegraphics[scale=0.75]{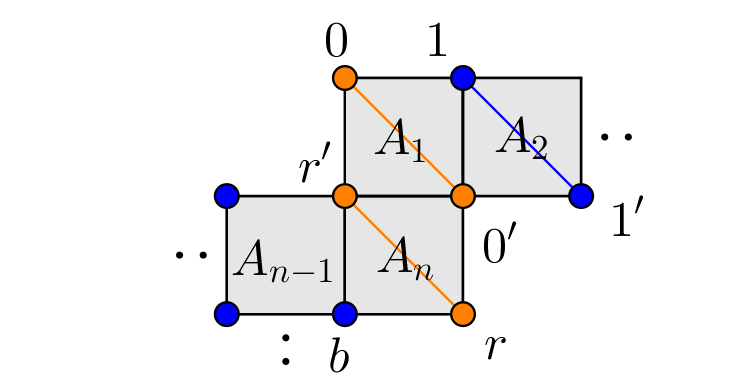}}
		\caption{Arrangements of $A_{n-1}$, $A_n$, $A_1$, $A_2$, $A_3$ in \textsl{Case II} of Proposition \ref{Lemma: A closed path with a tetromino is Konig type}.}
		\label{Figure:tetromino A_3 is at East of A_2}
	\end{figure}

	\end{proof}

	\begin{exa}\rm \label{exa: konig procedure}
	An example of the procedure described in Lemma~\ref{Lemma: A closed path with a tetromino is Konig type} can be seen in Figure~\ref{Figure: example closed path of konig type 1}. In particular, $\cP$ is of K\H{o}nig type with respect to the lexicographic order induced by
	$$ x_1>x_2>\dots>x_{30}>x_{1'}>x_{2'}>\dots>x_{30'}$$
	and to the thirty generators of $I_{\cP}$ corresponding to the inner intervals having $i$ and $i^{\prime}$ as diagonal or anti-diagonal corners, for all $i\in [30]$.
	\begin{figure}[h]
		\centering
		\includegraphics[scale=0.85]{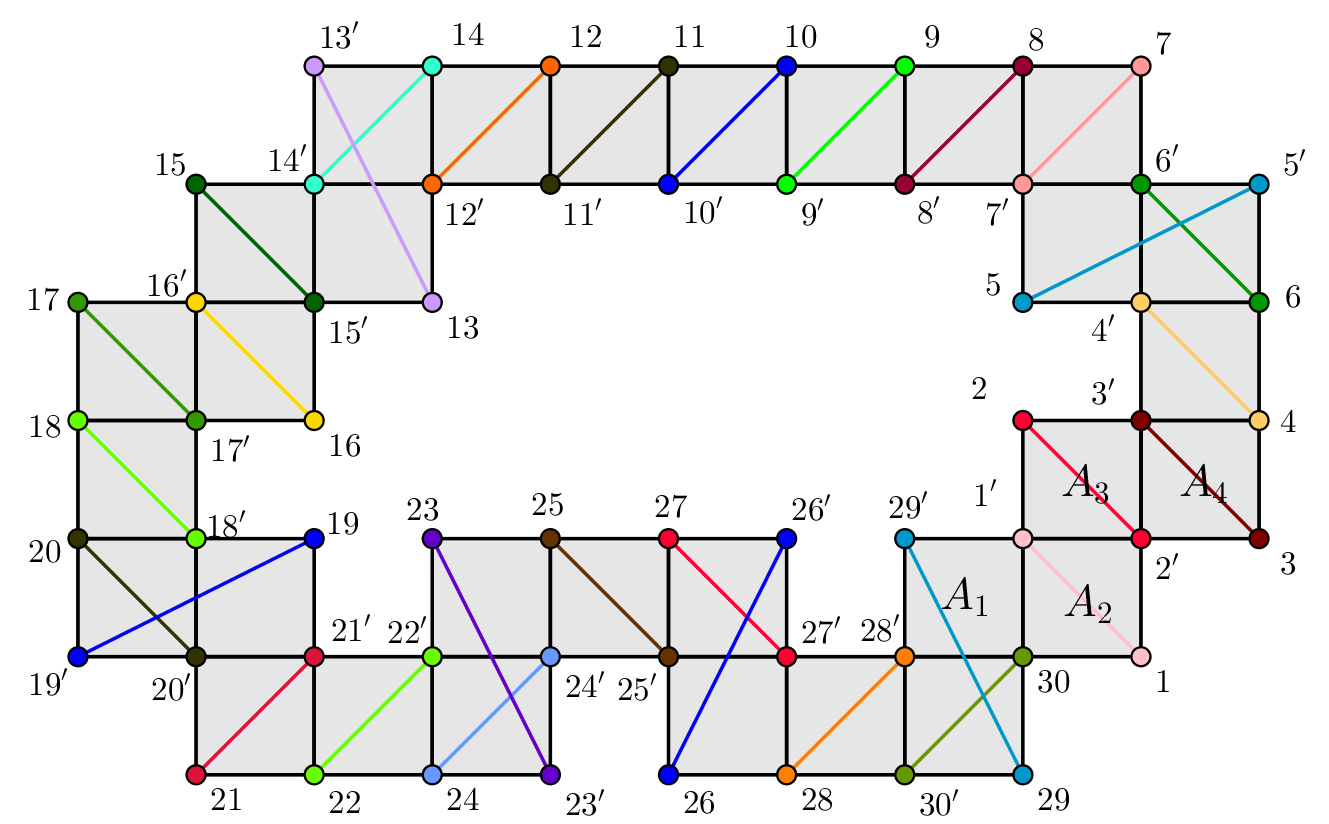}
		\caption{A closed path polyomino and the related total order of the vertices.}
		\label{Figure: example closed path of konig type 1}
	\end{figure}

    \noindent We describe the algorithm given in Proposition~\ref{Lemma: A closed path with a tetromino is Konig type} and we show how it works step by step, with reference to Figure \ref{Figure: example closed path of konig type 1}.\\
    \noindent \textsl{Step 1.} Starting from the tetromino $\{A_1,A_2,A_3,A_4\}$, observe that $A_4$ is at East of $A_3$ so we are in the case of Figure \ref{Figure: particular tetromino} (B). Hence we set $Y^{(1)}=Y^{(1)}_1\sqcup Y^{(1)}_2$ with $Y^{(1)}_1=\{x_1,x_2\}$ and $Y^{(1)}_2=\{x_{1'},x_{2'}\}$, where $x_1>x_2>x_{1'}>x_{2'}$.\\
    \noindent \textsl{Step 2.} Consider now the tetromino $\{A_3,A_4,A_5,A_6\}$, so it occurs the case (V) of Table \ref{Table2}. Hence we set $Y^{(2)}=Y^{(2)}_1\sqcup Y^{(2)}_2$ with $Y^{(2)}_1=Y^{(1)}_1\sqcup Z^{(2)}_1$ where $Z^{(2)}_1=\{x_3,x_4\}$ and $Y^{(2)}_2=Y^{(1)}_2\sqcup Z^{(2)}_2$ where $Z^{(2)}_2=\{x_{3'},x_{4'}\}$, where $x_1>x_2>x_3>x_4>x_{1'}>x_{2'}>x_{3'}>x_{4'}$.\\
    \noindent \textsl{Step 3.} Focusing on $\{A_5,A_6,A_7,A_8\}$, we are in the case (III) of Table \ref{Table2} after suitable rotations of $\cP$. Hence $Z^{(3)}_1=\{x_5,x_6\}$ and $Z^{(2)}_2=\{x_{5'},x_{6'}\}$, where $x_1>\dots>x_4>x_5>x_6>x_{1'}>\dots>x_{4'}>x_{5'}>x_{6'}$.\\
    \noindent \textsl{Step 4.} Take the trimino $\{A_8,A_9,A_{10}\}$, so we have the case (I) of Table \ref{Table2} after a reflection of $\cP$ with respect to $y$-axis. Hence $Z^{(5)}_1=\{x_8\}$ and $Z^{(2)}_2=\{x_{8'}\}$, with $x_1>\dots>x_7>x_8>x_{1'}>\dots>x_{7'}>x_{8'}$.\\
    \noindent \textsl{Steps 5-9.} We can argue as done in the previous step so we obtain $x_1>\dots>x_{11}>x_{12}>x_{1'}>\dots>x_{11'}>x_{12'}$. From this point, it should be clear to the reader how we continue. \\
    \noindent \textsl{Step 10.} Consider $\{A_{13},A_{14},A_{15},A_{16}\}$, so we are in the case (III) of Table \ref{Table2} after suitable rotations of $\cP$. Hence $Z^{(10)}_1=\{x_{13},x_{14}\}$ and $Z^{(10)}_2=\{x_{13'},x_{14'}\}$.\\
    \noindent \textsl{Step 11.} Let $\{A_{14},A_{15},A_{16},A_{17}\}$, so we are in the case (VIII) of Table \ref{Table2} after suitable rotations of $\cP$. Hence $Z^{(11)}_1=\{x_{15}\}$ and $Z^{(11)}_2=\{x_{15'}\}$.\\
    \noindent \textsl{Step 12.} Considering $\{A_{16},A_{17},A_{18},A_{19}\}$, we get the case (IV) of Table \ref{Table2} up to rotations of $\cP$. Hence $Z^{(12)}_1=\{x_{16},x_{17}\}$ and $Z^{(12)}_2=\{x_{16'},x_{17'}\}$.\\
    \noindent \textsl{Step 13.} Get the trimino $\{A_{18},A_{19},A_{20}\}$, so we are in the case (II) of Table \ref{Table2} after suitable rotation of $\cP$. Hence $Z^{(13)}_1=\{x_{18}\}$ and $Z^{(13)}_2=\{x_{18'}\}$.\\
    \noindent \textsl{Step 14.} Take $\{A_{19},A_{20},A_{21},A_{22}\}$, so we have the case (III) of Table \ref{Table2} after a suitable rotation of $\cP$. Hence $Z^{(14)}_1=\{x_{19},x_{20}\}$ and $Z^{(14)}_2=\{x_{19'},x_{20'}\}$.\\
    \noindent \textsl{Step 15.} Focus on the tetromino $\{A_{20},A_{21},A_{22},A_{23}\}$, so we get the case (VIII) of Table \ref{Table2} after a suitable rotation of $\cP$. Hence $Z^{(15)}_1=\{x_{21}\}$ and $Z^{(15)}_2=\{x_{21'}\}$.\\
    \noindent \textsl{Step 16.} Considering the trimino $\{A_{22},A_{23},A_{24}\}$, we are in the case (II) of Table \ref{Table2} and $Z^{(16)}_1=\{x_{22}\}$ and $Z^{(16)}_2=\{x_{22'}\}$.\\
    \noindent \textsl{Step 17.} Get $\{A_{23},A_{24},A_{25},A_{26}\}$, so we are in the case (III) of Table \ref{Table2}. Hence $Z^{(17)}_1=\{x_{23},x_{24}\}$ and $Z^{(17)}_2=\{x_{23'},x_{24'}\}$.\\
    \noindent \textsl{Step 18.} Consider $\{A_{24},A_{25},A_{26},A_{27}\}$, we are in the case (VII) of Table \ref{Table2}. Therefore $Z^{(18)}_1=\{x_{25}\}$ and $Z^{(18)}_2=\{x_{25'}\}$.\\
    \noindent \textsl{Step 19.} Take $\{A_{26},A_{27},A_{28},A_{29}\}$, we are in the case (III) of Table \ref{Table2}, after a reflection with respect to $x$-axis. Hence $Z^{(19)}_1=\{x_{26},x_{27}\}$ and $Z^{(19)}_2=\{x_{26'},x_{27'}\}$.\\
    \noindent \textsl{Step 20.} Consider $\{A_{27},A_{28},A_{29},A_{30}\}$, so we get the case (VII) of Table \ref{Table2}, after a reflection with respect to $x$-axis. Hence $Z^{(20)}_1=\{x_{28}\}$ and $Z^{(20)}_2=\{x_{28'}\}$.\\
    \noindent \textsl{Step 21.} Consider $\{A_{29},A_{30},A_{1},A_{2}\}$, so we are in the case of Figure \ref{Figure: tetromino in the last part CASE 1} (B), or equivalently of (III) in Table \ref{Table2}. Hence $Z^{(21)}_1=\{x_{29},x_{30}\}$ and $Z^{(21)}_2=\{x_{29'},x_{30'}\}$.\\
    In conclusion, we obtain the order set of variables as
    $$ x_1>x_2>\dots>x_{30}>x_{1'}>x_{2'}>\dots>x_{30'}.$$
    \end{exa}
  
	\begin{prop}\label{Lemma: A closed path with a L-conf is Konig type}
	Let $\cP:A_1,\dots,A_n$ be a closed path polyomino. Suppose that $\cP$ has an $L$-configuration in every change of direction. Consider such an $L$-configuration as in Figure~\ref{Figure: L configuration A4 at north of A3} (A), up to relabeling of the cells of $\cP$. Then $I_{\cP}$ is of K\H{o}nig type.
	\end{prop}

	\begin{proof}
	We distinguish three cases depending on the position of $A_4$ with respect to $A_3$. First of all, we assume that $A_4$ is at North of $A_3$. We set $Y^{(1)}=Y^{(1)}_1 \sqcup Y^{(1)}_2$ where $Y^{(1)}_1=\{x_1,x_2\}$ and $Y^{(1)}_2=\{x_{1'},x_{2'}\}$ with $x_1>x_2>x_{1'}>x_{2'}$, with reference to Figure~\ref{Figure: L configuration A4 at north of A3} (A). The procedure described in Definition~\ref{Procedure: to define Y} ends with one of the two cases displayed in Figure~\ref{Figure: L configuration A4 at north of A3} (B) and (C). As done in Proposition~\ref{Lemma: A closed path with a tetromino is Konig type} the desired conclusion follows.
	
	\begin{figure}[h]
		\centering
		\subfloat[]{\includegraphics[scale=0.7]{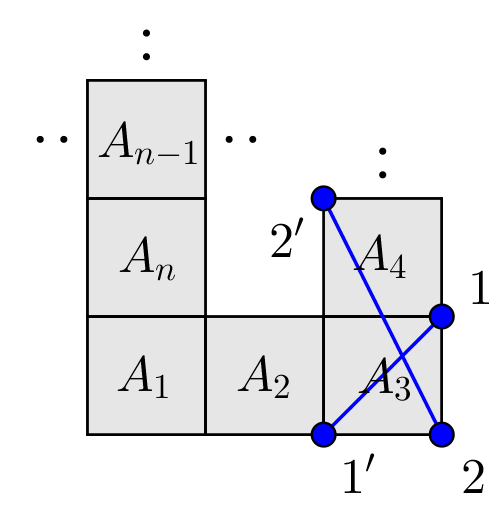}}\qquad
		\subfloat[$a \in Y_2$]{\includegraphics[scale=0.7]{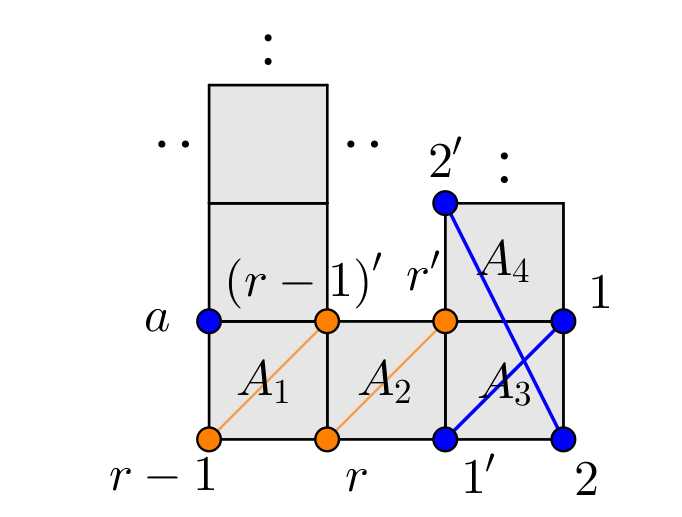}}\qquad
		\subfloat[$b\in Y_2$]{\includegraphics[scale=0.7]{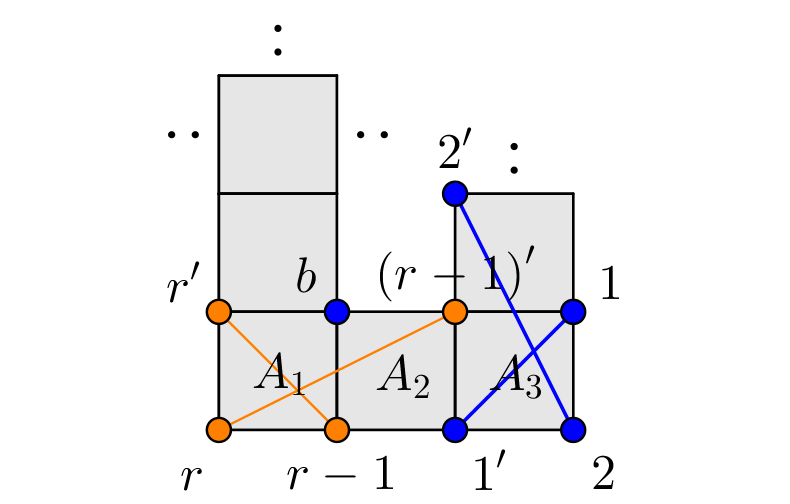}}
		\caption{Case in Proposition \ref{Lemma: A closed path with a L-conf is Konig type} when $A_4$ is at North of $A_3$.}
		\label{Figure: L configuration A4 at north of A3}
	\end{figure}

      We assume that $A_4$ is at South of $A_3$. In such a case we set $Y^{(1)}=Y^{(1)}_1 \sqcup Y^{(1)}_2$ where $Y^{(1)}_1=\{x_1\}$ and $Y^{(1)}_2=\{x_{1'}\}$ with $x_1>x_{1'}$, with reference to Figure \ref{Figure: L configuration A4 at south of A3} (A). Observe that the only two last cases are in Figure \ref{Figure: L configuration A4 at south of A3} (B) and (C).
	
	\begin{figure}[h]
		\centering
		\subfloat[]{\includegraphics[scale=0.7]{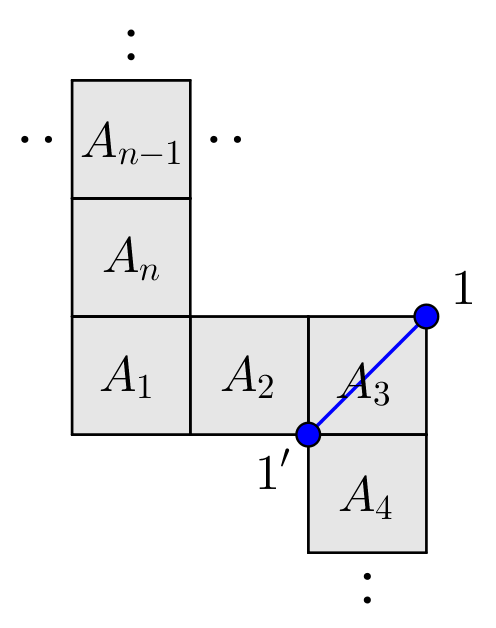}}
		\subfloat[$a\in Y_2$]{\includegraphics[scale=0.7]{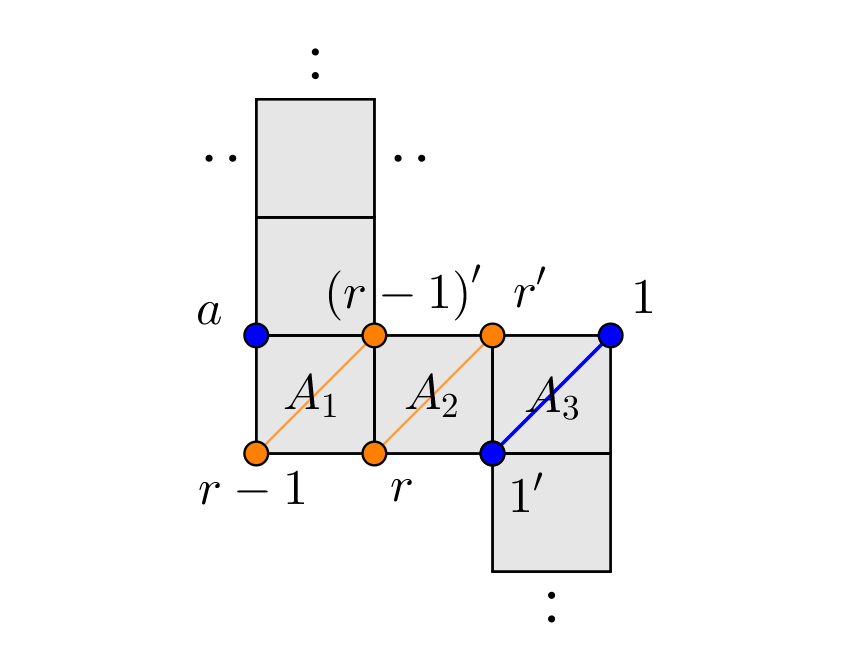}}
		\subfloat[$b\in Y_2$]{\includegraphics[scale=0.7]{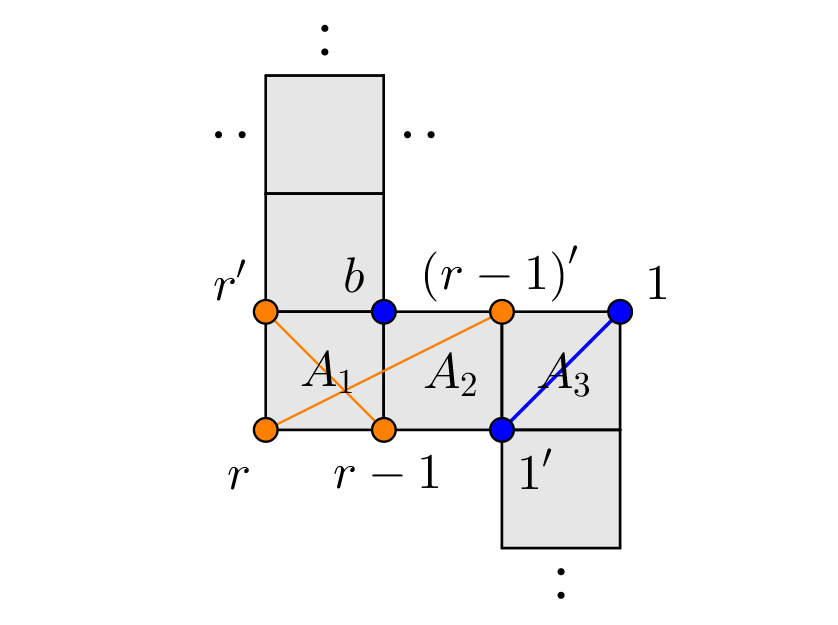}}
		\caption{Case in Proposition \ref{Lemma: A closed path with a L-conf is Konig type} when $A_4$ is at South of $A_3$.}
		\label{Figure: L configuration A4 at south of A3}
	\end{figure}

	  We assume that $A_4$ is at East of $A_3$. In such a case we set $Y^{(1)}=Y^{(1)}_1 \sqcup Y^{(1)}_2$ where $Y^{(1)}_1=\{x_3\}$ and $Y^{(1)}_2=\{x_{3'}\}$ with $x_3>x_{3'}$, with reference to Figure~\ref{Figure: L configuration A4 at east of A3} (A). Observe that the only two last cases are in Figure \ref{Figure: L configuration A4 at east of A3} (B) and (C), where we set:
	$$ x_1>x_2>x_3\dots>x_r>x_{1'}>x_{2'}>x_{3'}\dots>x_{r'}.$$
	The conclusion follows arguing as before.
	\begin{figure}[h]
		\centering
		\subfloat[]{\includegraphics[scale=0.7]{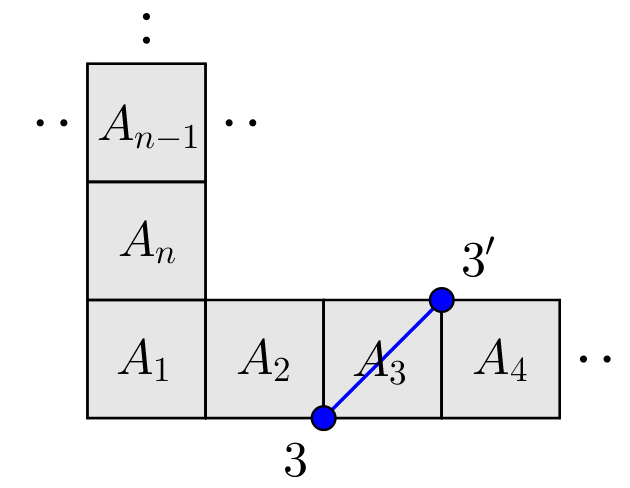}}\quad
		\subfloat[]{\includegraphics[scale=0.7]{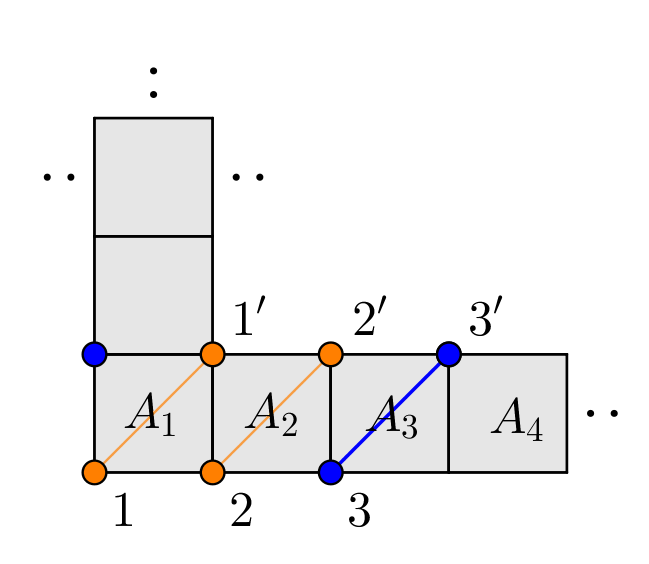}}\quad
		\subfloat[]{\includegraphics[scale=0.7]{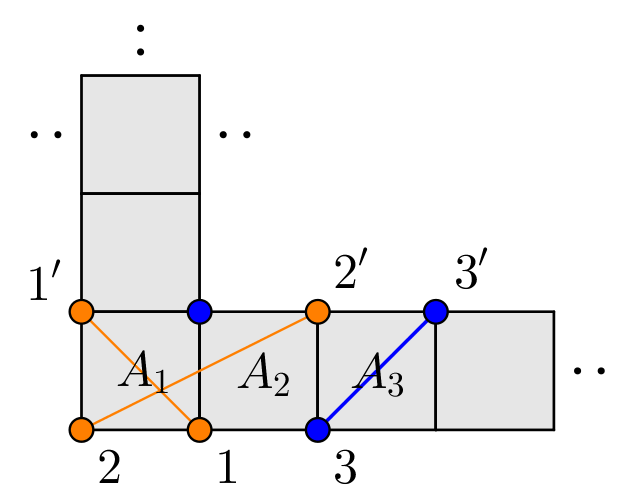}}
		\caption{Case in Proposition \ref{Lemma: A closed path with a L-conf is Konig type} when $A_4$ is at East of $A_3$.}
		\label{Figure: L configuration A4 at east of A3}
	\end{figure} 
	\end{proof}

	\begin{thm}\label{konigfinal}
	Let $\cP$ be a closed path polyomino. Then $I_{\cP}$ is of K\H{o}nig type.
	\end{thm}

	\begin{proof}
	If $\cP$ contains a configuration of four cells as in Figure~\ref*{Figure: particular tetromino} (A), then $I_{\cP}$ is of K\H{o}nig type by Proposition~\ref{Lemma: A closed path with a tetromino is Konig type}. If $\cP$ does not contain any such configuration, then it is easy to see that $\cP$ has an $L$-configuration in every change of direction, so $I_{\cP}$ is of K\H{o}nig type by Proposition~\ref{Lemma: A closed path with a L-conf is Konig type}. Hence, the desired claim is proved.
	\end{proof}

  \noindent \textbf{Question.} In general, it is not so easy to find classes of graded ideals that are of K\H{o}nig type. The results obtained in this section could initiate further open questions. For example, it is natural to ask if there are other classes of non-simple polyominoes enclosing more than one hole that are of K\H{o}nig type and if they can be characterized by using some combinatorial properties.

\section{A combinatorial interpretation of the canonical module of a circle closed path}\label{canonical}

\noindent As an application of the K\H{o}nig type property, in this section we present a sub-class of closed path polyominoes, for which the canonical module has a very nice combinatorial description; refer to \cite[Section 3.3]{Bruns_Herzog} for more details on the canonical module of a Cohen-Macaulay ring. Let us start by introducing some definitions and notations.

\begin{defn}\rm \label{circle}
    Let $I_1=[(1,1),(m,n)]$ and $I_2=[(2,2),(m-1,n-1)]$, where $m,n\geq 4$ and $(m,n)\neq (4,4)$. We say that a closed path $\cP$ is a \textit{circle} if $\cP=\cP_{I_1}\setminus \cP_{I_2}$. 
\end{defn}

\noindent We require that $m$ and $n$ are not equal to four at the same time, because in such a case $\cP$ consists of maximal blocks of length three so $K[\cP]$ is Gorenstein by \cite[Theorem 5.7]{Cisto_Navarra_Hilbert_series} and we know that the canonical module of $K[\cP]$ is isomorphic to $K[\cP]$ (see \cite[Theorem 3.3.7]{Bruns_Herzog}. \\

\noindent In the description of the canonical module of a circle closed path, the binomial ideal coming from the K\H{o}nig type property and a suitable monomial ideal will play a crucial role. We set the following.

\begin{defn}\rm \label{Defn: Ideals K and J for the canonical module}
\noindent We introduce two ideals to describe the canonical module of a circle closed path.\\
Let $\cP$ be a closed path. With reference to Proposition~\ref{Lemma: A closed path with a L-conf is Konig type}, we denote by $J(\cP)$ the binomial ideal generated by the generators of $I_{\cP}$ corresponding to the inner intervals of $\cP$ having $i$ and $i'$ as diagonal or anti-diagonal corners, for all $i\in [\mathrm{rank}(\cP)]$.\\ 
Now, let $\cP$ be a circle closed path. We introduce the following sets of vertices of $\cP$. Let $m,n>4$.

\begin{itemize}
    \item $L_{j-2}=\{(1,j),(2,j)\}$ and $R_{j-2}=\{(m-1,j),(m,j)\}$ for all $j\in [n-2]\setminus \{1,2\}$;
    \item $U_{k-2}=\{(k,n),(k,n-1)\}$ and $D_{k-2}=\{(k,1),(k,2)\}$ for all $k\in [m-2]\setminus \{1,2\}$.
\end{itemize}

\noindent In addition, we set the following Cartesian product.

\begin{enumerate}
    \item If $m,n>4$, then
    $$C=\Bigg(\bigtimes_{i=1}^{n-4} L_i \Bigg) \times \Bigg(\bigtimes_{i=1}^{n-4} R_i \Bigg)\times \Bigg(\bigtimes_{i=1}^{m-4} U_i \Bigg) \times \Bigg(\bigtimes_{i=1}^{m-4} D_i \Bigg).$$
    Set $t=2(n+m-8)$ in such a case.
    \item If $m>4$ and $n=4$, then
    $$ C=\Bigg(\bigtimes_{i=1}^{m-4} U_i \Bigg) \times \Bigg(\bigtimes_{i=1}^{m-4} D_i \Bigg).$$
    Let $t=2(m-4)$ in this case.
    \item If $m=4$ and $n>4$, then
    $$ C=\Bigg(\bigtimes_{i=1}^{n-4} L_i \Bigg) \times \Bigg(\bigtimes_{i=1}^{n-4} R_i \Bigg).$$
    Set $t=2(n-4)$ here.
\end{enumerate}

\noindent Now, we define a suitable subset of $C$. 
$$\mathcal{V}=\big\{ (v_1,\dots,v_t)\in C: v_i,v_j \text{ are not the diagonal corners of an inner interval of }\cP,  \text{ for }i,j\in[t] \big\}.$$

\noindent For all $\mathbf{v}=(v_1,\dots,v_t)\in \cV$, we define a monomial in $S_{\cP}$ as $$\mathbf{x}_{\mathbf{v}}=x_{v_1}x_{v_2}\dots x_{v_t}x_{(2,2)}x_{(1,n)}x_{(m,1)}x_{(m-1,n-1)},$$ and we denote by $K(\cP)$ the monomial ideal in $S_{\cP}$ generated by $\mathbf{x}_{\mathbf{v}}$ for all $\mathbf{v}\in \cV.$ 
\end{defn}

\begin{rmk}\rm\label{Rmk: number generators K(P)}
    Let $\cP$ be a circle closed path and $K(\cP)$ be the ideal defined in Definition \ref{Defn: Ideals K and J for the canonical module}. Observe that $K(\cP)$ is a squarefree monomial ideal generated in degree $t+4$. In addition, let us denote by $\mu(K(\cP))$ the number of monomial generators of $K(\cP)$. We have that $\mu(K(\cP))=(m-3)^2(n-3)^2$.\\
   \noindent We will assume that $m,n>4$. Consider $L=L_1\times\dots\times L_{n-4}$ and $w_0=\big((1,3),(1,4),\dots,(1,m-2)\big)\in L.$ All the other $(m-4)$-uple of $L$, where the vertices are not the diagonal corners of an inner interval of $\cP$, can be obtained replacing in $w_i$ in the $i$-th component $(i,1)$ with $(i,2)$, for all $i\in[m-4]$. Hence, the number of such configurations is $m-3$. As a consequence of the symmetric structure of $\cP$, we have that $\mu(K(\cP))=(m-3)^2(n-3)^2$. The same arguments hold in the other cases when $m=4, n>4$ and $m>4, n=4$.
\end{rmk}

We now state the main theorem of this section:
\begin{thm}\label{Thm: description canonical module}
    Let $\cP$ be a circle closed path. We denote by $\omega_{K[\cP]}$ the canonical module of $K[\cP]$. Then:
    $$\omega_{K[\cP]}\cong \bigslant{(J(\cP)+K(\cP))}{J(\cP)} $$
    where $J(\cP)$ and $K(\cP)$ are the ideals given in Definition \ref{Defn: Ideals K and J for the canonical module}.
\end{thm}

\noindent In order to prove Theorem \ref{Thm: description canonical module}, we recall two results that will be useful for our purpose.

\begin{prop}[\cite{P_S_Canonical_module}]\label{prop5.5}
  Let $I\subset S=K[x_1,\dots,x_n]$ be a graded ideal such that $S/I$ is a Cohen-Macaulay ring and $J\subset I$ be a complete intersection ideal with
$\mathrm{ht}(I) = \mathrm{ht}(J)$. Denote $\omega_{S/I}$ the canonical module of $S/I$. Then $\omega_{S/I}\cong (J : I)/J.$  
\label{Prop: canonical module from linkage theory}
\end{prop}

\begin{prop}\cite[Lemma 1.12]{Def. Konig type}\label{prop5.6}
     Let $I, J\subset S=K[x_1,\dots,x_n]$ be two ideals with $\mathrm{ht}(J) = \mathrm{ht} (I)$. Suppose that $J$ is a complete intersection and a radical ideal. Denote by $\mathcal{M}(I)$ and $\mathcal{M}(J)$ the sets of the minimal primes of $I$ and $J$ respectively. Then $$J : I = \bigcap_{\mathfrak{p} \in \cM(J)\setminus \cM(I) }\mathfrak{p}.$$
     \label{Prop: J:I = intersection minimal primes}
\end{prop}

\noindent The strategy of the proof consists in showing that the ideal $J(\cP)$, introduced in Definition \ref{Defn: Ideals K and J for the canonical module}, satisfies the assumptions of Proposition~\ref{prop5.5} and Proposition~\ref{prop5.6}, and then, by studying the minimal primes of $J(\cP)$, we prove that $J(\cP): I_{\cP}$ is $J(\cP)+K(\cP)$.

\begin{lemma}\label{Lemma: property of J(P) compared to I_P}
    Let $\cP$ be a closed path. Then $J(\cP)$ is contained in $I_{\cP}$ and is a complete intersection and radical ideal with $\mathrm{ht}(J(\cP))=\mathrm{ht}(I_{\cP})=\vert\cP\vert.$ 
\end{lemma}

\begin{proof}
    By Proposition~\ref{Lemma: A closed path with a L-conf is Konig type} we know that $I_{\cP}$ is of K\H{o}nig type, so there exist a monomial order $<$ on $S_{\cP}$ and $h=\vert \cP\vert$ generators $f_1,\dots,f_h$ of $I_{\cP}$ such that $\lt_{<}(f_1),\dots,\lt_{<}(f_h)$ forms a regular sequence. By Definition~\ref{Defn: Ideals K and J for the canonical module}, $J(\cP)$ is generated by $f_1,\dots,f_h$. We have obviously $J(\cP)\subset I_{\cP}$. Since $\lt_{<}(f_1),\dots,\lt_{<}(f_h)$ forms a regular sequence, then it is easy to see that $f_1,\dots,f_h$ is a regular sequence of $J(\cP)$, so $J(\cP)$ is a complete intersection. Now, since $\lt_{<}(f_1),\dots,\lt_{<}(f_h)$ is a regular sequence, we have that $\mathrm{ht}(J(\cP))=h$, which is equal to $\mathrm{ht}(I_\cP)$ by Corollary~\ref{Coro: height of P}, and moreover $f_1,\dots,f_h$ is a Gr\"obner basis with respect to $<$ of $J(\cP)$. Therefore, the initial ideal of $J(\cP)$ with respect to $<$ is squarefree and, thus, $J(\cP)$ is radical. 
\end{proof}

\noindent In order to study the minimal prime ideals of $J(\cP)$, we introduce the following notation and, in particular, a suitable definition of \textit{admissible set} of $V(\cP)$, inspired by \cite{Cisto_Navarra_Veer} and \cite{Herzog_Hibi_ad}. Let $f=x_ax_b-x_cx_d$ be a generator of $J(\cP)$, attached to the inner interval $[a,b]$ of $\cP$, with $c,d$ as anti-diagonal corners. We define $V(f)=\{a,b,c,d\}$ and $E(f)=\big\{ \{a,c\},\{a,d\},\{b,c\},\{b,d\} \big\}$ as the sets of vertices and edges of $f$, respectively.

\begin{defn}\label{Defn: aadmissible set}\rm\label{admissible}
Let $\cP$ be a circle closed path and $\cA\subset V(\cP)$. We say that $\cA$ is an \textit{admissible set} for $J(\cP)$ if:
\begin{enumerate}
    \item For each generator $f$ of $J(\cP)$, one of the following two conditions is satisfied:
\begin{enumerate}
    \item $V(f)\cap \cA=\emptyset$;
    \item $V(f)\cap \cA$ contains at least an edge of $f$ and $V(f)\cap \cA\neq V(f)$.
\end{enumerate}  
    \item Denote by $F_{\cA}$ the set of generators $f$ of $J(\cP)$ such that $V(f)\cap \cA=\emptyset$. Then $\vert F_{\cA}\vert +\vert \cA\vert=\vert \cP\vert.$ 
\end{enumerate}
\end{defn}

\begin{exa}\rm \label{Exa: admissible sets}
    \noindent In Figure~\ref{Figure: example minimial primes}: $\{(3,4),(3,5)\}$, $\{(4,1),(4,2),(1,3),(2,3)\}$, $\{(5,4),(6,4),(6,5)\}$, $\{(2,1),(2,2),(2,3),(2,4),(2,5)\}$ and $\{(6,2),(6,3),(6,4),(6,5),(5,2),(5,1),(3,2),(3,1)\}$ are examples of admissible sets for $J(\cP)$. The following ones are not admissible sets for $J(\cP)$:
    \begin{itemize}
        \item $\cA_1=\{(4,4),(6,4),(6,5)\}$, since $V(x_{(5,3)}x_{(6,4)}-x_{(5,4)}x_{(6,3)})\cap\cA_1=\{a\}$;
        \item $\cA_2=\{(4,4),(6,4),(6,5),(4,5),(3,1),(3,2)\}$, since $V(x_{(4,4)}x_{(6,5)}-x_{(4,5)}x_{(6,4)})\cap\cA_2=\{(4,4),(6,5),(4,5),(6,4)\}=V(x_{(4,4)}x_{(6,5)}-x_{(4,5)}x_{(6,4)})$.
        \item $\cA_3=\{(5,1),(5,2),(5,3),(5,4),(5,5),(6,2)\}$, since $\cA_3$ contains $\{(5,1),(5,2),(5,3),(5,4),(5,5)\}$, which is an admissible set for $J(\cP)$, so condition (2) cannot be satisfied.
        \item $\cA_4=\{(6,1),(6,2),(6,3),(6,4),(6,5),(5,2),(4,2),(3,2),(3,1)\}$, since $\vert F_{\cA_4}\vert + \vert\cA_4\vert=15\neq 14=\vert \cP\vert.$
    \end{itemize}
    
\end{exa}

\noindent Following \cite{Cisto_Navarra_Veer}, we recall the definition of a polyocollection. Firstly, if $I=[a,b]$ is an interval of $\Z^2$ with anti-diagonal corners $c,d$, we put $V(I):=\{a,b,c,d\}$ and $E(I):=\big\{[a,c],[a,d],[c,b],[d,b]\big\}.$ Let $\cC$ be a collection of intervals in $\Z^2$. We say that $\cC$ is a \textit{polyocollection} if for all $I, J \in \cC$, with $I \neq J$, we have that $I$ is not contained in $J$ and one of the following holds:
\begin{enumerate}
    \item $I \cap J$ is a common edge of $I$ and $J$;
    \item For all $F \in E(I)$ and for all $G \in E(J)$, $\vert F \cap G\vert \leq 1$.
\end{enumerate}
\noindent Observe that every collection of cells is a polyocollection but the converse is not true. Moreover, if $\cC$ is a polyocollection, then we can define a polynomial ring $S_{\cC}$ and a binomial ideal $I_\cC\subset S_{\cC}$ as done in \cite{Cisto_Navarra_Veer} generalizing in a natural way the definition of a polyomino ideal given in \cite{Qureshi}.

\begin{discussion}\rm\label{Disc: for an admissible set there is a polyocollection}
    Let $\cP$ be a circle closed path polyomino, $J(\cP)$ be the ideal as in Definition~\ref{Defn: Ideals K and J for the canonical module} and $\cA$ be an admissible set for $J(\cP)$. Remember that $F_{\cA}$ is the set of the generators $f$ of $J(\cP)$ such that $V(f)\cap \cA=\emptyset$ and $\vert F_{\cA}\vert +\vert \cA\vert=\vert \cP\vert$. Now, consider 
    $$\cC_{\cA}=\{I\ \text{is an interval of}\ \Z^2:V(I)=V(f)\ \text{for some}\ f\in F_{\cA}\}.$$ 
    Obviously $\cC_{\cA}$ is a polyocollection. We now discuss various aspects of $\cC_{\cA}$. Consider the configuration in Figure \ref{Figure: tetromino} up to rotations, which means that $(a,b)$ can be $\big((2,2),(2,1)\big)$, $\big((m-1,2),(m,2)\big)$, $\big((m-1,n-1),(m-1,n)\big)$ or $\big((2,n),(2,n-1)\big)$. Moreover, we point out that if the horizontal cell interval containing $[h,a]$ and $[c,b]$ has only three cells, then $h$ is $c-(2,0)$ in that case. However, without loss of generality, we may assume that $a=(m-1,2)$ and $b=(m,2)$, and we discuss the vertices of $\cA$ which are in $[c,(m,n)]\cup\{g,h\}$, since $\cP$ consists of this type of configuration. Consider the following two cases.  

    \begin{figure}[h]
		\centering
		\includegraphics[scale=0.6]{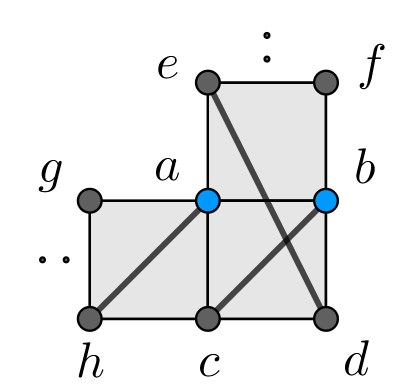}
		\caption{A trimino.}
		\label{Figure: tetromino}
	\end{figure}

\noindent Case 1) Let us suppose that $a,b\in \cA$. We discuss the following sub-cases:

\begin{itemize}
    \item If $c,d,e,f\notin \cA$, then $g$ is a vertex in $\cA$, otherwise we have a contradiction with (1)-(b) of Definition \ref{Defn: aadmissible set}. Hence $[h,a]$ and $[c,b]$ are not in $\cC_{\cA}$ instead $[c,f]\in \cC_{\cA}$. Therefore, $\cC_{\cA}$ is a polyocollection which is not a collection of cells.
    
    \item  We show firstly that $c$ cannot be in $\cA$. If $c\in \cA$, then $d\notin \cA$, otherwise we have a contradiction with (1)-(b) of Definition \ref{Defn: aadmissible set}, and, moreover, $e\in \cA$ necessarily. Since $e\in \cA$ and $f\notin \cA$, then $e+(0,1)\in \cA$. Continuing this argument for $e+(0,2)$, $e+(0,3)$ and so on, we can have two possibilities.
    \begin{itemize}
        \item Suppose $[e,(m-1,n)] \subset \cA$. We can assign the $2$-minor of $[h,a]$ to $c$,  that one of $[c,b]$ to $a$, that one of $[c,f]$ to $b$ and that one of every cell in $\cP_{[a,(m,n)]}$ to the lower left corner of the attached cell. Since $(m-1,n)$ is in $\cA$, then it is clear that $\vert \cA\vert+\vert F_{\cA}\vert \geq \vert \cP\vert +1$, which is a contradiction with (2) of Definition \ref{Defn: aadmissible set}. For instance, look at the case dealing for $\cA_4$ in Example \ref{Exa: admissible sets}.
        \item Suppose there exists a vertex $v\in [e,(m-1,n-1)]$ such that $[e,v]\cup [v+(1,0),(m,n)]$ is in $\cA$. Then we get the same previous contradiction by similar arguments.
    \end{itemize}
    Hence $c$ cannot be in $\cA$. Therefore, $g\in \cA$ otherwise we have a contradiction with (1)-(b) of Definition \ref{Defn: aadmissible set}. Hence $[h,a]$ and $[c,b]$ are not in $\cC_{\cA}$. 
    \begin{itemize}
        \item Now, if $e,f\notin \cA$, then $d\notin \cA$, so $[c,f]\in \cC_{\cA}$. Therefore, $\cC_{\cA}$ is a polyocollection which is not a collection of cells.
        \item If either $e,f\in \cA$ or $e\notin\cA$ and $f\in\cA$ (which implies $d\in \cA$), then $[c,f]\notin \cC_{\cA}$, and if this holds even in the other three corners of $\cP$, then $\cC_{\cA}$ is in particular a collection of cells.  
    \end{itemize}
\end{itemize} 
   \noindent Case 2) When at least one of $a$ and $b$ is not in $\cA$, then $\cP$ is a collection of cells, since $[c,b]\notin \cC_{\cA}$ and $[c,f] \in \cC_{\cA}$ cannot be possible.

Now, we discuss some algebraic properties of $I_{\cC}$, which will be useful in what follows. In particular, we show that $I_{\cC_{\cA}}$ is a prime ideal with $\mathrm{ht}(I_{\cC_{\cA}})=\vert F_{\cA}\vert.$ We need to distinguish two cases.\\
1) If $\cC_{\cA}$ is a collection of cells, in particular from the definition of $\cC_{\cA}$ we have that $\cC_{\cA}$ is either a simple polyomino or a disjoint union of simple polyominoes, so $I_{\cC_{\cA}}$ is a prime ideal from \cite[Corollary 2.2]{Simple equivalent balanced} or \cite[Corollary 2.3]{Simple are prime}. Moreover, from \cite[Corollary 2.3]{def balanced} and \cite[Corollary 2.2]{Simple equivalent balanced} it follows that $\mathrm{ht}(I_{\cC_{\cA}})=\vert \cC_{\cA}\vert$. Since there is a natural one-to-one correspondence between the cells of $\cC_{\cA}$ and the generators of $F_{\cA}$, we have $\vert \cC_{\cA}\vert=\vert F_{\cA}\vert$, so $\mathrm{ht}(I_{\cC_{\cA}})=\vert F_{\cA}\vert.$ \\
2) If $\cC$ is a polyocollection but not a collection of cells, that is in the cases $(1)-(a)$ and $(1)-(b)-(i)$, then $I_{\cC_{\cA}}$ can be identified with the inner $2$-minors ideal attached to the collection of cells $\cP'= \big(\cC_{\cA}\setminus\{[c,f]\}\big) \cup \{[a,f]\}$. We are replacing just an interval with a cell in $\cC_{\cA}$, so $\cP'$ is a simple polyomino or a disjoint union of simple polyominoes, so $I_{\cC_{\cA}}$ is a prime ideal $\mathrm{ht}(I_{\cC_{\cA}})=\vert F_{\cA}\vert$ by the same arguments done before.  
\end{discussion}

\noindent From Discussion~\ref{Disc: for an admissible set there is a polyocollection} we get the following result.

\begin{lemma}\label{Lemma: from an admissible set we have a polyocollection}
    Let $\cP$ be a circle closed path polyomino, $J(\cP)$ be the ideal as in Definition \ref{Defn: Ideals K and J for the canonical module} and $\cA$ be an admissible set for $J(\cP)$. Let $F_{\cA}$ be the set of the generators $f$ of $J(\cP)$ such that $V(f)\cap \cA=\emptyset$. Then
    $$\cC_{\cA}=\{I\ \text{is an interval of}\ \Z^2:V(I)=V(f)\ \text{for some}\ f\in F_{\cA}\}$$
    is a polyocollection $\cC_{\cA}$ such that  $I_{\cC_{\cA}}$ is a prime ideal and $\mathrm{ht}(I_{\cC_{\cA}})=\vert F_{\cA}\vert.$
\end{lemma}

\noindent Now, if $\cA\subset V(\cP)$ is an admissible set for $J(\cP)$, then we set $X(\cA)=\{x_a:a\in\cA\}$.

\begin{lemma}\label{Lemma: from admissible set to a minimal prime}
    Let $\cP$ be a circle closed path polyomino and $J(\cP)$ be the ideal as in Definition \ref{Defn: Ideals K and J for the canonical module}. Let $\cA$ be an admissible set for $J(\cP)$ and $\cC_{\cA}$ be the polyocollection defined in Lemma \ref{Lemma: from an admissible set we have a polyocollection}. Consider the ideal $\mathfrak{p}_{\cA}=I_{\cC_\cA}+(X(\cA))$ in $S_\cP$. Then $\mathfrak{p}_{\cA}$ is a minimal prime ideal of $J(\cP)$.
\end{lemma} 

\begin{proof}
    We start by showing that $J(\cP)$ is contained in $\mathfrak{p}_{\cA}$. Let $f$ be a generator of $J(\cP).$ If $f$ is in $F_{\cA}$, then $f$ is a generator of $I_{\cC_{\cA}}$. By contradiction, if we assume that $f\notin F_{\cA}$, then $V(f)\cap \cA$ contains at least an edge of $f$, which means that a diagonal corner and an anti-diagonal one of the interval given by $V(f)$ are in $\cA$, so $f$ belongs to $(X(\cA)) $. Hence, $J(\cP)$ is contained in $\mathfrak{p}_{\cA}$. \\
    Let us consider the ideal $(X(\cA))$. It is generated by $\vert \cA\vert$ variables, so it is a prime ideal with $\mathrm{ht}(X(\cA))=\vert\cA\vert.$  By Lemma~\ref{Lemma: from an admissible set we have a polyocollection}, we know that $\cC_{\cA}$ is a polyocollection such that $I_{\cC_{\cA}}$ is a prime ideal and $\mathrm{ht}(I_{\cC_{\cA}})=\vert F_{\cA}\vert.$ Now, we observe that for any generator $f$ of $I_{\cC_{\cA}}$ and for any $a\in\cA$, we have $\mathrm{supp}(f)\cap \mathrm{supp}(x_a)=\emptyset$. Hence $\mathfrak{p}_{\cA}$ is a prime ideal with $\mathrm{ht}(\mathfrak{p}_{\cA})=\vert F_{\cA}\vert+\vert \cA\vert,$ which is $\vert\cP\vert$ from Definition \ref{Defn: aadmissible set}, that is, the height of $J(\cP)$ by Lemma~\ref{Lemma: property of J(P) compared to I_P}. Then, $\mathfrak{p}_{\cA}$ is a minimal prime of $J(\cP).$
\end{proof}

\begin{lemma}\label{Lemma: from a minimal prime to admissible set}
    Let $\cP$ be a circle closed path polyomino and $J(\cP)$ be the ideal as in Definition \ref{Defn: Ideals K and J for the canonical module}. Let $\mathfrak{p}$ be a minimal prime of $J(\cP)$. Then there exists an admissible set $\cA$ for $J(\cP)$ such that $\mathfrak{p}=\mathfrak{p}_{\cA}$.
\end{lemma} 

\begin{proof}
     Let $\cA=\{a\in V(\cP):x_a\in \mathfrak{p}\}$. We want to prove that $\cA$ is an admissible set of $J(\cP)$ such that $\mathfrak{p}=\mathfrak{p}_{\cA}$. If $\cA=\emptyset$, then $\cA$ is an admissible set for $J(\cP)$. Moreover, in order to prove that $\mathfrak{p}=\mathfrak{p}_\emptyset:=I_{\cP}$, we first show that $I_{\cP}\subset \mathfrak{p}$. By being the structure of $\cP$ and the definition of $J(\cP)$ and $I_{\cP}$, it is enough to prove that $x_cx_f-x_ex_b$ and $x_ax_f-x_dx_e$, as in Figure \ref{Figure: tetromino for primality - V(f) cap A} (A) and (B) respectively, belong to $\mathfrak{p}$. 

      \begin{figure}[h]
		\centering
		\subfloat[]{\includegraphics[scale=0.6]{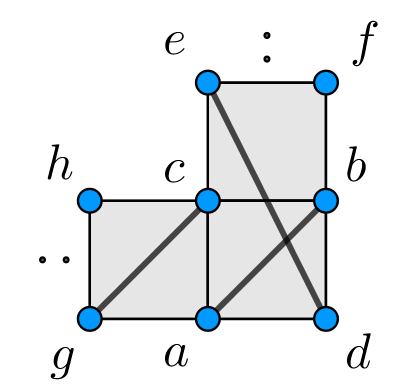}}\qquad\qquad\qquad
		\subfloat[]{\includegraphics[scale=0.6]{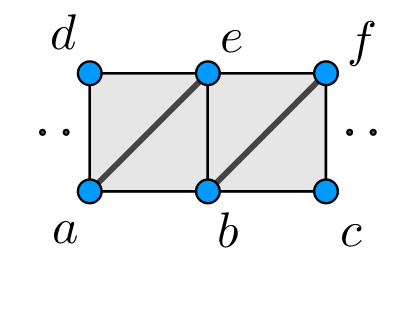}}
		\caption{Two configurations of cells in a circle closed path.}
		\label{Figure: tetromino for primality - V(f) cap A}
	\end{figure}

\noindent  Observe that $x_a(x_cx_f-x_ex_b)\in \mathfrak{p}$ because $x_a(x_cx_f-x_ex_b)=x_c(x_ax_f-x_ex_b)-x_e(x_ax_d-x_cx_b)$ and $x_ax_f-x_ex_b,x_ax_d-x_cx_b\in J(\cP)\subset  \mathfrak{p}$. So, $x_cx_f-x_ex_b\in \mathfrak{p}$ because $\mathfrak{p}$ is prime and $x_a\notin \mathfrak{p}$. Similarly, $x_e(x_ax_f-x_dx_c)=x_f(x_ax_e-x_dx_b)+x_d(x_bx_f-x_ex_c)\in \mathfrak{p}$, so $x_ax_f-x_dx_c\in \mathfrak{p}.$ Therefore $I_{\cP}\subset \mathfrak{p}$. Observe that $I_{\cP}$ is a prime ideal by \cite[Corollary 4.3]{Cisto_Navarra_closed_path} and $J(\cP)\subset I_{\cP}\subset \mathfrak{p}$ so $I_{\cP}= \mathfrak{p}$ from the minimality of $\mathfrak{p}$.\\
    Now, assume that $\cA\neq \emptyset$. Let us show that $\cA$ is an admissible set for $J(\cP).$ Suppose by contradiction that $\cA$ is not an admissible set for $J(\cP)$, so at least one of the conditions in Definition \ref{Defn: aadmissible set} does not hold. We need to analyze some cases.
     \begin{itemize}
         \item[Case 1)] Suppose that there exists a generator $f=x_px_q-x_rx_s$ of $J(\cP)$, where $p,q$ and $r,s$ are diagonal and anti-diagonal corners of $f$ respectively, such that $V(f)\cap \cA$ does not contains an edge of $f$ or $V(f)\cap \cA=V(f).$ 
         \begin{enumerate}
             \item Assume that $V(f)\cap \cA$ does not contains an edge of $f$, so either $\vert V(f)\cap \cA\vert =1$ or $V(f)\cap \cA$ contains $p,q$ (or $r,s$). If $\vert V(f)\cap \cA\vert =1$, then we may assume that $V(f)\cap \cA=\{p\}$, so $x_rx_s\in \mathfrak{p}$ since $f,x_p\in\mathfrak{p}$. Hence $x_r\in \mathfrak{p}$ or $x_s\in \mathfrak{p}$ because of the primality of $\mathfrak{p}$, so $V(f)\cap \cA$ contains an edge of $f$, which is a contradiction. The same argument holds if $V(f)\cap \cA$ contains $p,q$ (or $r,s$).
             \item Assume that $V(f)\cap \cA=V(f).$ With reference to Figure~\ref{Figure: tetromino for primality - V(f) cap A} (A), set $p=a$ and $q=b$, so $r=c$ and $s=d$.  Let $F_{\cA}$ be the set of the generators $f$ of $J(\cP)$ such that $V(f)\cap \cA=\emptyset$ and $\cC_{\cA}$ be the polyocollection defined by the intervals $I$ of $\Z^2$ such that $V(I)=V(f)$, for some $f\in F_{\cA}$. Denote $X_{b}=\{x_v:v\in \cA\setminus\{b\}\}$. Consider $\mathfrak{p}':=I_{\cC_{\cA}}+(X_{b})$. It is easy to see, as in Lemma \ref{Lemma: from admissible set to a minimal prime}, that $J(\cP)\subset \mathfrak{p}'$ and $\mathfrak{p}'$ is prime. Moreover, as done before, it can be shown that $I_{\cC_{\cA}}\subset \mathfrak{p}$, so $\mathfrak{p}'\subset \mathfrak{p}$. From the minimality of $\mathfrak{p}$, we have $\mathfrak{p}'= \mathfrak{p}$, but this is a contradiction since $x_{b}\in \mathfrak{p}$ but $x_b\notin \mathfrak{p}'$. Similar arguments can be made for the case $p=a$ and $q=f$, referring to Figure \ref{Figure: tetromino for primality - V(f) cap A} (A). Now, we refer to Figure \ref{Figure: tetromino for primality - V(f) cap A} (B) and assume that $p=a$ and $q=e$, so $r=d$ and $s=b$. In such a case, consider $F_{\cA\setminus\{e,b\}}$ as the set of the generators $f$ of $J(\cP)$ such that $V(f)\cap (\cA\setminus \{e,b\})=\emptyset$, $\cC_{\cA}$ as the polyocollection defined by the intervals $I$ of $\Z^2$ such that $V(I)=V(f)$, for some $f\in F_{\cA\setminus\{e,b\}}$, $X_{e,b}:=\{x_v:v\in \cA\setminus\{e,b\}\}$ and $\mathfrak{p}':=I_{\cC_{\cA}}+(X_{e,b})$. Arguing as done before we get $\mathfrak{p}=\mathfrak{p}'$, so a contradiction arises also in this case.
         \end{enumerate}

         \item[Case 2)] Here we discuss the case when $\vert F_{\cA}\vert +\vert\cA\vert \neq \vert \cP\vert$, where $F_{\cA}$ is the set of the generators $f$ of $J(\cP)$ such that $V(f)\cap \cA=\emptyset$. First of all, we observe that $b$ in Figure \ref{Figure: tetromino for primality - V(f) cap A} (A) is the diagonal corner of the only binomial $x_ax_b-x_cx_d$ and $d$ is the anti-diagonal corner of the binomials $x_ax_b-x_cx_d$ and $x_ax_f-x_ex_d$; the other vertices in Figures \ref{Figure: tetromino for primality - V(f) cap A} (A), except $b$ and $d$, and (B) at the same time diagonal and anti-diagonal corner of two different generators of $J(\cP)$, for instance, $a$ is diagonal corner of $x_ax_b-x_cx_d$ and anti-diagonal of $x_gx_c-x_hx_a$. \\
         Now, suppose that $\vert F_{\cA}\vert +\vert\cA\vert \neq \vert \cP\vert.$  Denote by $\cG(J(\cP))$ the set of generators of $J(\cP)$ and recall that $\vert J(\cP)\vert =\vert \cP\vert$. This means that we are assuming that $\vert\cA\vert \neq \vert \cG(J(\cP))\setminus F_{\cA}\vert$. Consider $\vert\cA\vert < \vert \cG(J(\cP))\setminus F_{\cA}\vert$.
         Hence, $\vert\cA\vert < \vert \cG(J(\cP))\setminus F_{\cA}\vert$ implies that there exist two binomials $F$ and $G$ in $\cG(J(\cP))\setminus F_{\cA}$ such that $V(F)\cap \cA$ and $V(G)\cap \cA$ consist of the same vertex, which we may assume that is a diagonal corner, and the anti-diagonal corners of $F$ and $G$ are not in $\cA$. With reference to Figure \ref{Figure: tetromino for primality - V(f) cap A} (B), for instance, it means that if $F=x_ax_e-x_dx_b$, $G=x_bx_f-x_ex_c$ and $b\in \cA$, then $a,e,c\notin \cA$. But this leads a contradiction since $F, x_b\in \mathfrak{p}$ so $x_a\in \mathfrak{p}$ or $x_e\in \mathfrak{p}$ for the primality of $\mathfrak{p}$. All the other possible situations of $F$ and $G$ coming from Figure \ref{Figure: tetromino for primality - V(f) cap A} (A) give us a similar contradiction.\\ 
         Now, assume that $\vert\cA\vert > \vert \cG(J(\cP))\setminus F_{\cA}\vert$. Recall also that $\vert \cP\vert=\frac{\vert V(\cP)\vert}{2}$. Since $\vert\cA\vert > \vert \cG(J(\cP))\setminus F_{\cA}\vert$ then every segment in Figures \ref{Figure: tetromino for primality - V(f) cap A} (A) and (B) has a point in $\cA$, and there exists a segment having two vertices in $\cA$. This leads to a contradiction. In fact, with reference to Figure \ref{Figure: tetromino for primality - V(f) cap A} (A), we suppose that $b,c,d,e \in \cA$. Observe that $a\notin \cA$, otherwise we have $V(x_ax_b-x_cx_d)\cap \cA=V(x_ax_b-x_cx_d)$ and we have a contradiction as in Case 1)-(2). If $f\notin \cA$, then the contradiction arises as $x_a\in \mathfrak{p}$ or $x_f\in\mathfrak{p}.$ So, $f\in \cA$. But, in such a case, consider $F_{\cA\setminus\{d\}}$ as the set of the generators $f$ of $J(\cP)$ such that $V(f)\cap (\cA\setminus \{d\})=\emptyset$, $\cC_{\cA}$ as the polyocollection defined by the intervals $I$ of $\Z^2$ such that $V(I)=V(f)$, for some $f\in F_{\cA\setminus\{d\}}$, $X_{d}:=\{x_v:v\in \cA\setminus\{d\}\}$ and $\mathfrak{p}':=I_{\cC_{\cA}}+(X_{d})$. Arguing as in Case 1)-(2) we get $\mathfrak{p}=\mathfrak{p}'$, which leads to a contradiction. All other cases, coming by taking other different points, give us a contradiction as before.  Moreover, referring to Figure \ref{Figure: tetromino for primality - V(f) cap A} (B), we may suppose that $b,d,e,f \in \cA$. This case leads to a contradiction arguing as done before, considering $F_{\cA\setminus\{b\}}$ as the set of the generators $f$ of $J(\cP)$ such that $V(f)\cap (\cA\setminus \{b\})=\emptyset$, $\cC_{\cA}$ as the polyocollection defined by the intervals $I$ of $\Z^2$ such that $V(I)=V(f)$, for some $f\in F_{\cA\setminus\{b\}}$, $X_{b}:=\{x_v:v\in \cA\setminus\{b\}\}$ and $\mathfrak{p}':=I_{\cC_{\cA}}+(X_{b})$. \\
         We have a contradiction when $\vert F_{\cA}\vert +\vert\cA\vert\neq \vert \cP\vert$, so we get the claim: $\vert F_{\cA}\vert +\vert\cA\vert=\vert \cP\vert$.
     \end{itemize}
        Therefore, we have proved that $\cA$ is an admissible set of $J(\cP).$ Now, we need to show only that $\mathfrak{p}=\mathfrak{p}_{\cA}$. 
        Consider the prime ideal $\mathfrak{p}_{\cA}=I_{\cC_{\cA}}+(X(\cA))$. As done in the first part of the proof, where we proved that $\mathfrak{p}=I_{\cP}$, we have that $\mathfrak{p}_\cA\subset \mathfrak{p}$. Then, by the minimality of $\mathfrak{p}$, we have $\mathfrak{p}=\mathfrak{p}_\cA$.
\end{proof}

\begin{thm}\label{Thm: caract. minimal prime J(P)}
   Let $\cP$ be a circle closed path polyomino and $J(\cP)$ be the ideal as in Definition \ref{Defn: Ideals K and J for the canonical module}. An ideal $\mathfrak{p}$ is a minimal prime of $J(\cP)$ if and only if $\mathfrak{p}=I_{\cC_\cA}+(X(\cA))$, where $\cA$ is an admissible set for $J(\cP)$ and $\cC_{\cA}$ is the polyocollection defined in Lemma \ref{Lemma: from an admissible set we have a polyocollection}. In particular, $J(\cP)$ is an unmixed ideal. 
\end{thm}

\begin{proof}
    It follows from Lemma \ref{Lemma: from admissible set to a minimal prime} and \ref{Lemma: from a minimal prime to admissible set}. Moreover, all the minimal primes of $J(\cP)$ have the same height, hence $J(\cP)$ is unmixed. 
\end{proof}

 \begin{figure}[h]
		\centering
		\subfloat[]{\includegraphics[scale=0.7]{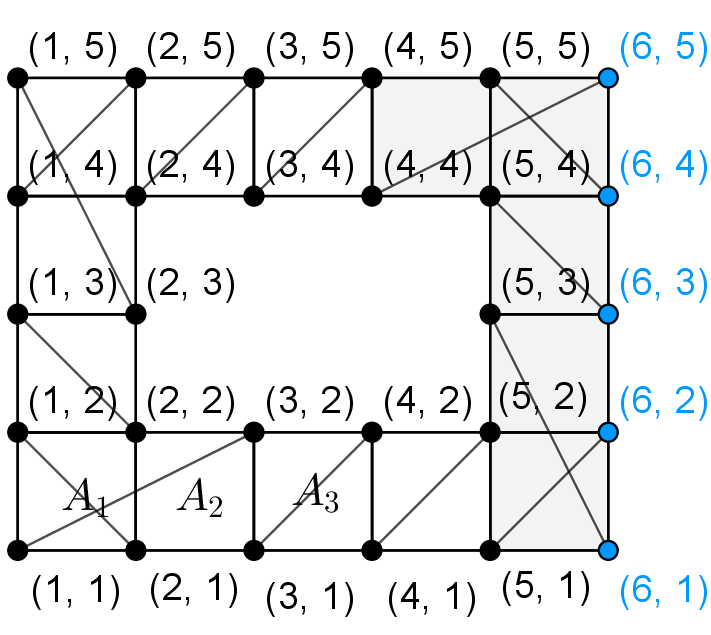}}\quad
		\subfloat[]{\includegraphics[scale=0.7]{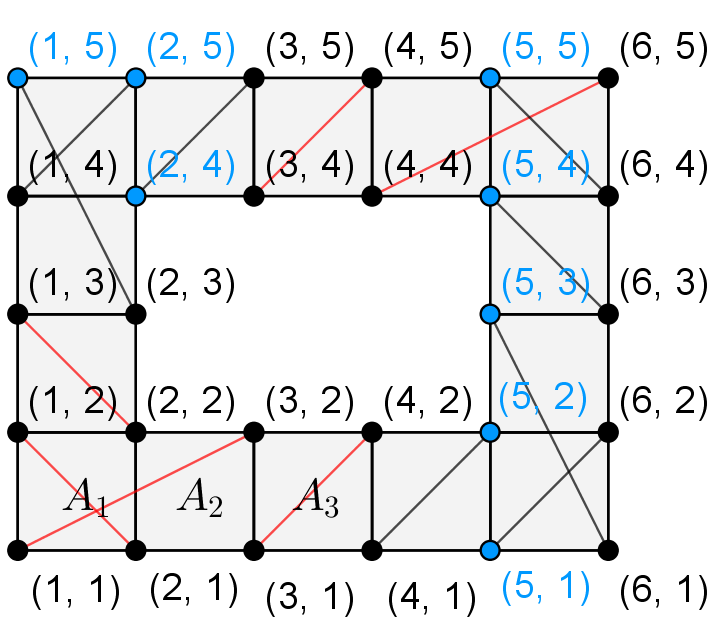}}\quad
		\subfloat[]{\includegraphics[scale=0.7]{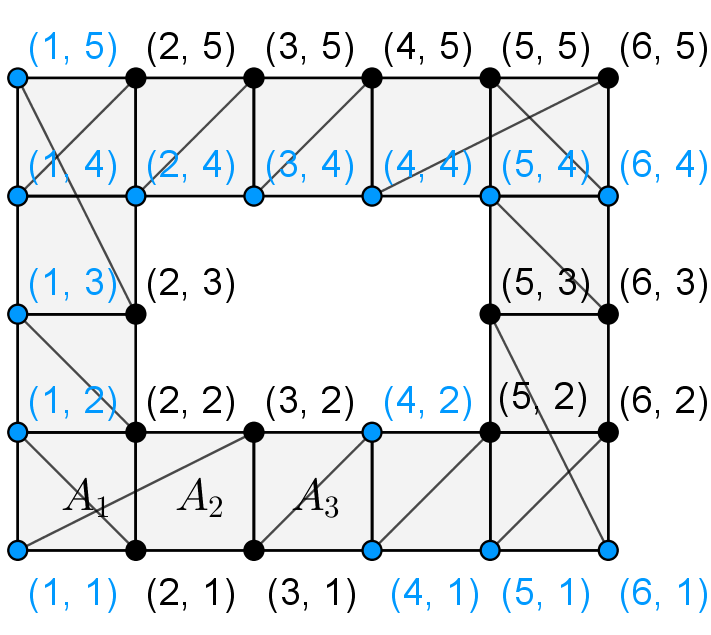}}
		\caption{Representations of three minimal primes in a circle closed path.}
		\label{Figure: example minimial primes}
	\end{figure}

\begin{exa}\rm\label{exa: minimalprimes}
Consider the circle closed path $\cP$ as in Figure \ref{Figure: example minimial primes}. Using \texttt{Macaulay2} and the package \texttt{Binomials} (see \cite{Kahle}), we compute all the minimal primes of $J(\cP)$, which are $1448$. Here, we figure out just some of them.

\begin{enumerate}
    \item $\mathfrak{p}_1=I_{\cC_1}+M_1$, where $M_1=(x_{(6,5)},\,x_{(6,4)},\,x_{(6,3)},\,x_{(6,2)},\,x_{(6,1)})$ and $I_{\cC_1}$ is the ideal attached to the collection of white cells in Figure \ref{Figure: example minimial primes} (A);
    \item $\mathfrak{p}_2=I_{\cC_2}+M_2$, where $M_2=(x_{(1,5)},\,x_{(2,5)},\,x_{(5,5)},\,x_{(2,4)},\,x_{(5,4)},x_{(5,3)},\,x_{(5,2)},\,x_{(5,1)})$ and $I_{\cC_2}$ is the ideal attached to the polyocollection $$\cC_2=\big\{[(1,1),(2,2)],[(1,2),(2,3)],[(2,1),(3,2)],[(3,1),(4,2)],[(3,4),(4,5)],[(4,4),(6,5)]\big\}.$$ Look at Figure \ref{Figure: example minimial primes} (B);
    \item $\mathfrak{p}_3=(x_{(1,5)},\,x_{(1,4)},\,x_{(1,3)},\,x_{(1,2)},\,x_{(1,1)},x_{(2,4)},\,x_{(3,4)},\,x_{(4,4)},\,x_{(4,2)},\,x_{(4,1)},\,x_{(5,4)},\,x_{(5,1)},\,x_{(6,4)},\,x_{(6,1)})$. Look at Figure \ref{Figure: example minimial primes} (C).
\end{enumerate}
\end{exa}

\noindent Now, we are ready to prove Theorem \ref{Thm: description canonical module}.

\begin{proof}[Proof of Theorem \ref{Thm: description canonical module}]
 Let $\cP$ be a circle closed path polyomino and $J(\cP)$ and $K(\cP)$ be the ideals as in Definition \ref{Defn: Ideals K and J for the canonical module}. We want to show that $J(\cP) : I_{\cP} = J(\cP)+K(\cP)$. We denote
 $$T(\cP):=\bigcap_{\cA \text{ admissible set} \atop \text{ of } J(\cP) \text{ with } \cA\neq \emptyset}\Big(I_{\cC_{\cA}}+(X(\cA))\Big).$$
  By Proposition \ref{Prop: J:I = intersection minimal primes}, Lemma \ref{Lemma: property of J(P) compared to I_P} and Theorem \ref{Thm: caract. minimal prime J(P)} we have $J(\cP):I_{\cP}=T(\cP)$. We prove that $J(\cP)+K(\cP)=T(\cP)$.\\
 $\subseteq)$ Let $f$ be a generator of $J(\cP)+K(\cP)$. Firstly, consider that $f$ is a generator of $J(\cP)$, as $x_ax_b-x_cx_d$, and we show that $f\in I_{\cC_{\cA}}+(X(\cA))$ for all admissible set $\cA\neq \emptyset$ of $J(\cP)$. Let $\cA$ be an admissible set of $J(\cP)$, different from the empty set. If $V(f)\cap \cA=\emptyset$, then $f$ is a generator of $I_{\cC_{\cA}}$. Assume that $V(f)\cap \cA=\emptyset$ contains an edge of $f$, so we may consider that $\{a,c\}\in V(f)\cap \cA$. This implies that $f\in (X(\cA))$, so $f\in I_{\cC_{\cA}}+(X(\cA))$. Now, suppose that $f$ is a monomial generator of $K(\cP)$. Using notations from Definition~\ref{Defn: Ideals K and J for the canonical module}, we have that $f=x_{v_1}x_{v_2}\dots x_{v_t}x_{(2,2)}x_{(1,n)}x_{(m,1)}x_{(m-1,n-1)}$, for some $(v_1, v_2, \dots, v_t)\in \cV$. It is enough to show that for any admissible set $\cA\neq \emptyset$, there exists a vertex $w\in \{v_i:i\in[t]\}\cup\{(2,2), (1,n), (m,1), (m-1,n-1)\}$ such that $w\in \cA$, so that $f\in (X(\cA))$. Let us suppose by contradiction that there exists an admissible set $\cA\neq \emptyset$ such that every vertex in $\{v_i:i\in[t]\}\cup\{(2,2), (1,n), (m,1), (m-1,n-1)\}$ does not belong to $\cA$. Without loss of generality, we may consider only the part of $\cP$ as in Figure \ref{Figure: piece of circle closed path}, where $a=(m-1,2)$ and $b=(m,2)$, since $\cP$ consists only suitable rotations of this type of configuration. With reference to Figure \ref{Figure: piece of circle closed path} we have that $d,a'\notin \cA$. Consider $a_1$ and $b_1$. Using our assumption, we have $a_1 \notin\cA$ or $b_1\notin \cA$. If $a_1 \notin\cA$ and $b_1\in \cA$, then we obtain a contradiction, because $a_1,d\notin \cA$, hence $\cA$ cannot be an admissible set for $J(\cP)$, because the condition (1) of Definition \ref{Defn: aadmissible set} is not satisfied. If $a_1 \in\cA$ and $b_1\notin \cA$, then $a_2\in \cA$, because $\cA$ is an admissible set for $J(\cP)$, so $b_2$ cannot be in $\cA$ from our assumption. Continuing this arguments until $a_{n-4}$, we get that $a_{n-4}\in\cA$ and $b_{n-4}\notin\cA$. Since $\cA$ is an admissible set for $J(\cP)$ and $b_{n-4}\notin\cA$, then $a'\in \cA$ but this is a contradiction since $a'\notin\cA$. Hence $a_1,b_1$ cannot be in $\cA$. These arguments can be repeated for any $a_i$ and $b_i$, where $i=2,\dots,n-4$, getting that $a_i,b_i\notin \cA$ for all $i\in [n-4]$. Therefore, the only vertices in $\cA$ can be $a,c,b,c',b',d'$ and the analogous six vertices in the other two changes of direction. But it is impossible to make an admissible set with just these vertices, so $\cA$ cannot be an admissible set for $J(\cP)$, which is a contradiction. Hence, $f\in (X(\cA))$ and, in conclusion, $J(\cP)+K(\cP)\subseteq T(\cP)$.
        
        \begin{figure}[H]
		\centering
		\includegraphics[scale=0.6]{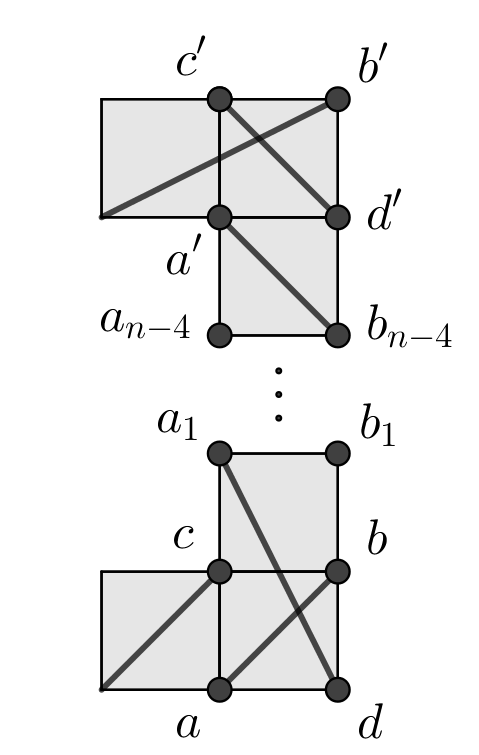}
		\caption{A part of a circle closed path.}
		\label{Figure: piece of circle closed path}
	\end{figure}
 
\noindent $\supseteq)$ Now, let $f\in T(\cP)$, so $f\in I_{\cC_{\cA}}+(X(\cA))$ for all admissible sets $\cA\neq \emptyset$. If $f\in J(\cP)$ then we have finished. Assume that $f\notin J(\cP)$ and we prove that $f\in K(\cP)$. Suppose that $f\notin K(\cP)$. Then it follows that there exists $m \in \mathrm{supp}(f)$ such that $m\notin K(\cP)$. We need to examine two cases.
\begin{enumerate}
    \item With reference to the notations in Definition \ref{Defn: aadmissible set}, suppose that there is a component $\cA_1=\{a,b\}$ of $\cV$ such that $x_{a}$ and $x_{b}$ do not divide $m$. It is not restrictive to assume that $\cA_1=\{a_1,b_1\}$, referring to Figure \ref{Figure: piece of circle closed path}, because the arguments are the same if $\cA_1$ is in $\{R_j:j=4,\dots,n-2\}$, $\{L_j:j=3,\dots,n-2\}$, $\{U_k:k=3,\dots,m-2\}$ or $\{D_k:j=3,\dots,m-2\}$. Observe that $\cA_1$ is an admissible set for $J(\cP)$ and $m\notin (X(\cA_1))$, so $f$ does not belong to $(X(\cA_1))$. Since $f\in T(\cP)$, then $f\in I_{\cC_{\cA_1}}$, where $\cC_{\cA_1}$ is the polyomino obtained by $\cP$ removing the cells having $\{a,b\}$ as common edge. Hence $f\in I_{\cP}$.
    \item The other case so that $m\notin K(\cP)$ is $x_d$ does not divide $m$ (or $a'$ does not divide $m$). Taking $\cA=\{a,c,d\}$ as an admissible set and using some arguments as done before, we get $f\in I_{\cP}$. 
\end{enumerate}
    Hence, we obtain that $f\in I_{\cP}$ in both cases. Since $f\in T(\cP)$, then it follows that $f \in T(\cP)\cap I_{\cP}$. Recall that $J(\cP)$ is a radical ideal from Lemma \ref{Lemma: property of J(P) compared to I_P} and, moreover, $I_\cP$ is the minimal prime of $J(\cP)$ coming from $\emptyset$ as an admissible set. By being $J(\cP)$ radical, we have that $J(\cP)$ is the intersection of all minimal prime ideals of $J(\cP)$, so $T(\cP)\cap I_{\cP}=J(\cP)$. Hence, we get that $f\in J(\cP)$, which is a contradiction. Therefore, $T(\cP)\subseteq J(\cP)+K(\cP)$. \\
    We have that $J(\cP):I_{\cP}=J(\cP)+K(\cP)$, and by Proposition~\ref{Prop: canonical module from linkage theory}, we obtain the desired conclusion, namely $$\omega_{K[\cP]}\cong \big(J(\cP)+K(\cP)\big)/J(\cP).$$
\end{proof}

\begin{exa}\rm\label{Exa: polyomino no good for c.m.}
Let $\cP$ be the closed path polyomino in Figure \ref{Figure: example no good polyomino for c.m.}. By using \texttt{Macaulay2}, we observe that $J(\cP):I_{\cP}=J(\cP)+K(\cP)$, where $K(\cP)$ is still a squarefree monomial ideal but it is not equigenerated, which means that it is not generated in a single degree. For instance, two generators of $K(\cP)$ in different degrees are $x_{(1,2)}x_{(1,3)}x_{(1,4)}x_{(2,6)}x_{(3,1)}x_{(3,6)}x_{(4,1)}x_{(4,6)}x_{(5,1)}x_{(6,3)}x_{(6,4)}x_{(6,5)}$ and $x_{(1,4)}x_{(2,2)}x_{(2,3)}x_{(3,2)}x_{(3,6)}x_{(4,3)}x_{(4,5)}x_{(5,1)}x_{(6,4)}$.

\begin{figure}[h!]
		\centering
		\includegraphics[scale=0.6]{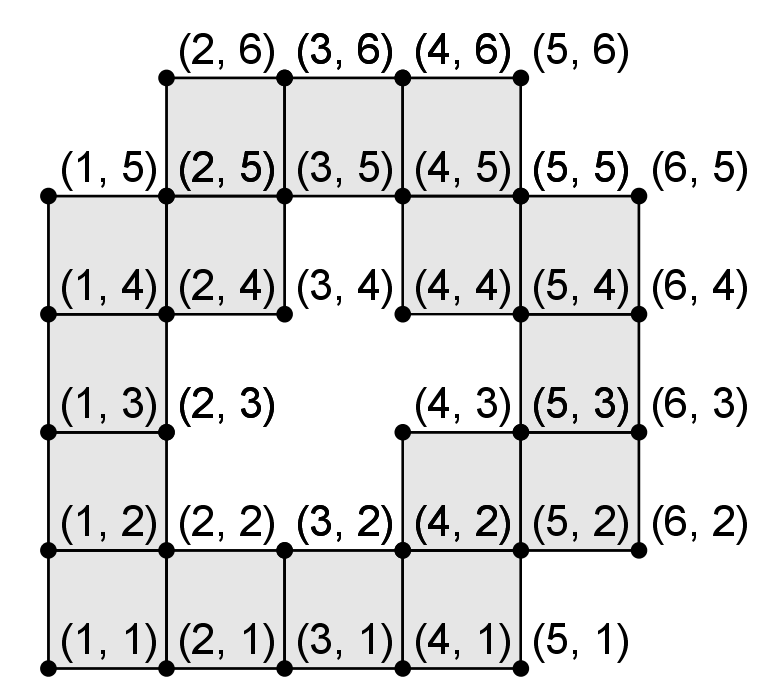}
		\caption{A non-level closed path.}
		\label{Figure: example no good polyomino for c.m.}
	\end{figure}

\end{exa}

\vspace{.1in}
\noindent \textbf{Question.} Is it possible to give a combinatorial interpretation of the canonical module of a closed path?\\

\noindent The description of the canonical module of $K[\cP]$ given in Theorem~\ref{Thm: description canonical module} allows us to compute the Cohen-Macaulay type of $K[\cP]$, i.e. the number of generators of the canonical module, where $\cP$ is a circle closed path polyomino. As a consequence, we show that $K[\cP]$ is a level ring, i.e. the generators of $\omega_{K[\cP]}$ are of the same degree.

\begin{coro}\rm\label{CMtype}
 Let $\cP$ be a circle closed path polyomino. Then $\mathrm{type}(K[\cP])=(m-3)^2(n-3)^2$ and $K[\cP]$ is a level ring.
\end{coro}

\begin{proof}
    If $m=n=4$, then $K[\cP]$ is Gorenstein, hence $\mathrm{type}(K[\cP])=1$. Assume that $m,n>4$. From \cite[Corollary 5.3.6]{Villareal}, the Cohen-Macaulay type of $K[\cP]$ is equal to the minimum number of the generators of the canonical module of $K[\cP]$. 
    The set $\cG=\{{\mathbf{x_v}}+J(\cP):\mathbf{v}\in \cV\}$ is a set of generators of $(J(\cP)+K(\cP))/J(\cP)$. We need to show that $\cG$ minimally generates $(J(\cP)+K(\cP))/J(\cP)$. Suppose by contradiction that $\cG$ is not a set of minimum number of generators of $(J(\cP)+K(\cP))/J(\cP)$, so there exists $\mathbf{w} \in\cV$ such that $\mathbf{x}_\mathbf{w}$ belongs to the ideal $J(\cP)+(\mathbf{x}_{\mathbf{v}}:\mathbf{v}\in \cV\setminus \{\mathbf{w}\})$. Set $J'=J(\cP)+(\mathbf{x}_{\mathbf{v}}:\mathbf{v}\in \cV\setminus \{\mathbf{w}\})$ and $\cG(J')=\{f:f \text{ generator of } J(\cP)\}\cup \{\mathbf{x}_{\mathbf{v}}:\mathbf{v}\in \cV\setminus \{\mathbf{w}\}\}$. Let $<$ be the lexicographic order on $S_{\cP}$ defined in Proposition \ref{Lemma: A closed path with a L-conf is Konig type}. 
    Denote by $S(f,g)$ the $S$-polynomial of two polynomials $f$ and $g$ of $S_{\cP}$. We note the following properties.
    
    \begin{enumerate}
        \item  For any two generators $f,g$ of $J(\cP)$, $S(f,g)$ reduces to $0$ since $\mathrm{gcd}(\mathrm{in}_<(f),\mathrm{in}_<(g))=1$.
        \item Trivially $S(u,u')=0$ for all monomial $u$ and $u'$ in $S_{\cP}$. 
        \item Consider a monomial $m$ in $S_{\cP}$ and suppose that there exists a generator $f$ of $J(\cP)$ such that $\mathrm{in}_<(f)$ divides $m$. Then $m=\mathrm{in}_<(f)u$, for a suitable monomial in $S_{\cP}$, and we may write $f$ as $\pm\mathrm{in}_<(f)\mp f'$. Hence by dividing $m$ with respect to $f$, we have $m=(\pm u)(\pm\mathrm{in}_<(f)\mp f')+uf'$. The latter means that $m$ reduces to $uf'$ with respect to $f$, where $f'$ is a monomial related to either diagonal or anti-diagonal corners of $V(f)$. 
        \item Consider a monomial $m$ in $S_{\cP}$ and assume that there exists a generator $f$ of $J(\cP)$ such that, if $\mathrm{in}_<(f)=x_ax_b$ ($a,b\in V(\cP)$), then $m=ux_a$ for a suitable monomial $u$ in $S_{\cP}$ and $x_b$ does not divide $u$. Then $S(m,f)=\pm uf'$, where $\deg(uf')>\deg(m)$, and it cannot be reduced to $0$ modulo $\cG(J')$ from the arguments used in $(3)$.
        \item If $m$ is a monomial in $S_{\cP}$ and suppose that there exists a generator $f$ of $J(\cP)$ such that $\mathrm{in}_<(f)$ divides $m$, then $S(m,f)=\pm uf'$, where $\deg(uf')=\deg(m)$, and it cannot be reduced to $0$ modulo $\cG(J')$ as well.
    \end{enumerate}

    \noindent Denote by $\mathfrak{G}$ the Gr\"obner basis of $J'$ with respect to $<$. From the properties described above, we can easily deduce that $\mathbf{x}_{\mathbf{w}}$ cannot be reduced to $0$ modulo $\mathfrak{G}$, which is a contradiction, since $\mathbf{x}_{\mathbf{v}}\in J'$. Therefore, $\cG$ is a minimal set of generators of $(J(\cP)+K(\cP))/J(\cP)$. By Remark~\ref{Rmk: number generators K(P)}, it follows that $\mathrm{type}(K[\cP])=(m-3)^2(n-3)^2$. Finally, since the monomials in $\cG$ have the same degree, $K[\cP]$ is a level ring. Similar arguments can be used in the cases $m=4,n>4$ and $m>4,n=4$. 
\end{proof}

\noindent \textbf{Question.} Is it possible to give an estimate, or an upper bound or a lower bound at least, for the Cohen-Macaulay type in terms of the combinatorial properties of a polyomino?\\

\noindent \textbf{Acknowledgement.} RD was supported by the Alexander von Humboldt Foundation and a grant of the Ministry of Research, Innovation and Digitization, CNCS - UEFISCDI, project number PN-III-P1-1.1-TE-2021-1633, within PNCDI III. FN was supported by Scientific and Technological Research Council of Turkey T\"UB\.{I}TAK under the Grant No: 122F128, and he is thankful to T\"UB\.{I}TAK for their supports. He is also enrolled in the group GNSAGA of INDAM and he acknowledges its support.

\end{document}